\documentclass[a4paper,12pt]{amsart}
\usepackage{amssymb,amsfonts,amsxtra,amscd,mathrsfs,dsfont}
\usepackage[all]{xy}
\usepackage{fullpage}
\usepackage{graphicx}
\usepackage{yfonts}
\theoremstyle{plain}
\newtheorem{theorem}{Theorem}[section]
\newtheorem{lemma}[theorem]{Lemma}
\newtheorem{cor}[theorem]{Corollary}
\newtheorem{prop}[theorem]{Proposition}
\theoremstyle{definition}
\newtheorem{defi}[theorem]{Definition}
\newtheorem{example}[theorem]{Example}
\theoremstyle{remark}
\newtheorem{rem}[theorem]{Remark}
\numberwithin{equation}{section}
\newcommand{\gf}{\ensuremath{\mathbb{K}}}
\newcommand{\sg}[1]{\ensuremath{\mathbb{S}_{#1}}}
\newcommand{\cyc}[1]{\ensuremath{\mathbb{Z}/{#1}\mathbb{Z}}}
\newcommand{\cycsum}[1]{\ensuremath{\Theta_{#1}}}
\newcommand{\mat}[2]{\ensuremath{\mathrm{M}_{#1}\left(#2\right)}}
\newcommand{\gl}[2]{\ensuremath{\mathfrak{gl}_{#1}\left(#2\right)}}
\newcommand{\un}[1]{\ensuremath{\mathfrak{u}_{#1}}}
\newcommand{\cotimes}{\ensuremath{\hat{\otimes}}}
\newcommand{\innprod}[1][-,-]{\ensuremath{\langle #1 \rangle}}
\newcommand{\innproddelim}[1][-,-]{\ensuremath{\left\langle #1 \right\rangle}}
\newcommand{\innprodloc}[1][-,-]{\ensuremath{\langle #1 \rangle_{\mathrm{loc}}}}
\newcommand{\OTFT}[2]{\ensuremath{\mathbf{T}^{#1}_{#2}}}
\newcommand{\Morita}[1][\textgoth{A}]{\ensuremath{\mathcal{M}_{\gamma,\nu}^{#1}}}
\newcommand{\den}[1]{\ensuremath{|\Lambda_{#1}|}}
\newcommand{\csalg}[1]{\ensuremath{\widehat{S}\left(#1\right)}}
\newcommand{\csalgP}[1]{\ensuremath{\widehat{S}_+\left(#1\right)}}
\newcommand{\csalgdelim}[1]{\ensuremath{\widehat{S} #1}}
\newcommand{\csalgPdelim}[1]{\ensuremath{\widehat{S}_+ #1}}
\newcommand{\Surface}[2]{\ensuremath{\Sigma^{#1}_{#2}}}
\newcommand{\smooth}[1]{\ensuremath{C^{\infty}(#1)}}
\newcommand{\smoothlimit}[1]{\ensuremath{C^{\infty}_{\geq 0}(#1)}}
\newcommand{\smoothsing}[1]{\ensuremath{C^{\infty}_{<0}(#1)}}
\newcommand{\ptree}[1][p]{\ensuremath{\mathscr{T}_{#1}}}
\newcommand{\intcomm}[1]{\ensuremath{\mathscr{O}_{\hbar}\left(#1\right)}}
\newcommand{\intcommP}[1]{\ensuremath{\mathscr{O}_{\hbar}^+\left(#1\right)}}
\newcommand{\intcommI}[1]{\ensuremath{\mathscr{O}_{\hbar}^{\mathrm{Int}}\left(#1\right)}}
\newcommand{\intcommL}[1]{\ensuremath{\mathscr{O}_{\hbar}^{\mathrm{Loc}}\left(#1\right)}}
\newcommand{\intcommLI}[1]{\ensuremath{\mathscr{O}_{\hbar}^{\mathrm{LocInt}}\left(#1\right)}}
\newcommand{\intcommLIdelim}[1]{\ensuremath{\mathscr{O}_{\hbar}^{\mathrm{LocInt}}#1}}
\newcommand{\KontHam}[1]{\ensuremath{\mathscr{H}\left(#1\right)}}
\newcommand{\KontHamP}[1]{\ensuremath{\mathscr{H}_+\left(#1\right)}}

\newcommand{\KontHamPdelim}[1]{\ensuremath{\mathscr{H}_+ #1}}
\newcommand{\intnuP}[1]{\ensuremath{\mathscr{N}_{\nu}\left(#1\right)}}
\newcommand{\intnoncomm}[1]{\ensuremath{\mathscr{N}_{\gamma,\nu}\left(#1\right)}}
\newcommand{\intnoncommP}[1]{\ensuremath{\mathscr{N}_{\gamma,\nu}^+\left(#1\right)}}
\newcommand{\intnoncommI}[1]{\ensuremath{\mathscr{N}_{\gamma,\nu}^{\mathrm{Int}}\left(#1\right)}}
\newcommand{\intnoncommL}[1]{\ensuremath{\mathscr{N}_{\gamma,\nu}^{\mathrm{Loc}}\left(#1\right)}}
\newcommand{\intnoncommLI}[1]{\ensuremath{\mathscr{N}_{\gamma,\nu}^{\mathrm{LocInt}}\left(#1\right)}}
\newcommand{\intnoncommItree}[1]{\ensuremath{\mathscr{N}_{\gamma,\nu}^{\mathrm{Tree}}\left(#1\right)}}
\newcommand{\NCEffThy}[1]{\ensuremath{\mathbf{NCEffThy}\left(#1\right)}}
\newcommand{\NCEffThyL}[2]{\ensuremath{\mathbf{NCEffThy}_{#1}\left(#2\right)}}
\newcommand{\EffThy}[1]{\ensuremath{\mathbf{EffThy}\left(#1\right)}}
\newcommand{\EffThyL}[2]{\ensuremath{\mathbf{EffThy}_{#1}\left(#2\right)}}
\DeclareMathAlphabet{\mathpzc}{OT1}{pzc}{m}{it}
\DeclareMathOperator{\Hom}{Hom}
\DeclareMathOperator{\Aut}{Aut}
\DeclareMathOperator{\Bij}{Bij}
\DeclareMathOperator{\Tr}{Tr}

\DeclareMathOperator{\Sing}{Sing}
\begin{document}
\title{Noncommutative effective field theories and the large $N$ correspondence}
\author{Alastair Hamilton}
\address{Department of Mathematics and Statistics, Texas Tech University, Lubbock, TX 79409-1042. USA.} \email{alastair.hamilton@ttu.edu}
\begin{abstract}
We integrate the notion of an effective field theory, as described by Costello, with the framework of noncommutative symplectic geometry introduced by Kontsevich; providing a definition for the renormalization group flow in noncommutative geometry that is defined through the use of ribbon graphs. As in the commutative case, the resulting noncommutative effective field theories are in one-to-one correspondence with local interaction functionals. We explain how in this setting, the large $N$ correspondence discovered by 't~Hooft appears as a relation between noncommutative and commutative effective field theories. As an example, we apply this framework to study a noncommutative analogue of Chern-Simons theory.
\end{abstract}
\makeatletter
\@namedef{subjclassname@2020}{%
  \textup{2020} Mathematics Subject Classification}
\makeatother
\maketitle
\small
\tableofcontents
\normalsize

\section{Introduction}

In this paper we will combine the framework of effective field theory that was developed by Costello in \cite{CosEffThy}, with the machinery of noncommutative symplectic geometry that was introduced by Kontsevich in \cite{KontSympGeom}. In the former, our interactions live in a commutative space of functions on the fields, given by a powerseries algebra of distributions on the space of smooth sections of a graded vector bundle. In the latter, full commutative symmetry is replaced with a more restrictive symmetry described by the action of the cyclic groups. This is the kind of notion of symmetry that appears in Feynman diagrams when we replace the usual world-line picture of particles interacting, with world-sheets formed by open strings.

\subsection{Background}
\subsubsection{Renormalization and effective field theory}

The definition of an \emph{effective field theory}, as we will use the term for the remainder of the paper, was formulated by Costello in \cite{CosEffThy}, in which he gave a mathematically precise definition for a quantum field theory in the perturbative Lagrangian formulation, based on the ideas of Kadanoff \cite{KadScaling}, Wilson \cite{WilRGcrit, WilRenormScalar}, Polchinski \cite{PolRenormLag} and others. In an effective field theory (also known as a low-energy effective field theory), a family of effective actions are specified, parameterized by the length scale (or equivalently, the energy scale) cut-off for the theory. These interactions must be suitably local in nature and related by the \emph{renormalization group flow}. The basic idea may be stated simply enough: these low-energy effective interactions should be all that is required to study phenomena up to the prescribed energy cut-off.

One of the main theorems that Costello proves in \cite{CosEffThy} is that there is a one-to-one correspondence between local functionals and effective field theories. This correspondence is defined through \emph{renormalization}. In this paper, we formulate and prove a noncommutative analogue of this result.

Separately, in \cite{CosEffThy}, Costello explains how to treat \emph{gauge theories} in the framework of effective field theory, which is done through the use of the Batalin-Vilkovisky formalism. The treatment of gauge theories lies outside the scope of the current article, but we take the subject up in earnest in \cite{NCRBV}, where we explain how to bring the noncommutative geometry that we apply here into this picture as well.

\subsubsection{Noncommutative symplectic geometry}\label{sec_introNCgeom}

We should begin by being careful to emphasize that the noncommutative geometry that we consider in this article arises on the space of fields, and the associated space of functionals; not on the spacetime itself (as it does, for example, through Moyal quantization \cite{KontDQ}). The type of noncommutative geometry that we employ in the current article was first considered by Kontsevich in \cite{KontSympGeom}, where it was shown to have a deep connection to moduli spaces of Riemann surfaces. This connection arises from the results of Harer \cite{HarOrbicells}, Mumford, Penner \cite{PenOrbicells} and Thurston; which provide an orbicellular decomposition of this moduli space in terms of \emph{ribbon graphs}. These ribbon graphs describe diagrams of interacting open strings---their vertices possess a natural cyclic structure and may be decorated by the interaction terms for our quantum field theory. This leads to a type of noncommutative geometry on the space of all such interaction functionals.

All of this may be extended to compactifications of the moduli space of Riemann surfaces. Such compactifications were considered by Looijenga in \cite{LooCompact} and by Kontsevich in \cite{KontAiry}, where in the latter they played a crucial role in his proof of Witten's conjectures. In this case, the orbicellular decomposition is described in terms of \emph{stable ribbon graphs}. In \cite{baran}, Barannikov explained how the noncommutative geometry of Kontsevich \cite{KontSympGeom} is modified when passing to the compactification described in \cite{KontAiry}. The corresponding noncommutative geometry for the compactification described by Looijenga in \cite{LooCompact} was described in \cite{HamCompact}. The latter is important, as the compactification picks up extra strata that arise from the degeneration of the boundary components of Riemann surfaces, and this leads to an extra parameter $\nu$ keeping track of this phenomena in the noncommutative geometry. As we shall see later, this parameter is crucial in describing large $N$ phenomena, cf.~\cite{GiGqHaZeLQT, GwHaZeGUE}.

\subsubsection{Noncommutative effective field theory}

In this paper we formulate a definition for a \emph{noncommutative effective field theory}. The preceding statement is subject to the same qualification as above---the purpose here is to incorporate the noncommutative geometry of Kontsevich \cite{KontSympGeom} into the work of Costello \cite{CosEffThy}, not to study quantum field theory on a deformation quantized spacetime; the latter being perhaps the more commonly understood meaning of this term, cf. \cite{NoteonNCCSThy}.

In this framework, the renormalization group flow is defined through the use of stable ribbon graphs. Following Costello, we prove noncommutative analogues of many of the foundational results of \cite{CosEffThy}. This includes establishing the fundamental properties of the renormalization group flow in this noncommutative setting, including:
\begin{itemize}
\item
the group identities for the flow, which describe it as an action by the abelian group of propagators on the space of interaction functionals, and leads to the definition of the renormalization group equation for an effective field theory;
\item
the existence (and uniqueness) of a set of counterterms for any local interaction,
\item
the use of these counterterms to define a noncommutative effective field theory through the process of renormalization, and
\item
that this process of renormalization defines a one-to-one correspondence between local interactions and noncommutative effective field theories.
\end{itemize}

One of the basic features of a noncommutative effective field theory is that it lifts the definition of an effective field theory, as defined in \cite{CosEffThy}, to the realm of noncommutative geometry. There is a natural way to pass from the noncommutative geometry that we use in this paper to the commutative geometry used by Costello in \cite{CosEffThy}. This is ultimately performed in a fairly straightforward fashion, by simply taking a further quotient by the additional commutativity relations. We show that under this map, the renormalization group flow in noncommutative geometry is transformed into the flow on commutative geometry. This implies, in particular, that any noncommutative effective field theory produces an effective field theory under this transformation.

Another important feature of noncommutative effective field theories is that any two-dimensional Open Topological Field Theory (OTFT) defines a transformation on the set of all such theories. In fact, it is precisely the axioms for an OTFT that ensure that such a transformation is well-defined, and has the requisite properties. It is of course well-known that an OTFT is the same thing as a Frobenius algebra \cite{AtiyahTFT, ChLaOTFT}. The transformation defined by this OTFT will take the space of fields to its tensor product with this Frobenius algebra. If we take as our input the algebra of $N$-by-$N$ matrices, then we get a family of transformations for a noncommutative effective field theory indexed by the rank $N$. This will allow us to formulate a version of the large $N$ correspondence in the spirit of 't~Hooft \cite{MAGOO, tHooftplanar}.

\subsubsection{The large $N$ correspondence}

Combining the transformations described in the preceding section, we may produce from any noncommutative effective field theory, a family of (commutative) effective field theories indexed by a positive integer $N$. The procedure is that first, given a noncommutative effective field theory, we apply the transformations defined by the OTFTs that are constructed from the Frobenius algebras of $N$-by-$N$ matrices. This yields a family of noncommutative effective field theories indexed by the rank $N$, from which we may then pass to the corresponding (commutative) effective field theories. This is represented by Diagram \eqref{dig_largeN}.

If the fields that we start with are differential forms on a manifold, then the fields that we end up with will be matrix-valued differential forms---in other words, they will represent connections on the manifold. As we explained in Section \ref{sec_introNCgeom} above, a noncommutative effective field theory describes a type of open string theory. The effective field theories that are spawned from this noncommutative theory will then be the corresponding gauge theories. Again, a more complete discussion of gauge theories will be provided in \cite{NCRBV}, in which we will also examine this large $N$ correspondence.

This, in principle, allows one to study this family of effective field theories using the noncommutative theory from which they originate. In fact, more is possible; in the reverse direction, there is a way to go backwards and deduce properties of the originating noncommutative effective field theory from the properties of the generated effective field theories. This arises (somewhat indirectly) as a consequence of the Loday-Quillen-Tsygan Theorem from algebraic K-theory \cite{LodayQuillen, Tsygan}. From this theorem we may deduce that a term vanishes in the noncommutative theory if and only if the corresponding term vanishes for all the commutative theories.

\subsubsection{Noncommutative Chern-Simons theory}

As an application of the ideas that are laid out in this paper, we introduce and briefly study a noncommutative analogue of Chern-Simons theory. Again, we should make clear that this differs from the approach---taken in \cite{SussCS} for instance---which makes use of the Moyal star product. Instead, the noncommutative version of Chern-Simons that we produce here has more in common with Costello's treatment of Chern-Simons theory in \cite{CosTCFT} as a differential form on the moduli space of Riemann surfaces, and with Witten's treatment of the theory in \cite{WitCSstring} as an open string theory.

Treating Chern-Simons theory as an open string theory leads to a definition of Chern-Simons theory as a noncommutative effective field theory. Under the large $N$ correspondence described in the preceding section, this noncommutative analogue of Chern-Simons theory yields effective field theories at all ranks $N$ defined on the same spacetime manifold. We show that these theories are precisely the $U(N)$ Chern-Simons theories. The calculation itself is quite straightforward, once the basic framework is set up.

As a simple demonstration of what may be achieved with the methods of this paper, we analyze the counterterms for our noncommutative Chern-Simons theory. It is known from the work of Axelrod-Singer \cite{AxelSingI, AxelSingII} and Kontsevich \cite{KontFeyn} that---by an extensive analysis of the counterterms---the counterterms for regular (commutative) Chern-Simons theory all vanish on a flat manifold. Presumably, a similar analysis employing the same techniques can be used to show that the counterterms for our noncommutative Chern-Simons theory all vanish as well. However, we prefer to deduce this result as a simple consequence of the large $N$ correspondence that we described above that makes use of the Loday-Quillen-Tsygan theorem.

Of course, in this paper we work over a compact manifold, and there are not many flat compact three-manifolds (in fact, there are exactly six orientable ones). In \cite{CosBVrenormalization}, Costello circumvents this problem by proving that effective field theories form a sheaf over the manifold, allowing him to use the same techniques to deal with curved spacetimes; however, this particular subject would take us too far afield and lies outside the scope of this article. We will have more to say about noncommutative Chern-Simons theory in the sequel \cite{NCRBV}.

It is worth mentioning that the same set of techniques should be equally applicable to Yang-Mills theory. That is, there should be a noncommutative analogue of Yang-Mills theory that generates the usual $U(N)$ Yang-Mills gauge theories under the large $N$ correspondence. In \cite[\S 6.5]{CosEffThy} Costello proves the renormalizability of Yang-Mills theory under certain conditions on the gauge group, and we believe that a similar application of our large $N$ correspondence should yield results here too.

\subsection{Acknowledgements}

The author is grateful to Owen Gwilliam, who successfully persuaded him that the Loday-Quillen-Tsygan theorem was relevant to the study of large $N$ limits of gauge theories, during a visit a number of years ago. The author also gratefully acknowledges many helpful conversations with Dmitri Pavlov, and with Mahmoud Zeinalian.

\subsection{Layout of the paper}

We begin in Section \ref{sec_freethy} by recalling the basic definition of a free theory. This structure will always serve as our initial starting point. In Section \ref{sec_RGflow} we introduce some basic definitions for graphs and Feynman amplitudes, and use these to define the renormalization group flow in noncommutative geometry and establish some of its basic properties. In Section \ref{sec_NCeffthy}, we introduce the definition of a noncommutative effective field theory. We explain how any noncommutative theory produces a commutative theory, and prove one of our main results stating that noncommutative effective field theories are in one-to-one correspondence with local functionals. In Section \ref{sec_largeN} we discuss the large $N$ correspondence in its most general context, and formulate a vanishing criteria using the Loday-Quillen-Tsygan Theorem. Finally, in Section \ref{sec_CSthy}, we provide an example of a noncommutative effective field theory in the form of a noncommutative analogue of Chern-Simons theory, and prove that on a flat spacetime it requires no counterterms. At the end of the paper we have included a brief appendix on topological vector spaces, in which we collect some standard facts and definitions.

\subsection{Notation and conventions}

For the rest of the paper, we will work over a ground field $\gf$, which will be either the real or complex numbers. We denote the algebra of $N$-by-$N$ matrices with entries in $\gf$ by $\mat{N}{\gf}$. Throughout the paper we will work with $\mathbb{Z}$-graded locally convex Hausdorff topological vector spaces $\mathcal{V}$ over $\gf$. Our convention will be to work with \emph{cohomologically} graded spaces, hence the suspension $\Sigma\mathcal{V}$ of a graded topological vector space $\mathcal{V}$ is defined by $\Sigma \mathcal{V}^i := \mathcal{V}^{i+1}$.

We will denote the space of \emph{continuous} $\gf$-linear maps between topological vector spaces $\mathcal{V}$ and $\mathcal{W}$ by
\[ \Hom_{\gf}\left(\mathcal{V},\mathcal{W}\right). \]
The continuous $\gf$-linear dual of $\mathcal{V}$ will be denoted by $\mathcal{V}^{\dag}$. These spaces will always be provided with the strong topology of uniform convergence on bounded sets.

The completed projective tensor product of two locally convex Hausdorff topological vector spaces $\mathcal{V}$ and $\mathcal{W}$ will be denoted by
\[ \mathcal{V}\cotimes\mathcal{W}. \]
The ordinary tensor product of two vector spaces will be denoted using the standard notation~$\otimes$. We shall denote by $\tau_{\mathcal{U},\mathcal{V}}$ (or just simply $\tau$), the permutation
\[ \mathcal{U}\cotimes\mathcal{V}\to\mathcal{V}\cotimes\mathcal{U}, \qquad u\otimes v \mapsto (-1)^{|v||u|}v\otimes u. \]

Given a locally convex Hausdorff topological vector space $\mathcal{V}$, we define the completed symmetric algebra of $\mathcal{V}$ and its nonunital counterpart by
\[ \csalg{\mathcal{V}}:=\prod_{i=0}^\infty \bigl[\mathcal{V}^{\cotimes i}\bigr]_{\sg{i}} \quad\text{and}\quad \csalgP{\mathcal{V}}:= \prod_{i=1}^\infty \bigl[\mathcal{V}^{\cotimes i}\bigr]_{\sg{i}}, \]
where $\sg{i}$ denotes the symmetric group. Here we have adopted the standard convention of denoting coinvariants using a subscript and invariants using a superscript.

Given a smooth manifold $M$, the algebra of smooth $\gf$-valued functions on $M$ will be denoted by $\smooth{M,\gf}$, with smooth real-valued functions being simply denoted by $\smooth{M}$. Likewise, the de Rham algebra of $\gf$-valued forms on $M$ will be denoted by $\Omega^{\bullet}(M,\gf)$. If $E$ is a vector bundle over $M$ then the space $\mathcal{E}:=\Gamma(M,E)$ of smooth sections of $E$ will always be endowed with the $C^{\infty}$-topology of uniform convergence of sections and their derivatives on compact sets, see the appendix for a brief definition and discussion of this topology.

The number of elements in a finite set $X$ will be denoted by $|X|$. In formulas, the identity map on a set $X$ will be denoted by $\mathds{1}_X$, or simply by $\mathds{1}$.

\section{Free theories} \label{sec_freethy}

In this section we will introduce and recall the most basic principle structure that we will be working with for the duration of the paper---the notion of a free theory, which is the starting point for an interacting theory. This will include the structure of the space of fields, the construction of propagators, and the notion of a local distribution and operator.

\subsection{Definition of a free theory}

\begin{defi} \label{def_freethy}
A \emph{free theory} consists of the following data:
\begin{enumerate}
\item \label{itm_freethy1}
A smooth compact manifold $M$.
\item
A $\mathbb{Z}$-graded $\gf$-vector bundle $E$ over $M$. We denote the space of smooth sections by $\mathcal{E}:=\Gamma(M,E)$.
\item \label{itm_freethy3}
A local pairing, consisting of a map of vector bundles
\[ \innprodloc:E\otimes E \to \den{M}\underset{\mathbb{R}}{\otimes}\gf \]
of some homogeneous degree $n$, where the real density bundle $\den{M}$ is concentrated in degree zero. This pairing must be nondegenerate over each fiber and graded (skew-)symmetric according to the parity of $n$:
\[ \innprodloc[v_1,v_2] = (-1)^{n+|v_1||v_2|}\innprodloc[v_2,v_1]. \]
\item
A generalized Laplacian
\[ H:\mathcal{E}\to\mathcal{E}. \]
This is a second order differential operator of degree zero, whose principal symbol $\sigma_2(H)$ is multiplication by a Riemannian metric $g$ on $M$, see \cite[\S 2.1]{BerGetVer}.
\end{enumerate}
The local pairing $\innprodloc$ determines an integration pairing on sections
\[ \innprod:\mathcal{E}\otimes\mathcal{E}\to\gf, \qquad  \innprod[s_1,s_2]:=\int_M \innprodloc[s_1,s_2]. \]
We require that the generalized Laplacian $H$ be self-adjoint under this pairing;
\[ \innprod[H s_1,s_2] = \innprod[s_1,H s_2]; \quad s_1,s_2\in\mathcal{E}. \]
\end{defi}

\begin{example} \label{exm_freescalar}
Choosing the trivial line bundle $E:=\gf$ on a compact Riemannian manifold $M$, the Riemannian density defines a local pairing and we may choose the Laplacian on $\mathcal{E}=\smooth{M,\gf}$ to be $(H+m^2)$, where $H$ is the canonical Laplace-Beltrami operator determined by the Riemannian metric. This is the starting point for scalar field theories with kinetic term given by
\[ -\frac{1}{2}\innproddelim[s,(H+m^2)s]= -\frac{1}{2}\int_M s (H+m^2)s, \quad s\in\mathcal{E}. \]
\end{example}

\begin{example} \label{exm_freeCSthy}
Let $\mathfrak{g}$ be a Lie algebra with a symmetric nondegenerate pairing $\innprod_{\mathfrak{g}}$. Choosing the bundle
\[E:=\Sigma\Lambda^{\bullet} T^*M\underset{\mathbb{R}}{\otimes}\mathfrak{g} \]
on a compact oriented Riemannian manifold $M$ of odd dimension $n$ yields the space of fields
\[ \mathcal{E}=\Sigma\Omega^{\bullet}(M,\gf)\underset{\mathbb{\gf}}{\otimes}\mathfrak{g} \]
of $\mathfrak{g}$-valued differential forms. Combining the pairing on $\mathfrak{g}$ with the exterior multiplication of forms, the orientation of $M$ yields a local pairing on $E$ of degree $(2-n)$. Note that this pairing becomes skew-symmetric after the shift in parity. The Hodge-Laplacian determines a generalized Laplacian $H:=[Q,Q^*]$ on $\mathcal {E}$, where $Q$ is the exterior derivative and $Q^*$ is its customary Hodge adjoint---since the dimension $n$ of $M$ is odd, $H$ will be self-adjoint under the pairing. This is the starting point for Chern-Simons types theories, with kinetic term (for $n=3$) given by
\[ \frac{1}{2}\innprod[A,QA]_{\mathcal{E}} = \frac{1}{2}\int_M\innprod[A, QA]_{\mathfrak{g}}, \quad A\in\mathcal{E}^0 = \Omega^1(M,\gf)\underset{\mathbb{\gf}}{\otimes}\mathfrak{g}. \]
\end{example}

\subsection{Propagators}

The kinetic term of a free theory will be accommodated through the appropriate choice of a propagator, which will be described in terms of the heat kernel. To define the heat kernel, we first introduce the convolution operator.

\begin{defi}
Given a free theory $\mathcal{E}$ as above, the convolution operator
\[ \star:\mathcal{E}\cotimes\mathcal{E}\to\Hom_{\gf}(\mathcal{E},\mathcal{E}) \]
is defined as follows. Given a homogeneous $K\in\mathcal{E}\cotimes\mathcal{E}$, define
\[ K\star:\mathcal{E}\to\mathcal{E}, \qquad K\star(s):=(-1)^{|K|}\left(\mathds{1}\cotimes\innprod\right)[K\otimes s]. \]
\end{defi}

Given the data of a free theory as above, the results of \cite[\S 2.3-2.5]{BerGetVer} show the existence of a unique heat kernel
\[ K\in\smooth{0,\infty}\underset{\mathbb{R}}{\cotimes}\mathcal{E}\cotimes\mathcal{E} = \Gamma\left((0,\infty)\times M\times M,E\boxtimes E\right). \]
If $t\in (0,\infty)$ then we denote the value of the heat kernel $K$ at the point $t$ by $K_t\in\mathcal{E}\cotimes\mathcal{E}$. The heat kernel possesses the following two defining properties:
\[ \left(\frac{\mathrm{d}}{\mathrm{d}t}\cotimes\mathds{1}\cotimes\mathds{1} + \mathds{1}\cotimes H\cotimes\mathds{1}\right)[K] = 0
\quad\text{and}\quad
\lim_{t\to 0}\left(K_t\star(s)\right) =  s, \quad\text{for all }s\in\mathcal{E}. \]

Note that if the pairing $\innprod$ has degree~$n$ then the heat kernel $K$ must have degree~$-n$ so that the operator $K_t\star$ has degree zero. Since the generalized Laplacian $H$ is self-adjoint, the operator $K_t\star$ will be self-adjoint too; cf. Proposition 2.17 of \cite{BerGetVer}. When this fact is combined with the (skew-)symmetry (according to its parity) of the pairing $\innprod$, it follows that the heat kernel $K_t\in\mathcal{E}\cotimes\mathcal{E}$ is symmetric; that is, it is invariant under the action of $\sg{2}$.

\begin{defi}
A \emph{propagator} for a free theory $\mathcal{E}$ is a symmetric  degree zero tensor
\[ P\in\bigl[\mathcal{E}\cotimes\mathcal{E}\bigr]^{\sg{2}}. \]
\end{defi}

\begin{rem}
In some sense, it would appear reasonable to try to include the propagator as part of the structure of a free theory, since it in principle determines the kinetic term of the theory. However, this is not so straightforwardly done, as we will soon see that what is required for the renormalization group flow is in fact a family of such propagators parameterized using length scale cutoffs for the theory.
\end{rem}

\begin{example}
Consider the free scalar field theory $\mathcal{E}:=\smooth{M,\gf}$ from Example \ref{exm_freescalar}. We may use the heat kernel $K\in\smooth{0,\infty}\cotimes_{\mathbb{R}}\mathcal{E}\cotimes\mathcal{E}$ for this free theory to define a family of propagators for this theory by
\[ P(\varepsilon,L) := \left(\int_{t=\varepsilon}^L\mathrm{d}t\cotimes\mathds{1}\cotimes\mathds{1}\right)[K]:(x,y)\longmapsto\int_{t=\varepsilon}^L K_t(x,y)\,\mathrm{d}t; \quad 0<\varepsilon<L<\infty. \]
\end{example}

\begin{example} \label{exm_CSpropagator}
Consider Chern-Simons theory
\[ \mathcal{E}=\Sigma\Omega^{\bullet}(M,\gf)\otimes\mathfrak{g} \]
on a three-manifold $M$, see Example \ref{exm_freeCSthy}. Using the heat kernel $K$ for this free theory we may define a family of propagators for this theory by
\[ P(\varepsilon,L) := \left(\int_{t=\varepsilon}^L\mathrm{d}t\cotimes Q^*\cotimes\mathds{1}\right)[K] = \int_{t=\varepsilon}^L\left(Q^*\cotimes\mathds{1}\right)[K_t]\,\mathrm{d}t; \quad 0<\varepsilon<L<\infty. \]
\end{example}

\subsection{Families of theories}

We will at some point have the need to consider families of free field theories indexed by some smooth parameter space $X$ with corners. This occurs for instance if we wish to explore how the scalar or Chern-Simons type field theories depend upon the Riemannian metric. This will be less relevant for us here, but will be more important in the sequel \cite{NCRBV}.

This is facilitated by considering the global sections $\mathcal{A}:=\Gamma(X,A)$ of some sheaf of graded commutative $\gf$-algebras, cf. Appendix 1 of \cite{CosEffThy}. We may safely limit ourselves to the case where $X:=\Delta^n$ is a simplex and $\mathcal{A}$ is either the algebra of smooth functions $\smooth{X,\gf}$ or the de Rham algebra $\Omega^{\bullet}(X,\gf)$. In either case, $\mathcal{A}$ will be a topological unital graded commutative $\gf$-algebra, equipped with the $C^{\infty}$-topology of Definition \ref{def_cinftopology}. We will denote the commutative multiplication on $\mathcal{A}$ by $\mu_{\mathcal{A}}$.

Any point $p\in X$ gives rise to a degree zero evaluation map from $\mathcal{A}$ to $\gf$. This will allow us to pass from a family of theories to an instance of a theory satisfying the requirements of Definition \ref{def_freethy}. Hence the reader that is uninterested in considering families of theories may simply take $\mathcal{A}:=\gf$ in everything that follows and hence ignore any further mention of the topic in the paper.

\begin{defi}
A \emph{family of free theories} over $\mathcal{A}$ consists of the following data:
\begin{itemize}
\item
Items \ref{itm_freethy1} to \ref{itm_freethy3} of Definition \ref{def_freethy}, as before.
\item
An operator of degree zero
\[ H:\mathcal{E}\to\mathcal{E}\cotimes\mathcal{A} \]
such that the canonical $\mathcal{A}$-linear extension
\[ \bigl(\mathds{1}\cotimes\mu_{\mathcal{A}}\bigr)\bigl(H\cotimes\mathds{1}\bigr):\mathcal{E}\cotimes\mathcal{A}\to\mathcal{E}\cotimes\mathcal{A}, \]
which we will also denote simply by $H$, is a differential operator.

Denoting by
\[ H_p:\mathcal{E}\to\mathcal{E} \]
the composition of $H$ with the evaluation map at a point $p\in X$, we require that $H_p$ is a generalized Laplacian for all $p\in X$.

Additionally, we require that $H$ be self-adjoint with respect to the pairing
\begin{equation} \label{eqn_innprodalg}
\innprod_{\mathcal{A}}:=\bigl(\innprod\cotimes\mu_{\mathcal{A}}\bigr)\bigl(\mathds{1}\cotimes\tau\cotimes\mathds{1}\bigr):\bigl(\mathcal{E}\cotimes\mathcal{A}\bigr)\cotimes\bigl(\mathcal{E}\cotimes\mathcal{A}\bigr)\to\mathcal{A}. \end{equation}
\end{itemize}
\end{defi}

The results from the appendix to Chapter 9 of \cite{BerGetVer} assert the existence of a heat kernel
\[ K\in\smooth{0,\infty}\underset{\mathbb{R}}{\cotimes}\mathcal{E}\cotimes\mathcal{E}\cotimes\mathcal{A}. \]
Let $K_t\in\mathcal{E}\cotimes\mathcal{E}\cotimes\mathcal{A}$ denote the evaluation of $K$ at a point $t>0$, and given $s\in\mathcal{E}\cotimes\mathcal{A}$ define
\[ K_t\star s := (-1)^{|K_t|}\bigl(\mathds{1}\cotimes\innprod_{\mathcal{A}}\bigr)[K_t\otimes s]\in\mathcal{E}\cotimes\mathcal{A}. \]
The heat kernel is characterized by the following two properties:
\[ \frac{\mathrm{d}}{\mathrm{d}t}\left(K_t\star s\right) = -H\left(K_t\star s\right) \quad\text{and}\quad \lim_{t\to 0}\left(K_t\star s\right) = s. \]
Note that, as before, the heat kernel $K$ and the pairing $\innprod$ must be of opposing degrees. If we denote by
\[ K|_p\in\smooth{0,\infty}\underset{\mathbb{R}}{\cotimes}\mathcal{E}\cotimes\mathcal{E} \]
the evaluation of $K$ at a point $p\in X$, then $K|_p$ will be the heat kernel for the generalized Laplacian $H_p$.

Since $H$ is self-adjoint, it follows that the operator $K_t\star$ is self-adjoint as well. Consequently, the heat kernel $K$ must be symmetric in $\mathcal{E}$.

\begin{defi}
A \emph{family of propagators} for a family of free theories $\mathcal{E}$ over $\mathcal{A}$ is a tensor
\[ P\in\mathcal{E}\cotimes\mathcal{E}\cotimes\mathcal{A} \]
of total degree zero that is symmetric in $\mathcal{E}$, that is $(\tau\cotimes\mathds{1})[P]=P$.
\end{defi}

\begin{example} \label{exm_canonicalpropagator}
Let $\mathcal{E}$ be a family of free theories over $\mathcal{A}$ and let
\[ D:\mathcal{E}\to\mathcal{E}\cotimes\mathcal{A} \]
be a homogeneous operator which has the same degree as the pairing $\innprod$ on $\mathcal{E}$, and whose canonical $\mathcal{A}$-linear extension to $\mathcal{E}\cotimes\mathcal{A}$ is a differential operator, which we will also denote by the same symbol $D$.

Consider the kernel $[D,K]$ given by
\[ [D,K]_t := (\mathds{1}_{\mathcal{E}\cotimes\mathcal{E}}\cotimes\mu)(\mathds{1}_{\mathcal{E}}\cotimes\tau\cotimes\mathds{1}_{\mathcal{A}})(D\cotimes\mathds{1}_{\mathcal{E}\cotimes\mathcal{A}})[K_t] - (\mathds{1}_{\mathcal{E}}\cotimes D)[K_t]. \]
Now suppose that $[D,K]=0$. This will occur for instance if:
\begin{itemize}
\item
the $\mathcal{A}$-linear operators $D$ and $H$ commute, and
\item
the operator $D$ is self-adjoint with respect to the pairing \eqref{eqn_innprodalg} in the sense that
\[ \innprod[Ds_1,s_2]_{\mathcal{A}} = (-1)^{|D||s_1|}\innprod[s_1,Ds_2]_{\mathcal{A}},\quad\text{for all }s_1,s_2\in\mathcal{E}. \]
\end{itemize}
In this case we may define a family of propagators for $\mathcal{E}$ over $\mathcal{A}$ by
\begin{equation} \label{eqn_canonicalpropagator}
P(\varepsilon,L):=\bigl(\mathds{1}_{\mathcal{E}\cotimes\mathcal{E}}\cotimes\mu\bigr)\bigl(\mathds{1}_{\mathcal{E}}\cotimes\tau\cotimes\mathds{1}_{\mathcal{A}}\bigr)\left(\int_{t=\varepsilon}^L\mathrm{d}t\cotimes D\cotimes\mathds{1}_{\mathcal{E}\cotimes\mathcal{A}}\right)[K]; \quad 0<\varepsilon,L<\infty.
\end{equation}
By definition, $P(L,\varepsilon)=-P(\varepsilon,L)$.
\end{example}

\subsection{Local distributions and operators}

When we come to discuss interacting theories, we will need the notion of a local functional. Given a free field theory $\mathcal{E}$, the space
\[ \bigl(\mathcal{E}^{\dag}\bigr)^{\cotimes l} = \bigl(\mathcal{E}^{\cotimes l}\bigr)^{\dag} = \Hom_{\gf}\bigl(\Gamma(M^l,E^{\boxtimes l}),\gf\bigr) \]
is a space of distributions.

\begin{defi} \label{def_localdistribution}
If $\mathcal{E}:=\Gamma(M,E)$ is the space of sections of a graded vector bundle $E$ over a compact manifold $M$ then we say that a distribution in $(\mathcal{E}^{\cotimes l})^{\dag}$ is a \emph{local distribution} if it can be written as a finite sum of distributions of the form
\begin{equation} \label{eqn_localdistribution}
s_1,\ldots,s_l \longmapsto \int_M D_1 s_1\cdots D_l s_l\,\mathrm{d}\varrho,
\end{equation}
where $D_1,\ldots D_l:\mathcal{E}\to\smooth{M,\gf}$ are differential operators and $\mathrm{d}\varrho\in\den{M}$ is a density.

More generally, if $\mathcal{E}$ is a family of free theories over $\mathcal{A}$ then we will say that an operator from $\mathcal{E}^{\otimes l}$ to $\mathcal{A}$ is a \emph{local $\mathcal{A}$-valued distribution} if it can be written as a finite sum of operators of the form \eqref{eqn_localdistribution}, where now each $D_i$ is an operator from $\mathcal{E}$ to $\smooth{M,\gf}\cotimes\mathcal{A}$ whose $\mathcal{A}$-linear extension to $\mathcal{E}\cotimes\mathcal{A}$ is a differential operator.
\end{defi}

\begin{example}
The kinetic terms from Example \ref{exm_freescalar} and Example \ref{exm_freeCSthy} provide examples of local distributions.
\end{example}

\section{The renormalization group flow} \label{sec_RGflow}

In this section we will define an analogue of the renormalization group flow in noncommutative geometry. We will explain how the noncommutative renormalization group flow lifts the customary (commutative) renormalization group flow, as defined by Costello in \cite[Ch. 2]{CosEffThy}, to the framework of noncommutative geometry.

\subsection{Feynman amplitudes}

The renormalization group flow will be defined in terms of a Feynman diagram expansion, and so we begin by recalling some basic material on graphs and Feynman amplitudes. While ordinary graphs appear in the definition of the renormalization group flow in commutative geometry \cite[\S 2.3]{CosEffThy}, defining the noncommutative renormalization group flow will involve replacing these graphs with ribbon graphs---in other words, the usual worldlines perspective for path integrals will be replaced through the use of worldsheets.

\subsubsection{Topological vector spaces attached to a set}

To make sense of the Feynman amplitudes, we begin by recalling some basic definitions from the theory of operads describing how we can attach vector spaces to sets using a monoidal structure, cf. \cite[\S 2]{GetKap}.

\begin{defi}
Let $\mathcal{V}$ be a $\mathbb{Z}$-graded nuclear space and $A$ a finite set of cardinality $n$. Define
\[ \mathcal{V}(\!(A)\!):=\Biggl[\bigoplus_{h\in\Bij\left(\{1,\ldots,n\},A\right)}\mathcal{V}^{\cotimes n}\Biggr]^{\sg{n}}, \]
where $\Bij\left(\{1,\ldots,n\},A\right)$ denotes the set of bijections and $\sg{n}$ acts simultaneously on both this set of bijections and the tensor product $\mathcal{V}^{\cotimes n}$. This is a functor on the groupoid of finite sets.
\end{defi}

There is an obvious family of isomorphisms
\begin{equation} \label{eqn_labelmaps}
\vartheta_h:\mathcal{V}^{\cotimes n}\to\mathcal{V}(\!(A)\!), \quad h\in\Bij\left(\{1,\ldots,n\},A\right);
\end{equation}
satisfying $\vartheta_{h\sigma}=\vartheta_h\circ\sigma$, and hence there are canonical maps
\begin{equation} \label{eqn_attachmapssym}
\bigl[\mathcal{V}^{\cotimes n}\bigr]^{\sg{n}}\to\mathcal{V}(\!(A)\!)\to\bigl[\mathcal{V}^{\cotimes n}\bigr]_{\sg{n}}.
\end{equation}

We have canonical identifications
\begin{equation} \label{eqn_tensorunion}
\mathcal{V}(\!(A\sqcup B)\!) = \mathcal{V}(\!(A)\!)\cotimes\mathcal{V}(\!(B)\!)
\end{equation}
and, when $\mathcal{V}$ is a nuclear Fr\'echet space, the identification\footnote{Here we are careful to apply the Koszul sign rule $(f\otimes g)[x\otimes y]=(-1)^{|g||x|}f(x)g(y)$.}
\[ \mathcal{V}^{\dag}(\!(A)\!) = \mathcal{V}(\!(A)\!)^{\dag}; \]
cf. Proposition 50.7 of \cite{Treves}.

More generally, if $\mathscr{X}$ is a finite set of cardinality $k$ and $\mathcal{V}$ is a function that assigns to every element $X$ of $\mathscr{X}$, a nuclear space $\mathcal{V}_X$, then we define
\[ \underset{X\in\mathscr{X}}{\widehat{\bigotimes}}\mathcal{V}_X := \Biggl[\bigoplus_{\mathbf{X}\in\Bij(\{1,\ldots,k\},\mathscr{X})}\mathcal{V}_{X_1}\cotimes\ldots\cotimes\mathcal{V}_{X_k}\Biggr]^{\sg{k}}. \]
Now if $\mathscr{X}$ is a finite collection of finite sets and $\mathcal{V}$ is a nuclear space, there is then a canonical identification
\[ \underset{A\in\mathscr{X}}{\widehat{\bigotimes}}\mathcal{V}(\!(A)\!) = \mathcal{V}\Bigl(\!\Bigl(\coprod_{A\in\mathscr{X}}A\Bigr)\!\Bigr). \]
Further, if $v$ is a function that assigns to every set $A\in\mathscr{X}$ a tensor $v_A\in\mathcal{V}(\!(A)\!)$ such that every element $v_A$, except possibly one of them, has even degree, then there is a canonically  defined tensor
\[ \underset{A\in\mathscr{X}}{\otimes} v_A \in \underset{A\in\mathscr{X}}{\widehat{\bigotimes}}\mathcal{V}(\!(A)\!) = \mathcal{V}\Bigl(\!\Bigl(\coprod_{A\in\mathscr{X}}A\Bigr)\!\Bigr). \]

\subsubsection{Stable graphs and commutative geometry}

Here we will recall from \cite{GetKap} the definition of a stable graph and from \cite{CosEffThy} the corresponding definition of the Feynman amplitude for a space of interactions arising from polynomial functions on the space of fields. Such a space of interactions is part of a commutative algebra; that is, it is based on a kind of commutative geometry for the space of fields. We will then subsequently generalize this to a space based on the noncommutative geometry of Kontsevich \cite{KontSympGeom}.

\begin{defi} \label{def_stablegraph}
A \emph{stable graph} $\Gamma$ consists of the following data:
\begin{itemize}
\item
A finite set $H(\Gamma)$, whose members are called the \emph{half-edges} of $\Gamma$.
\item
A finite set $V(\Gamma)$, whose members are called the \emph{vertices} of $\Gamma$.
\item
A map $\rho_{\Gamma}:H(\Gamma)\to V(\Gamma)$. The preimage of a vertex $v\in V(\Gamma)$ under this map consists of all those \emph{half-edges incident to the vertex $v$}. By an abuse of notation, we will denote this preimage by $v$ also and define the valency of $v$ to be its cardinality~$|v|$.
\item
An involution $\kappa_{\Gamma}:H(\Gamma)\to H(\Gamma)$. The fixed points of $\kappa_{\Gamma}$ are called the \emph{legs} of $\Gamma$, and the subset of all legs will be denoted by $L(\Gamma)$. Those orbits generated by $\kappa_{\Gamma}$ that consist of two elements will be called the \emph{edges} of $\Gamma$ and will be denoted by~$E(\Gamma)$.
\item
A map $l$ that assigns to every vertex $v\in V(\Gamma)$ a nonnegative integer $l(v)$ called the \emph{loop number} of $v$.
\end{itemize}
Additionally, we require the following restriction on the valencies of the vertices:
\begin{equation} \label{eqn_valencyconditions}
2l(v)+|v|\geq 3, \quad\text{for all }v\in V(\Gamma).
\end{equation}

We define the \emph{loop number} of the stable graph $\Gamma$ by
\[ \ell(\Gamma) := \mathbf{b}_1(\Gamma) + \sum_{v\in V(\Gamma)}l(v), \]
where $\mathbf{b}_1(\Gamma)$ is the first Betti number of the graph.
\end{defi}

From \eqref{eqn_valencyconditions} we may easily arrive at the inequalities:
\begin{equation} \label{eqn_graphinequalities}
\begin{split}
2\ell(\Gamma)+|L(\Gamma)| &\geq 2+|V(\Gamma)| \geq 3, \\
|H(\Gamma)| &\leq 3|L(\Gamma)| + 6(\ell(\Gamma) - 1);
\end{split}
\end{equation}
which hold for any connected\footnote{By convention, the empty graph is not considered to be connected.} stable graph $\Gamma$.

\begin{rem}
The nomenclature ``stable'' derives from the connection between such graphs and moduli spaces of stable curves, see for example \cite{DelMum} and \cite[\S 2]{GetKap}. Typically, in those settings the number $l(v)$ is instead called the ``genus'', as it refers to the genus of an irreducible component of the curve. However, this terminology will conflict with our later use of this term and hence we have renamed it as the ``loop number'' here, since it will essentially count the number of loops in our graph.
\end{rem}

\begin{defi}
Given a free field theory $\mathcal{E}$, consider the algebra of $\gf[[\hbar]]$-valued functionals on the fields
\[ \intcomm{\mathcal{E}}:=\csalg{\mathcal{E}^{\dag}}[[\hbar]] = \prod_{i=0}^\infty\hbar^i\csalg{\mathcal{E}^{\dag}} = \gf[[\hbar]]\cotimes\csalg{\mathcal{E}^{\dag}}. \]
We may write any functional $I\in\intcomm{\mathcal{E}}$ as a series
\[ I = \sum_{i,j=0}^{\infty}\hbar^i I_{ij}, \quad\text{where } I_{ij}\in\bigl[(\mathcal{E}^{\dag})^{\cotimes j}\bigr]_{\sg{j}}. \]
We will say that a functional $I\in\intcomm{\mathcal{E}}$ is an \emph{interaction} if it has degree zero, and if for all $i,j\geq 0$;
\begin{equation} \label{eqn_interactionconstraints}
I_{ij} = 0, \quad\text{whenever }2i+j<3.
\end{equation}
Denote the subspace of $\intcomm{\mathcal{E}}$ consisting of all those $I$ satisfying \eqref{eqn_interactionconstraints} by $\intcommP{\mathcal{E}}$, and the subspace consisting of all interactions by $\intcommI{\mathcal{E}}$.

More generally, if $\mathcal{E}$ is a family of free theories over $\mathcal{A}$ then we will say that
\[ I\in\intcomm{\mathcal{E},\mathcal{A}}:=\intcomm{\mathcal{E}}\cotimes\mathcal{A} \]
is a \emph{family of interactions} if it has total degree zero and if it satisfies the constraints given by \eqref{eqn_interactionconstraints}, where now each~$I_{ij}\in\bigl[(\mathcal{E}^{\dag})^{\cotimes j}\bigr]_{\sg{j}}\cotimes\mathcal{A}$. We define $\intcommP{\mathcal{E},\mathcal{A}}$ to be the subspace of $\intcomm{\mathcal{E},\mathcal{A}}$ consisting of those functionals satisfying \eqref{eqn_interactionconstraints} and $\intcommI{\mathcal{E},\mathcal{A}}$ to be the subspace consisting of all families of interactions.
\end{defi}

The notion of Feynman amplitude may be formulated in terms of the preceding structures as follows. If $W$ is any finite set of cardinality $n$ then there is a canonical map
\begin{equation} \label{eqn_attachtovertexsym}
\xymatrix{
\bigl[(\mathcal{E}^{\dag})^{\cotimes n}\bigr]_{\sg{n}}\cotimes\mathcal{A} \ar[rr]^{\sum_{\varsigma\in\sg{n}}\varsigma} && \bigl[(\mathcal{E}^{\dag})^{\cotimes n}\bigr]^{\sg{n}}\cotimes\mathcal{A} \ar[r] & \mathcal{E}^{\dag}(\!(W)\!)\cotimes\mathcal{A},
}
\end{equation}
where we have used the left-hand map of \eqref{eqn_attachmapssym}. Also note that for any finite set $W$ of cardinality $n$ there is a map
\begin{equation} \label{eqn_multiplymap}
\mathcal{A}(\!(W)\!)\to\bigl[\mathcal{A}^{\cotimes n}\bigr]_{\sg{n}}\to\mathcal{A},
\end{equation}
where we have used the right-hand map of \eqref{eqn_attachmapssym} as the first arrow and the commutative multiplication $\mu_{\mathcal{A}}$ of $\mathcal{A}$ as the last.

\begin{defi}
Given a family of interactions $I\in\intcommI{\mathcal{E},\mathcal{A}}$ for a family of free theories $\mathcal{E}$ over $\mathcal{A}$ and a stable graph $\Gamma$, attach the interaction $I$ to each vertex $v\in V(\Gamma)$ by considering the image of $I_{l(v)|v|}$ under the map \eqref{eqn_attachtovertexsym} to define tensors
\[ I(v)\in\mathcal{E}^{\dag}(\!(v)\!)\cotimes\mathcal{A}, \quad v\in V(\Gamma). \]
Now define
\begin{displaymath}
\begin{split}
I(\Gamma) := \underset{v\in V(\Gamma)}{\otimes} I(v) \in \underset{v\in V(\Gamma)}{\widehat{\bigotimes}}\bigl(\mathcal{E}^{\dag}(\!(v)\!)\cotimes\mathcal{A}\bigr) &= \mathcal{E}^{\dag}\bigl(\!\bigl(H(\Gamma)\bigr)\!\bigr)\cotimes\mathcal{A}\bigl(\!\bigl(V(\Gamma)\bigr)\!\bigr) \\
&= \Hom_{\gf}\bigl(\mathcal{E}\bigl(\!\bigl(H(\Gamma)\bigr)\!\bigr),\mathcal{A}\bigl(\!\bigl(V(\Gamma)\bigr)\!\bigr)\bigr).
\end{split}
\end{displaymath}

Likewise, if $P\in\mathcal{E}\cotimes\mathcal{E}\cotimes\mathcal{A}$ is a family of propagators for $\mathcal{E}$ then attach $P$ to every edge $e\in E(\Gamma)$ using the left-hand map of \eqref{eqn_attachmapssym} to define tensors
\[ P(e)\in\mathcal{E}(\!(e)\!)\cotimes\mathcal{A}, \quad e\in E(\Gamma). \]
Now define
\begin{equation} \label{eqn_graphpropagator}
P(\Gamma) := \underset{e\in E(\Gamma)}{\otimes} P(e) \in \underset{e\in E(\Gamma)}{\widehat{\bigotimes}}\bigl(\mathcal{E}(\!(e)\!)\cotimes\mathcal{A}\bigr) = \mathcal{E}\Bigl(\!\Bigl(\bigcup_{e\in E(\Gamma)}e\Bigr)\!\Bigr)\cotimes\mathcal{A}\bigl(\!\bigl(E(\Gamma)\bigr)\!\bigr).
\end{equation}

The \emph{Feynman amplitude}
\[ F_{\Gamma}(I,P)\in\Hom_{\gf}\bigl(\mathcal{E}\bigl(\!\bigl(L(\Gamma)\bigr)\!\bigr),\mathcal{A}\bigr) = \mathcal{E}^{\dag}\bigl(\!\bigl(L(\Gamma)\bigr)\!\bigr)\cotimes\mathcal{A} \]
is defined by evaluating $I(\Gamma)$ on $P(\Gamma)$. More precisely it is the composite
\begin{equation} \label{eqn_Feynmanamplitude}
\xymatrix{
\mathcal{E}\bigl(\!\bigl(L(\Gamma)\bigr)\!\bigr) \ar[rr]^-{x\mapsto x\otimes P(\Gamma)} \ar[d]_{F_{\Gamma}(I,P)} &&
\mathcal{E}\bigl(\!\bigl(L(\Gamma)\bigr)\!\bigr)\cotimes\mathcal{E}\Bigl(\!\Bigl(\underset{e\in E(\Gamma)}{\bigcup}e\Bigr)\!\Bigr)\cotimes\mathcal{A}\bigl(\!\bigl(E(\Gamma)\bigr)\!\bigr) \ar@{=}[d] \\
\mathcal{A} &
\mathcal{A}\bigl(\!\bigl(V(\Gamma)\bigr)\!\bigr)\cotimes\mathcal{A}\bigl(\!\bigl(E(\Gamma)\bigr)\!\bigr) \ar[l]_-{\mu} &
\mathcal{E}\bigl(\!\bigl(H(\Gamma)\bigr)\!\bigr)\cotimes\mathcal{A}\bigl(\!\bigl(E(\Gamma)\bigr)\!\bigr) \ar[l]_-{I(\Gamma)\cotimes\mathds{1}}
}
\end{equation}
where we first multiply by the propagators $P(\Gamma)$, then apply the interactions $I(\Gamma)$, and finally multiply all copies of $\mathcal{A}$ together using \eqref{eqn_multiplymap}.

The \emph{Feynman weight}
\[ w_{\Gamma}(I,P)\in\bigl[(\mathcal{E}^{\dag})^{\cotimes j}\bigr]_{\sg{j}}\cotimes\mathcal{A}, \]
where $j:=|L(\Gamma)|$ is the number of legs of $\Gamma$, is obtained from the Feynman amplitude $F_{\Gamma}(I,P)$ by applying the right-hand map of \eqref{eqn_attachmapssym}.
\end{defi}

\begin{rem} \label{rem_Feynmanequivariance}
The Feynman amplitude $F_{\Gamma}(I,P)$ is equivariant with respect to $\Gamma$ in the sense that if $\phi:\Gamma\to\Gamma'$ is an isomorphism of stable graphs then
\[ \phi^{\#}(F_{\Gamma}(I,P)) = F_{\Gamma'}(I,P), \]
where $\phi^{\#}$ denotes the induced morphism.
\end{rem}

Our distinction between Feynman amplitudes and weights arises as the latter will be used in the definition of certain functional integrals in the renormalization group flow, while passing from the amplitude to the weight entails some loss of information that we will need to keep track of in certain formulas.

\subsubsection{Stable ribbon graphs and noncommutative geometry}

Here we will recall from \cite{KontAiry} the notion of a stable ribbon graph and use this to give a definition for the Feynman amplitudes and weights for a space of interactions based upon the noncommutative geometry of Kontsevich \cite{KontSympGeom, KontFeyn}.

We begin by introducing the following terminology.

\begin{defi}
A \emph{cyclic decomposition} of a set $X$ is a partition of $X$ into cyclically ordered subsets.
\end{defi}

We note that this is precisely the same thing as a permutation $\varsigma$ on $X$, with the cyclic decomposition arising from the decomposition of $\varsigma$ into disjoint cycles.

\begin{defi}
A \emph{stable ribbon graph} consists of the following data:
\begin{itemize}
\item
A finite set of half-edges $H(\Gamma)$.
\item
A finite set of vertices $V(\Gamma)$.
\item
A map $\rho_{\Gamma}:H(\Gamma)\to V(\Gamma)$, which describes the incident half-edges, as explained before in Definition \ref{def_stablegraph}.
\item
An involution $\kappa_{\Gamma}:H(\Gamma)\to H(\Gamma)$ as before, describing the legs and edges of the graph.
\item
A cyclic decomposition $C(v)$ of the incident half-edges of every vertex $v\in V(\Gamma)$.
\item
Maps $g$ and $b$ that assign to every vertex $v\in V(\Gamma)$, nonnegative integers $g(v)$ and $b(v)$ called the \emph{genus} and \emph{boundary} of $v$ respectively.
\end{itemize}
We define the \emph{loop number} of a vertex $v\in V(\Gamma)$ by
\begin{equation} \label{eqn_vertexribbonloopnumber}
l(v) := 2g(v)+b(v)+|C(v)|-1.
\end{equation}
We require that:
\begin{equation} \label{eqn_valencyribbonconditions}
|C(v)|+b(v)>0 \quad\text{and}\quad 2l(v)+|v|\geq 3, \quad\text{for all }v\in V(\Gamma).
\end{equation}
\end{defi}

\begin{rem} \label{rem_loopnumber}
The nomenclature ``loop number'' has been chosen above as the first Betti number of a surface of genus $g(v)$ possessing $(b(v)+|C(v)|)$ boundary components that is placed at the vertex $v$ will be $l(v)$. We will make use of such a picture later in Section \ref{sec_OTFTflow} when we deal with certain aspects of the renormalization group flow in noncommutative geometry.
\end{rem}

We now need to explain how to contract an edge in a stable ribbon graph. These operations describe the various ways in which a Riemann surface may degenerate as we approach the boundary of the moduli space. For the same reason, they also describe what happens when we join the surfaces described in Remark \ref{rem_loopnumber} by replacing the edge with a thin ribbon, see Section \ref{sec_OTFTaxioms} and Figure \ref{fig_contractedge}. For the sake of brevity, we will just explain the details, but the reader interested in a more in depth discussion may consult \cite{ChLaOTFT, HamCompact, KontAiry, Mondello}.

\begin{defi} \label{def_contractedge}
Let $\Gamma$ be a stable ribbon graph and let $e\in E(\Gamma)$ be an edge. We may form a new stable ribbon graph $\Gamma/e$ by contracting the edge $e$ according to the following rules:
\begin{enumerate}
\item
When the edge $e$ joins distinct vertices $v_1,v_2\in V(\Gamma)$ then this edge is deleted and these two vertices are combined to form a new vertex $v$ in $\Gamma/e$. In forming the cycles of $v$, the cycles of $v_1$ and $v_2$ are unaltered, except that the cycles $c_1\in C(v_1)$ and $c_2\in C(v_2)$ that intersect $e$ are joined to form a new cycle $c$ of $v$, with a naturally defined cyclic ordering. The genus and boundary of $v$ are defined as the sum of those for $v_1$ and $v_2$.

The one exception to the above occurs when both $c_1$ and $c_2$ consist of just a single half-edge, for then the resulting cycle $c$ would be empty. Instead of this, the boundary $b(v)$ is further increased by one.
\item
When the edge $e$ is a loop joining a vertex $v$ to itself, we distinguish two cases:
\begin{enumerate}
\item
When the endpoints of $e$ lie in distinct cycles $c_1,c_2\in C(v)$, then the edge $e$ is deleted and these cycles are joined to form a new cycle $c$ of the vertex $v$, as above. The genus $g(v)$ is increased by one and the boundary $b(v)$ is unaltered. This situation is depicted in Figure \ref{fig_contractedge} in Section \ref{sec_OTFTaxioms}.

Again, the single exception to the above rule occurs when both $c_1$ and $c_2$ consist of a single half-edge. In this case we also increase $b(v)$ by one.
\item
Finally, when the endpoints of $e$ both lie in the same cycle $c\in C(v)$ then the edge $e$ is deleted and the cycle $c$ splits into two cycles $c_1$ and $c_2$ with naturally defined cyclic orderings. The genus $g(v)$ and boundary $b(v)$ are unaltered.

An exception occurs to the above rule when the endpoints of $e$ lie next to each other in the cyclic ordering on $c$, for then one of the cycles $c_1$ or $c_2$ would be empty. Instead of allowing an empty cycle, the boundary $b(v)$ is increased by one; except when $c$ contains no other half-edges except those from $e$, for then both $c_1$ and $c_2$ would be empty. In the latter case the boundary $b(v)$ is increased by two.
\end{enumerate}
\end{enumerate}
\end{defi}

The astute reader will notice that as a result of the various rules above, the loop numbers of vertices will increase by one when we contract a loop and add when we contract an edge that is not a loop. One may check quite simply that when contracting multiple edges in a stable ribbon graph, it makes no difference in which order we contract them.

Denote the collection of all cycles of a stable ribbon graph $\Gamma$ by
\[ C(\Gamma):=\bigcup_{v\in V(\Gamma)} C(v). \]
To define the number of boundary components of $\Gamma$, we assign to $\Gamma$ a permutation $c_{\Gamma}$ of its half-edges. This is defined by combining, from every vertex $v$ of $\Gamma$, the permutations that correspond to their cyclic decomposition $C(v)$. Now consider the permutation $\beta_{\Gamma}:=c_{\Gamma}\kappa_{\Gamma}$ and count the number of cycles (including $1$-cycles) in its cyclic decomposition, denoting the result by $\beta_{\Gamma}^{\#}$.

\begin{defi}
Given a connected stable ribbon graph $\Gamma$, we define the genus and number of boundary components of $\Gamma$ by
\begin{displaymath}
\begin{split}
g(\Gamma) &:= 1 - |V(\Gamma)| + \frac{1}{2}\bigl(|E(\Gamma)|+|C(\Gamma)|-\beta_{\Gamma}^{\#}\bigr) + \sum_{v\in V(\Gamma)} g(v), \\
B(\Gamma) &:= \beta_{\Gamma}^{\#} + \sum_{v\in V(\Gamma)} b(v).
\end{split}
\end{displaymath}
We will also need to keep track of
\[ b(\Gamma):=B(\Gamma) - |C(\Gamma/\Gamma)|, \]
where $\Gamma/\Gamma$ denotes the graph $\Gamma$ with all of its edges contracted.
\end{defi}

The terminology above is chosen because there is a canonical construction (cf. Section \ref{sec_OTFTflow}) that associates a surface to any stable ribbon graph $\Gamma$ by placing surfaces of genus $g(v)$ possessing $(b(v)+|C(v)|)$ boundary components at every vertex $v$ of $\Gamma$ and attaching ribbons to them. The numbers $g(\Gamma)$ and $B(\Gamma)$ record the genus and boundary components of that surface.

\begin{rem} \label{rem_genbdryedgecontract}
The reader will notice that the above numbers do not change if we contract an edge in the graph. In particular, if $v_{\Gamma/\Gamma}$ denotes the unique vertex of $\Gamma/\Gamma$ then
\begin{displaymath}
\begin{split}
b(\Gamma) &= b(\Gamma/\Gamma) = b(v_{\Gamma/\Gamma})\geq 0, \\
g(\Gamma) &= g(\Gamma/\Gamma) = g(v_{\Gamma/\Gamma})\geq 0.
\end{split}
\end{displaymath}
\end{rem}

One consequence of the above observation is that we have the inequalities
\begin{equation} \label{eqn_genbdryinequalities}
1 - |V(\Gamma)| + \frac{1}{2}\bigl(|E(\Gamma)|+|C(\Gamma)|-\beta_{\Gamma}^{\#}\bigr) \geq 0, \quad\text{and}\quad \beta_{\Gamma}^{\#} \geq |C(\Gamma/\Gamma)|,
\end{equation}
which are obtained by setting the genus and boundary function at every vertex to zero. In particular, it follows that
\begin{equation} \label{eqn_vertexinequalities}
g(v)\leq g(\Gamma) \quad\text{and}\quad b(v)\leq b(\Gamma), \quad\text{for all } v\in V(\Gamma).
\end{equation}

\begin{defi}
Given a connected stable ribbon graph $\Gamma$, we define the \emph{loop number} of $\Gamma$ by
\[ \ell(\Gamma) := 2g(\Gamma) + B(\Gamma) - 1. \]
\end{defi}

A straightforward calculation shows that this is given by
\begin{equation} \label{eqn_loopnumber}
\ell(\Gamma) = 1-\chi(\Gamma) + \sum_{v\in V(\Gamma)} l(v) = \mathbf{b}_1(\Gamma) + \sum_{v\in V(\Gamma)} l(v),
\end{equation}
where $\chi(\Gamma)$ is the Euler characteristic of $\Gamma$ and $\mathbf{b}_1(\Gamma)$ is the first Betti number of $\Gamma$. Just as before, the inequalities \eqref{eqn_graphinequalities} hold for any connected stable ribbon graph.

The interactions for our stable ribbon graphs will be taken from a space based on the noncommutative symplectic geometry of Kontsevich \cite{KontSympGeom}, which we now recall.

\begin{defi}
Given a free field theory $\mathcal{E}$, define
\[ \KontHam{\mathcal{E}}:=\prod_{n=0}^{\infty}\bigl[(\mathcal{E}^{\dag})^{\cotimes n}\bigr]_{\cyc{n}} \quad\text{and}\quad \KontHamP{\mathcal{E}}:=\prod_{n=1}^{\infty}\bigl[(\mathcal{E}^{\dag})^{\cotimes n}\bigr]_{\cyc{n}}, \]
where here we have taken the \emph{cyclic} coinvariants.

Now define
\[ \intnuP{\mathcal{E}} := \csalgP{\KontHam{\mathcal{E}}} \]
and the space in which our interactions will live by
\[ \intnoncomm{\mathcal{E}} := \intnuP{\mathcal{E}}[[\gamma]] = \gf[[\gamma]]\cotimes\intnuP{\mathcal{E}}. \]
\end{defi}

To explain the choice of notation above; note that if we denote the generator of the ground field $\gf$ by $\nu$ then
\begin{displaymath}
\begin{split}
\intnuP{\mathcal{E}} \subset \csalg{\KontHam{\mathcal{E}}} &= \csalg{\gf\oplus\KontHamP{\mathcal{E}}} = \csalg{\gf}\cotimes\csalg{\KontHamP{\mathcal{E}}} \\
&= \gf[[\nu]]\cotimes\csalg{\KontHamP{\mathcal{E}}}
\end{split}
\end{displaymath}
and so
\[ \intnoncomm{\mathcal{E}} \subset \gf[[\gamma,\nu]]\cotimes\csalg{\KontHamP{\mathcal{E}}}. \]

In this way we may write any $I\in\intnoncomm{\mathcal{E}}$ as a series
\[ I = \sum_{i,j,k=0}^{\infty} \gamma^i\nu^j I_{ijk}, \quad\text{where }I_{ijk}\in\bigl[\KontHamP{\mathcal{E}}^{\cotimes k}\bigr]_{\sg{k}}. \]
Note that from the definition of $\intnuP{\mathcal{E}}$ we must have
\[ I_{i00}=0, \quad\text{for all }i\geq 0. \]

We may also refine this further by noting that $\KontHamP{\mathcal{E}}$ is graded by the order of the tensor powers of $\mathcal{E}^{\dag}$ and that this induces a corresponding grading on $\csalg{\KontHamP{\mathcal{E}}}$. We may then write
\[ I_{ijk} = \sum_{l=0}^{\infty} I_{ijkl}, \]
where~$I_{ijkl}\in\bigl[\KontHamP{\mathcal{E}}^{\cotimes k}\bigr]_{\sg{k}}$ is a homogenous tensor of order $l$ in $\mathcal{E}^{\dag}$.

\begin{defi}
Given a free field theory $\mathcal{E}$, we say that $I\in\intnoncomm{\mathcal{E}}$ is an \emph{interaction} if it has degree zero and if
\begin{equation} \label{eqn_interactionNCconstraints}
I_{ijkl} = 0, \quad\text{whenever } 2(2i+j+k-1)+l<3.
\end{equation}
Denote the subspace of $\intnoncomm{\mathcal{E}}$ consisting of those $I$ satisfying \eqref{eqn_interactionNCconstraints} by $\intnoncommP{\mathcal{E}}$ and the subspace consisting of all interactions by $\intnoncommI{\mathcal{E}}$.

More generally, if $\mathcal{E}$ is a family of free theories over $\mathcal{A}$ then we will say that
\[ I\in\intnoncomm{\mathcal{E},\mathcal{A}}:=\intnoncomm{\mathcal{E}}\cotimes\mathcal{A} \]
is a \emph{family of interactions} if it has total degree zero and if it satisfies the constraints given by \eqref{eqn_interactionNCconstraints}, where now~$I_{ijkl}\in\bigl[\KontHamP{\mathcal{E}}^{\cotimes k}\bigr]_{\sg{k}}\cotimes\mathcal{A}$ and has homogeneous order $l$ in $\mathcal{E}^{\dag}$. Denote the subspace of $\intnoncomm{\mathcal{E},\mathcal{A}}$ consisting of those $I$ satisfying \eqref{eqn_interactionNCconstraints} by $\intnoncommP{\mathcal{E},\mathcal{A}}$ and the subspace consisting of all families of interactions by $\intnoncommI{\mathcal{E},\mathcal{A}}$.
\end{defi}

We are now ready to provide the definition of the Feynman amplitude for these interactions. The cyclic structures on the stable ribbon graphs are precisely what is needed for us to be able to attach our interactions to them. First note that if $A$ is a cyclically ordered set of cardinality $n$ and $\mathcal{V}$ is a $\mathbb{Z}$-graded nuclear space then \eqref{eqn_labelmaps} provides canonical maps
\[ \bigl[\mathcal{V}^{\cotimes n}\bigr]^{\cyc{n}}\to\mathcal{V}(\!(A)\!)\to\bigl[\mathcal{V}^{\cotimes n}\bigr]_{\cyc{n}}. \]
Hence we have maps
\begin{equation} \label{eqn_attachmapscyc}
\begin{split}
& \xymatrix{
\KontHam{\mathcal{E}} \ar[r] & \bigl[(\mathcal{E}^{\dag})^{\cotimes n}\bigr]_{\cyc{n}} \ar[r]^{\cycsum{n}} & \bigl[(\mathcal{E}^{\dag})^{\cotimes n}\bigr]^{\cyc{n}} \ar[r] & \mathcal{E}^{\dag}(\!(A)\!)
}, \\
& \xymatrix{
\mathcal{E}^{\dag}(\!(A)\!) \ar[r] & \bigl[(\mathcal{E}^{\dag})^{\cotimes n}\bigr]_{\cyc{n}} \ar[r] & \KontHam{\mathcal{E}}
};
\end{split}
\end{equation}
where we have used the projections from and inclusions into $\KontHam{\mathcal{E}}$, along with the cyclic symmetrization map
\[ \cycsum{n}(x) := \sum_{\varsigma\in\cyc{n}}\varsigma\cdot x. \]

Now if $C$ is a cyclic decomposition of a set $X$ consisting of $k$ cycles then we have maps
\begin{align}
\label{eqn_attachmapscycdecompon} & \xymatrix{
\bigl[\KontHamP{\mathcal{E}}^{\cotimes k}\bigr]_{\sg{k}} \ar[rr]^{\sum_{\varsigma\in\sg{k}}\varsigma} \ar[rrrd] && \bigl[\KontHamP{\mathcal{E}}^{\cotimes k}\bigr]^{\sg{k}} \ar[r] & \KontHamP{\mathcal{E}}(\!(C)\!) = \underset{c\in C}{\bigotimes}\KontHamP{\mathcal{E}} \ar[d] \\
&&& \underset{c\in C}{\bigotimes}\mathcal{E}^{\dag}(\!(c)\!) = \mathcal{E}^{\dag}(\!(X)\!)
} \\
\label{eqn_attachmapscycdecompoff} & \xymatrix{
\mathcal{E}^{\dag}(\!(X)\!) = \underset{c\in C}{\bigotimes}\mathcal{E}^{\dag}(\!(c)\!) \ar[r] & \underset{c\in C}{\bigotimes}\KontHamP{\mathcal{E}} = \KontHamP{\mathcal{E}}(\!(C)\!) \ar[r] & \bigl[\KontHamP{\mathcal{E}}^{\cotimes k}\bigr]_{\sg{k}}
}
\end{align}
where we have used the maps \eqref{eqn_attachmapssym} and \eqref{eqn_attachmapscyc}. The following example may help to clarify some of the above details.

\begin{example} \label{exm_cyclicwordproduct}
Suppose we are given nonnegative integers $k$ and $l$ and a list $r_1,\ldots,r_k$ of $k$ positive integers whose sum is $l$. Denote by $C_{\mathbf{r}}$ the canonical cyclic decomposition, consisting of $k$ cycles, on the set of integers between $1$ and $l$ defined by
\begin{equation} \label{eqn_cycdecpartition}
(1,\ldots,r_1)(r_1+1,\ldots,r_1+r_2)\ldots(l-r_k+1,\ldots,l).
\end{equation}
Suppose that we are given a tensor
\[ x = x_{11}\cdots x_{1r_1}x_{21}\cdots x_{2r_2}\cdots x_{k1}\cdots x_{kr_k} \in (\mathcal{E}^{\dag})^{\cotimes l} = \mathcal{E}^{\dag}(\!(\{1,\ldots,l\})\!). \]
Then the image of $x$ under the map \eqref{eqn_attachmapscycdecompoff} defined by the cyclic decomposition $C_{\mathbf{r}}$ is the product of cyclic words
\[ (x_{11}\cdots x_{1r_1})(x_{21}\cdots x_{2r_2})\cdots (x_{k1}\cdots x_{kr_k}). \]

Now given any family of interactions $I\in\intnoncomm{\mathcal{E},\mathcal{A}}$, we may write
\begin{equation} \label{eqn_intcycrep}
I_{ijkl} = \sum_{\begin{subarray}{c} r_1,\ldots,r_k\geq 1: \\ r_1+\cdots+r_k=l \end{subarray}}\lceil I_{ijkl}^{\mathbf{r}}\rceil, \quad\text{for some } I_{ijkl}^\mathbf{r}\in(\mathcal{E}^{\dag})^{\cotimes l}\cotimes\mathcal{A};
\end{equation}
where $\lceil I_{ijkl}^{\mathbf{r}}\rceil$ denotes the image of $I_{ijkl}^{\mathbf{r}}$ under the map \eqref{eqn_attachmapscycdecompoff} defined by the cyclic decomposition $C_{\mathbf{r}}$. Note that the choice of the $I_{ijkl}^{\mathbf{r}}$ are not unique.
\end{example}

\begin{defi} \label{def_Feynmanampnoncomm}
Let $I\in\intnoncommI{\mathcal{E},\mathcal{A}}$ be a family of interactions for a family $\mathcal{E}$ of free field theories over $\mathcal{A}$. Given a stable ribbon graph $\Gamma$, attach the interaction $I$ to every vertex $v$ of $\Gamma$ using the cyclic decomposition $C(v)$ of that vertex; that is, we consider the image of $I_{g(v)b(v)|C(v)||v|}$ under the map \eqref{eqn_attachmapscycdecompon}. This defines a collection of tensors
\[ I(v)\in\mathcal{E}^{\dag}(\!(v)\!)\cotimes\mathcal{A}, \quad v\in V(\Gamma). \]
Now define
\[ I(\Gamma) := \underset{v\in V(\Gamma)}{\otimes} I(v) \in \Hom_{\gf}\bigl(\mathcal{E}\bigl(\!\bigl(H(\Gamma)\bigr)\!\bigr),\mathcal{A}\bigl(\!\bigl(V(\Gamma)\bigr)\!\bigr)\bigr) \]
as before.

Now if $P\in\mathcal{E}\cotimes\mathcal{E}\cotimes\mathcal{A}$ is a family of propagators for $\mathcal{E}$ then we may attach $P$ to every edge of $\Gamma$ to form
\[ P(\Gamma) \in \mathcal{E}\Bigl(\!\Bigl(\bigcup_{e\in E(\Gamma)}e\Bigr)\!\Bigr)\cotimes\mathcal{A}\bigl(\!\bigl(E(\Gamma)\bigr)\!\bigr), \]
see \eqref{eqn_graphpropagator}. The Feynman amplitude
\[ F_{\Gamma}(I,P) \in \mathcal{E}^{\dag}\bigl(\!\bigl(L(\Gamma)\bigr)\!\bigr)\cotimes\mathcal{A} \]
is then defined by \eqref{eqn_Feynmanamplitude}, in the standard way.

To define the Feynman weight, we require a cyclic decomposition of the legs of $\Gamma$. First note that if $\Gamma$ is a corolla---that is a graph with just a single vertex and no edges, but possibly some legs---then there is a canonical cyclic decomposition of the legs of $\Gamma$ defined by the cyclic decomposition of the lone vertex. More generally, if $\Gamma$ is connected then the graph obtained from $\Gamma$ by contracting all of its edges will be a corolla with the same set of legs as $\Gamma$. Hence this will provide a cyclic decomposition of the legs $L(\Gamma)$, which we will call the \emph{canonical decomposition}.

Using the canonical cyclic decomposition of $L(\Gamma)$, we define the Feynman weight
\[ w_{\Gamma}(I,P)\in\csalg{\KontHamP{\mathcal{E}}}\cotimes\mathcal{A} \]
of a connected stable ribbon graph $\Gamma$ to be the image of the Feynman amplitude $F_{\Gamma}(I,P)$ under the map \eqref{eqn_attachmapscycdecompoff} determined by this cyclic decomposition.
\end{defi}

\begin{rem} \label{rem_Feynmanequivarianceribbon}
The Feynman amplitude $F_{\Gamma}(I,P)$ is equivariant with respect to isomorphisms of stable ribbon graphs in the sense explained by Remark \ref{rem_Feynmanequivariance}. For this reason, $w_{\Gamma}(I,P)$ only depends upon the isomorphism class of $\Gamma$.
\end{rem}

\subsubsection{Subgraphs and their amplitudes}

Many basic facts about the renormalization group flow and other calculations in this paper and in \cite{NCRBV} will come down to some simple facts about subgraphs and their associated Feynman amplitudes.

\begin{defi}
Let $\Gamma$ be a stable ribbon graph. A \emph{subgraph} of $\Gamma$ is a subset $\beta$ of the set of edges $E(\Gamma)$. We will denote by $\Gamma/\beta$ the result of contracting all the edges of the subgraph $\beta$ inside $\Gamma$.
\end{defi}

To any subgraph $\beta$ of $\Gamma$ we can associate an actual stable ribbon graph $\Gamma[\beta]$ in an obvious way. Formally:
\begin{itemize}
\item
The vertices consist of all those vertices of $\Gamma$ touched by the subgraph,
\[ V(\Gamma[\beta]):=\rho_{\Gamma}\Bigl(\bigcup_{e\in\beta}e\Bigr). \]
\item
The half-edges are then taken from these vertices,
\[ H(\Gamma[\beta]) := \rho_{\Gamma}^{-1}(V(\Gamma[\beta])). \]
\item
The vertices inherit their structure from $\Gamma$. That is the map $\rho_{\Gamma[\beta]}$ and the genus and boundary maps for $\Gamma[\beta]$ are just the restrictions of those for $\Gamma$. The vertices of $\Gamma[\beta]$ inherit their cyclic decompositions from $\Gamma$.
\item
The edges are $E(\Gamma[\beta]):=\beta$.
\end{itemize}

The opposite operation to contracting a subgraph is inserting a graph, which we will now define.

\begin{defi}
Let $\Gamma$ and $G$ be two stable ribbon graphs. An \emph{insertion} of $G$ into $\Gamma$ consists of the following data:
\begin{itemize}
\item
An injective map $\iota$ that assigns to every connected component $\mathscr{G}$ of $G$, a vertex $\iota(\mathscr{G})$ of $\Gamma$ satisfying
\[ g(\iota(\mathscr{G})) = g(\mathscr{G}) \quad\text{and}\quad b(\iota(\mathscr{G})) = b(\mathscr{G}). \]
\item
A map, also denoted by $\iota$, that assigns to every connected component $\mathscr{G}$ of $G$, a bijection $\iota_{\mathscr{G}}$ from the legs $L(\mathscr{G})$ to the incident half-edges of the vertex $\iota(\mathscr{G})$. Each bijection $\iota_{\mathscr{G}}$ must respect the canonical cyclic decomposition of the legs $L(\mathscr{G})$ and the cyclic decomposition at the vertex $\iota(\mathscr{G})$.
\end{itemize}
\end{defi}

Given an insertion $\iota$ of $G$ into $\Gamma$, we will denote by $\Gamma\circ_{\iota}G$ the stable ribbon graph formed by inserting $G$ into $\Gamma$ along $\iota$. Formally; if $\mathscr{G}_{\alpha}$,~$\alpha\in\mathcal{I}$ are the connected components of $G$ then:
\begin{itemize}
\item
The half-edges $H(\Gamma\circ_{\iota}G) := H(\Gamma)\sqcup(H(G)-L(G))$ are obtained by removing the legs of $G$.
\item
The edges are $E(\Gamma\circ_{\iota}G) := E(\Gamma)\sqcup E(G)$.
\item
The legs are $L(\Gamma\circ_{\iota}G) := L(\Gamma)$.
\item
The vertices chosen by $\iota$ are then removed,
\[ V(\Gamma\circ_{\iota}G) := \bigl(V(\Gamma)-\{\iota(\mathscr{G}_{\alpha}),\alpha\in\mathcal{I}\}\bigr)\sqcup V(G). \]
\item
The half-edges of $\Gamma$ that were incident to the removed vertices are then reassigned to the vertices of $G$ through its legs using the map~$\iota$,
\[ \rho_{\Gamma\circ_{\iota}G}(h):=\left\{\begin{array}{ll} \rho_G(h), & h\in H(G)-L(G) \\ \rho_{\Gamma}(h), & h\in H(\Gamma)-\rho_{\Gamma}^{-1}\{\iota(\mathscr{G}_{\alpha}),\alpha\in\mathcal{I}\} \\ \rho_G\bigl(\iota^{-1}_{\iota^{-1}(\rho_{\Gamma}(h))}(h)\bigr), & h\in\rho_{\Gamma}^{-1}\{\iota(\mathscr{G}_{\alpha}),\alpha\in\mathcal{I}\} \end{array}\right\}. \]
\item
The genus, boundary and cyclic decomposition at each vertex of $\Gamma\circ_{\iota}G$ is inherited directly from those of $\Gamma$ and $G$:
\begin{itemize}
\item
If $v\in V(\Gamma)$ and $v\notin\{\iota(\mathscr{G}_{\alpha}),\alpha\in\mathcal{I}\}$ then $\rho_{\Gamma\circ_{\iota}G}^{-1}(v) = \rho_{\Gamma}^{-1}(v)$ and the cyclic decomposition of $v$ is inherited directly from $\Gamma$.
\item
If $v\in V(G)$ then the insertion $\iota$ provides a bijection between $\rho_{\Gamma\circ_{\iota}G}^{-1}(v)$ and $\rho_G^{-1}(v)$ and the cyclic decomposition of $v$ is transferred from $G$ using this.
\end{itemize}
\end{itemize}

\begin{lemma} \label{lem_contractsubgraph}
Let $\Gamma$ be a connected stable ribbon graph and let $\beta$ be a subgraph of $\Gamma$. Then for every connected component $\mathscr{G}$ of $\Gamma[\beta]$, there is a unique vertex $v_{\mathscr{G}}\in V(\Gamma/\beta)$ such that the incident half-edges of $v_{\mathscr{G}}$ are the legs of $\mathscr{G}$,
\begin{equation} \label{eqn_legstovertices}
\rho_{\Gamma/\beta}^{-1}(v_{\mathscr{G}}) = L(\mathscr{G}).
\end{equation}

Additionally, for every component $\mathscr{G}$, the cyclic decomposition $C(v_{\mathscr{G}})$ at the vertex $v_{\mathscr{G}}$ agrees with the canonical cyclic decomposition of the legs of $\mathscr{G}$ under \eqref{eqn_legstovertices}, and the genus and boundary correspond too;
\[ g_{\Gamma/\beta}(v_{\mathscr{G}}) = g(\mathscr{G}) \quad\text{and}\quad b_{\Gamma/\beta}(v_{\mathscr{G}}) = b(\mathscr{G}). \]

Furthermore, if we denote by $V_{\beta}\subset V(\Gamma/\beta)$ the set consisting of all those vertices $v_{\mathscr{G}}$, where $\mathscr{G}$ runs over the connected components of $\Gamma[\beta]$, then
\begin{equation} \label{eqn_stableverticescontract}
V(\Gamma/\beta) - V_{\beta} = V(\Gamma) - V(\Gamma[\beta]).
\end{equation}
Under \eqref{eqn_stableverticescontract} we have that
\begin{displaymath}
\begin{split}
g_{\Gamma/\beta}(v) &= g_{\Gamma}(v), \quad b_{\Gamma/\beta}(v) = b_{\Gamma}(v) \quad\text{and} \\
\rho_{\Gamma/\beta}^{-1}(v) &= \rho_{\Gamma}^{-1}(v), \quad\text{for all }v\in V(\Gamma/\beta)-V_{\beta}
\end{split}
\end{displaymath}
and that the cyclic decompositions in the last equality coming from $\Gamma/\beta$ and $\Gamma$ coincide.
\end{lemma}

\begin{figure}[htp]
\centering
\includegraphics{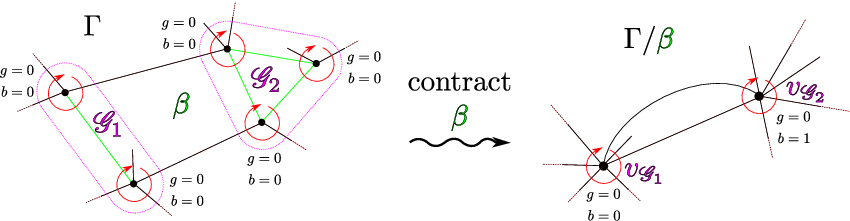}
\caption{The correspondence between the connected components of a subgraph and the vertices when the subgraph is collapsed.}
\label{fig_graphcollapse}
\end{figure}

\begin{proof}[Proof of Lemma \ref{lem_contractsubgraph}]
The proof is by induction on the number of edges in the subgraph~$\beta$ and follows tautologically from the definitions and the observations made in Remark \ref{rem_genbdryedgecontract}. A simple example is described by Figure \ref{fig_graphcollapse}.
\end{proof}

Lemma \ref{lem_contractsubgraph} and Equation \eqref{eqn_legstovertices} define a canonical insertion
\[ \iota_{\Gamma,\beta}:\mathscr{G}\mapsto v_{\mathscr{G}} \]
of $\Gamma[\beta]$ into $\Gamma/\beta$, which is the identity on the legs of $\mathscr{G}$.

The operations of insertion and contraction are inverse to one another, as we will now explain. We consider two categories:

\begin{defi} \label{def_catpairs}
The objects of our first category consist of pairs~$(\Gamma,\beta)$ where~$\Gamma$ is a connected stable ribbon graph and~$\beta$ is a subgraph of~$\Gamma$. A morphism from~$(\Gamma,\beta)$ to~$(\Gamma',\beta')$ in this category is an isomorphism $\phi:\Gamma\to\Gamma'$ of stable ribbon graphs satisfying~$\phi(\beta) = \beta'$.
\end{defi}

\begin{defi} \label{def_cattriples}
The objects of our second category are triples~$(\Gamma,\iota, G)$ in which both~$\Gamma$ and~$G$ are stable ribbon graphs, $\Gamma$~is connected, no connected component of~$G$ is a corolla and~$\iota$ is an insertion of~$G$ into~$\Gamma$. A morphism from $(\Gamma,\iota, G)$ to $(\Gamma',\iota', G')$ in this category consists of a pair of isomorphisms $\phi:\Gamma\to\Gamma'$ and $\psi:G\to G'$ of stable ribbon graphs that respect the insertions~$\iota$ and~$\iota'$ in the obvious sense.
\end{defi}

\begin{theorem} \label{thm_catequiv}
The categories defined by Definition \ref{def_catpairs} and Definition \ref{def_cattriples} are equivalent under the functors of contraction and insertion:
\begin{displaymath}
\begin{array}{ccc}
(\Gamma,\beta) & \rightharpoonup & (\Gamma/\beta,\iota_{\Gamma,\beta},\Gamma[\beta]) \\
(\Gamma\circ_{\iota} G,E(G)) & \leftharpoondown & (\Gamma,\iota,G)
\end{array}
\end{displaymath}
\end{theorem}

\begin{proof}
Applying Lemma \ref{lem_contractsubgraph}, it follows that the top arrow is a right-inverse to the bottom arrow. Likewise, in the opposite direction a natural transformation to the identity functor is constructed from the insertion $\iota$ using Lemma \ref{lem_contractsubgraph}.
\end{proof}

The preceding theorem is very simple, but is the basis for important calculations in the sections that follow and the sequel \cite{NCRBV}. Looking ahead in this regard, we require a slight modification of the notion of Feynman amplitude.

\begin{defi} \label{def_Feynmanampnoncommsubgraph}
Let $\mathcal{E}$ be a family of free field theories over $\mathcal{A}$ and let $\mathcal{F}$ be a mapping that assigns to every stable ribbon graph $\Gamma$, a functional
\[ \mathcal{F}(\Gamma) \in \mathcal{E}^{\dag}\bigl(\!\bigl(L(\Gamma)\bigr)\!\bigr)\cotimes\mathcal{A} \]
that is equivariant with respect to $\Gamma$ in the sense of Remark \ref{rem_Feynmanequivarianceribbon}.

Given a stable ribbon graph $\Gamma$ and a subgraph $\beta$ of $\Gamma$, we define a Feynman amplitude
\[ F_{\Gamma,\beta;\mathcal{F}}(I,P) \in \mathcal{E}^{\dag}\bigl(\!\bigl(L(\Gamma)\bigr)\!\bigr)\cotimes\mathcal{A} \]
for every family of interactions $I\in\intnoncommI{\mathcal{E},\mathcal{A}}$ and propagators $P$ by replacing the rules for the subgraph $\beta$ by $\mathcal{F}$.

More precisely, consider
\[ I(\Gamma,\beta) := \underset{v\in V(\Gamma/\beta) - V_{\beta}}{\otimes} I(v) \in \mathcal{E}^{\dag}\Bigl(\!\Bigl(\bigcup_{v\in V(\Gamma/\beta) - V_{\beta}}v\Bigr)\!\Bigr)\cotimes\mathcal{A}\bigl(\!\bigl(V(\Gamma/\beta) - V_{\beta}\bigr)\!\bigr) \]
where $V_{\beta}$ was defined in Lemma \ref{lem_contractsubgraph}. Then we have
\[ I(\Gamma,\beta)\otimes\mathcal{F}(\Gamma[\beta]) \in \mathcal{E}^{\dag}\bigl(\!\bigl(H(\Gamma/\beta)\bigr)\!\bigr)\cotimes\mathcal{A}\bigl(\!\bigl(V(\Gamma/\beta) - V_{\beta}\bigr)\!\bigr)\cotimes\mathcal{A}, \]
where we have used Lemma \ref{lem_contractsubgraph}. Applying this as before to $P(\Gamma/\beta)$ and using the commutative multiplication $\mu$ on $\mathcal{A}$ yields the Feynman amplitude $F_{\Gamma,\beta;\mathcal{F}}(I,P)$. Using the canonical cyclic decomposition of the legs $L(\Gamma)$, we may consider the image of this Feynman amplitude under the map \eqref{eqn_attachmapscycdecompoff}, which we denote by
\[ w_{\Gamma,\beta;\mathcal{F}}(I,P) \in \csalg{\KontHamP{\mathcal{E}}}\cotimes\mathcal{A}. \]
\end{defi}

\begin{rem}
The Feynman amplitude $F_{\Gamma,\beta;\mathcal{F}}(I,P)$ is equivariant with respect to morphisms of the pair $(\Gamma,\beta)$, as defined in Definition \ref{def_catpairs}, in the sense described by Remark \ref{rem_Feynmanequivariance}.
\end{rem}

\subsection{The renormalization group flow in noncommutative geometry}

In this section we will provide the definition for the renormalization group flow on the space of interactions $\intnoncomm{\mathcal{E}}$ and prove its most basic properties, including that it lifts the renormalization group flow defined on $\intcomm{\mathcal{E}}$ by \cite{CosEffThy}. We begin by recalling the details of the latter.

\subsubsection{The commutative case}

\begin{defi}
Let $\mathcal{E}$ be a family of free theories over $\mathcal{A}$ and let
\[ I\in\intcommI{\mathcal{E},\mathcal{A}} \quad\text{and}\quad P\in\mathcal{E}\cotimes\mathcal{E}\cotimes\mathcal{A} \]
be families of interactions and propagators for $\mathcal{E}$ respectively. We define the renormalization group flow by
\[ W(I,P) := \sum_{\Gamma} \frac{\hbar^{\ell(\Gamma)}}{|\Aut(\Gamma)|}w_{\Gamma}(I,P) \in \intcommI{\mathcal{E},\mathcal{A}}, \]
where the sum is taken over all (isomorphism classes of) connected stable graphs.
\end{defi}

The inequalities \eqref{eqn_graphinequalities} ensure that the renormalization group flow is well-defined (that is, that the sum converges appropriately) and that it maps interactions to interactions. The flow is continuous, and satisfies the following fundamental identities, which state that the (abelian) group of propagators acts on the space of interactions. They appear in \cite[\S 2.3]{CosEffThy} as Lemma 3.3.1 and a corollary of Lemma 3.4.1.

\begin{theorem}
Given families
\[ I\in\intcommI{\mathcal{E},\mathcal{A}} \quad\text{and}\quad P_1,P_2\in\mathcal{E}\cotimes\mathcal{E}\cotimes\mathcal{A} \]
of interactions and propagators for a family $\mathcal{E}$ of free theories over $\mathcal{A}$ we have:
\begin{align*}
W(I,0) &= I, \\
W(I,P_1+P_2) &= W(W(I,P_1),P_2).
\end{align*}
\end{theorem}

\subsubsection{The noncommutative case} \label{sec_noncommflow}

We will now explain how to define the renormalization group flow for interactions taken from the space $\intnoncomm{\mathcal{E}}$ based on the noncommutative geometry of Kontsevich, and how to prove the same fundamental properties of this flow.

\begin{defi}
Let
\[ I\in\intnoncommI{\mathcal{E},\mathcal{A}} \quad\text{and}\quad P\in\mathcal{E}\cotimes\mathcal{E}\cotimes\mathcal{A} \]
be families of interactions and propagators for a family $\mathcal{E}$ of free theories over $\mathcal{A}$. The renormalization group flow in the noncommutative case is defined by
\begin{equation} \label{eqn_RGFnoncomm}
W(I,P) := \sum_{\Gamma} \frac{\gamma^{g(\Gamma)}\nu^{b(\Gamma)}}{|\Aut(\Gamma)|}w_{\Gamma}(I,P) \in \intnoncommI{\mathcal{E},\mathcal{A}},
\end{equation}
where now the sum is taken over all connected stable \emph{ribbon} graphs.
\end{defi}

The flow $I\mapsto W(I,P)$ is continuous. The same inequalities \eqref{eqn_graphinequalities} prove, as before, that the renormalization group flow takes interactions to interactions, and that the sum \eqref{eqn_RGFnoncomm} is finite in each multidegree. More precisely, we may write
\begin{displaymath}
\begin{split}
W(I,P) &= \sum_{i,j,k=0}^{\infty}\gamma^i\nu^j W_{ijk}(I,P), \quad W_{ijk}(I,P)\in\bigl[\KontHamP{\mathcal{E}}^{\cotimes k}\bigr]_{\sg{k}}\cotimes\mathcal{A}, \\
W_{ijk}(I,P) &= \sum_{l=0}^{\infty}W_{ijkl}(I,P);
\end{split}
\end{displaymath}
where $W_{ijkl}(I,P)$ is the sum \eqref{eqn_RGFnoncomm}, except taken over only those connected stable ribbon graphs~$\Gamma$ satisfying:
\[ g(\Gamma)=i, \quad b(\Gamma)=j, \quad |C(\Gamma/\Gamma)|=k \quad\text{and}\quad |L(\Gamma)|=l; \]
which becomes a finite sum.

Again, the renormalization group flow satisfies the same fundamental identities as before, which are listed in the following theorem.

\begin{theorem} \label{thm_intgrpact}
Let $\mathcal{E}$ be a family of free theories over $\mathcal{A}$ and let
\[ I\in\intnoncommI{\mathcal{E},\mathcal{A}} \quad\text{and}\quad P_1,P_2\in\mathcal{E}\cotimes\mathcal{E}\cotimes\mathcal{A} \]
be families of interactions and propagators for $\mathcal{E}$ respectively, then:
\begin{align}
\label{eqn_intgrpactzero}
W(I,0) &= I, \\
\label{eqn_intgrpactsum}
W(I,P_1+P_2) &= W(W(I,P_1),P_2).
\end{align}
\end{theorem}

To prove Theorem \ref{thm_intgrpact}, we introduce two auxiliary lemmas:

\begin{lemma} \label{lem_cyclicsetdettachattach}
Let $u$ and $v$ be sets with cyclic decompositions and consider the  canonical maps
\[ \mathcal{E}^{\dag}(\!(u)\!)\to\csalg{\KontHamP{\mathcal{E}}} \quad\text{and}\quad \csalg{\KontHamP{\mathcal{E}}}\to\mathcal{E}^{\dag}(\!(v)\!) \]
defined by \eqref{eqn_attachmapscycdecompoff} and \eqref{eqn_attachmapscycdecompon} respectively. Then their composite is the map
\[ \mathcal{E}^{\dag}(\!(u)\!)\to\mathcal{E}^{\dag}(\!(v)\!), \qquad x\mapsto\sum_{\varphi:u\to v}\varphi^{\#}(x); \]
where the sum is taken over all those bijections $\varphi:u\to v$ respecting the cyclic decompositions on $u$ and $v$ (with the answer being zero if there are none).
\end{lemma}

\begin{proof}
The statement ultimately follows tautologically from the definitions.
\end{proof}

\begin{lemma} \label{lem_Feynmansubamplitudes}
Let $\mathcal{E}$ be a family of free theories over $\mathcal{A}$ and let
\[ I\in\intnoncommI{\mathcal{E},\mathcal{A}} \quad\text{and}\quad P_1,P_2\in\mathcal{E}\cotimes\mathcal{E}\cotimes\mathcal{A} \]
be families of interactions and propagators for $\mathcal{E}$ respectively. For every connected stable ribbon graph $\Gamma$,
\[ F_{\Gamma}(I,P_1+P_2) = \sum_{\beta\subset\Gamma} F_{\Gamma,\beta;F_{\ast}(I,P_1)}(I,P_2), \]
where $F_{\ast}(I,P_1)$ denotes the mapping $\Gamma\mapsto F_{\Gamma}(I,P_1)$ and the sum is taken over all subgraphs $\beta$ of $\Gamma$.
\end{lemma}

\begin{proof}
Note that, under the notation adopted in Definition \ref{def_Feynmanampnoncomm} and \ref{def_Feynmanampnoncommsubgraph}, we have:
\begin{displaymath}
\begin{split}
I(\Gamma) &= I(\Gamma,\beta)\otimes I(\Gamma[\beta]), \\
(P_1+P_2)(\Gamma) &= \sum_{\beta\subset\Gamma} P_2(\Gamma/\beta)\otimes P_1(\Gamma[\beta]);
\end{split}
\end{displaymath}
where we have used Lemma \ref{lem_contractsubgraph}. The result follows immediately.
\end{proof}

\begin{proof}[Proof of Theorem \ref{thm_intgrpact}]
We begin by proving Equation \eqref{eqn_intgrpactzero}. Since the propagator is zero, the sum in \eqref{eqn_RGFnoncomm} will only be over all corollas. For a corolla $\Gamma$, the automorphisms consist of all those permutations of the half-edges that preserve the cyclic decomposition at the vertex, and this factor precisely cancels the factor in $w_{\Gamma}(I,P)$ that is picked up by attaching $I$ to $\Gamma$ when we sum over all such permutations, as may be deduced from Lemma \ref{lem_cyclicsetdettachattach}. This proves \eqref{eqn_intgrpactzero}.

Before proceeding, we note that one consequence of the above is the formula
\begin{equation} \label{eqn_Feynmansplit}
W(I,P_1) = I + \sum_G \frac{\gamma^{g(G)}\nu^{b(G)}}{|\Aut(G)|}w_{G}(I,P_1),
\end{equation}
where on the right we now sum over all connected stable ribbon graphs $G$ \emph{that are not corollas}.

Now we prove Equation \eqref{eqn_intgrpactsum}. We start by computing the left-hand side of \eqref{eqn_intgrpactsum}, since this is simple. Using Lemma \ref{lem_Feynmansubamplitudes} we have
\[ W(I,P_1+P_2) = \sum_{\Gamma}\sum_{\beta\subset\Gamma} \frac{\gamma^{g(\Gamma)}\nu^{b(\Gamma)}}{|\Aut(\Gamma)|} w_{\Gamma,\beta;F_{\ast}(I,P_1)}(I,P_2). \]
It follows from the orbit-stabilizer theorem that given a pair $(\Gamma,\beta)$ as in Definition \ref{def_catpairs}, the number of subgraphs $\beta'$ of $\Gamma$ such that $(\Gamma,\beta')$ is isomorphic to $(\Gamma,\beta)$ is
\[ \frac{|\Aut(\Gamma)|}{|\Aut(\Gamma,\beta)|}. \]
Hence it follows that
\begin{equation} \label{eqn_intgrpactsumleft}
W(I,P_1+P_2) = \sum_{(\Gamma,\beta)} \frac{\gamma^{g(\Gamma)}\nu^{b(\Gamma)}}{|\Aut(\Gamma,\beta)|} w_{\Gamma,\beta;F_{\ast}(I,P_1)}(I,P_2),
\end{equation}
where the sum is over all (isomorphism classes of) pairs $(\Gamma,\beta)$ taken from the category in Definition \ref{def_catpairs}.

Turning our attention to the right-hand side of \eqref{eqn_intgrpactsum} and using Equation \eqref{eqn_Feynmansplit}, we have
\[ W(W(I,P_1),P_2) = W\biggl(I + \sum_G \frac{\gamma^{g(G)}\nu^{b(G)}}{|\Aut(G)|}w_{G}(I,P_1),P_2\biggr). \]
We therefore turn our attention to computing the Feynman amplitudes in the above. For a connected stable ribbon graph $\Gamma$ we may use Lemma \ref{lem_cyclicsetdettachattach} to write
\begin{multline*}
\biggl(I + \sum_G \frac{\gamma^{g(G)}\nu^{b(G)}}{|\Aut(G)|}w_{G}(I,P_1)\biggr)(\Gamma) =
\\
\sum_{V\subset V(\Gamma)}\sum_{\Lambda:v\mapsto\Lambda(v)}\sum_{\Phi:v\mapsto\Phi_v} \frac{1}{\prod_{v\in V}|\Aut(\Lambda(v))|}\biggl(\underset{v\in V}{\otimes}\Phi_v^{\#}(F_{\Lambda(v)}(I,P_1))\biggr)\otimes\biggl(\underset{v\in V(\Gamma)-V}{\otimes}I(v)\biggr),
\end{multline*}
where the sum is taken over all:
\begin{itemize}
\item
subsets $V$ of $V(\Gamma)$;
\item
functions $\Lambda$ that assign to every vertex $v\in V$, a connected graph (isomorphism class) $\Lambda(v)$ that is not a corolla and satisfies $g(\Lambda(v))=g(v)$ and $b(\Lambda(v))=b(v)$;
\item
functions $\Phi$ that assign to every vertex $v\in V$, a bijective map $\Phi_v$ from the legs of the graph $\Lambda(v)$ to the incident half-edges of $v$, such that $\Phi_v$ respects the cyclic decomposition $C(v)$ on $v$ and the canonical decomposition on $L(\Lambda(v))$.
\end{itemize}

Given a triple $(V,\Lambda,\Phi)$ as above, there is an obvious insertion of the graph $\bigsqcup_{v\in V}\Lambda(v)$ into $\Gamma$ that sends the connected component $\Lambda(v)$ to the vertex $v$. We will denote this insertion by $\iota_{(V,\Lambda,\Phi)}$.

Collecting the terms in the above formula we may write
\[ W(W(I,P_1),P_2) = \sum_{(\Gamma,\iota,G)}\frac{\gamma^{g(\Gamma)}\nu^{b(\Gamma)}|B_{(\Gamma,\iota,G)}|}{|\Aut(\Gamma)|\prod_{\mathscr{G}\subset G}|\Aut(\mathscr{G})|} w_{\Gamma,\iota,G;F_{\ast}(I,P_1)}(I,P_2), \]
where:
\begin{itemize}
\item
the sum is over all triples $(\Gamma,\iota,G)$ taken from the category in Definition \ref{def_cattriples};
\item
we denote by $B_{(\Gamma,\iota,G)}$, the set consisting of all those triples $(V,\Lambda,\Phi)$ as above, for which
\[ \Bigl(\Gamma,\iota_{(V,\Lambda,\Phi)},\bigsqcup_{v\in V}\Lambda(v)\Bigr) \cong (\Gamma,\iota,G); \]
\item
we denote by $\prod_{\mathscr{G}\subset G}$, the product over all the connected components $\mathscr{G}$ of $G$, and
\item
we denote by $w_{\Gamma,\iota,G;F_{\ast}(I,P_1)}(I,P_2)$ the Feynman weight formed by:
\begin{itemize}
\item
attaching, for every connected component $\mathscr{G}$ of $G$, the Feynman amplitude $\iota_{\mathscr{G}}^{\#}(F_{\mathscr{G}}(I,P_1))$ to the vertex $\iota(\mathscr{G})$ of $\Gamma$;
\item
attaching the interaction $I$ to all those remaining vertices of $\Gamma$ that are not targeted by the insertion $\iota$, and
\item
attaching the propagator $P_2$ to all the edges of $\Gamma$.
\end{itemize}
\end{itemize}

It just remains to count $B_{(\Gamma,\iota,G)}$. Given a stable ribbon graph $G$, denote by $S_G$ the group consisting of all those permutations $T$ on the set of connected components of $G$ for which~$T(\mathscr{G})\cong\mathscr{G}$ for every connected component $\mathscr{G}$ of $G$. In this way we have
\[ |\Aut(G)| = |S_G|\prod_{\mathscr{G}\subset G}|\Aut(\mathscr{G})|. \]

Denote by $F_{(\Gamma,\iota,G)}$, the set fibered over $B_{(\Gamma,\iota,G)}$ whose fiber over a point $(V,\Lambda,\Phi)$ consists of all those bijections $T$ from $V$ to the set of connected components of $G$, for which $T(v)\cong\Lambda(v)$ for all $v\in V$. The group $S_G$ acts freely and transitively on the fibers, hence
\[ |F_{(\Gamma,\iota,G)}|=|S_G||B_{(\Gamma,\iota,G)}|. \]
Any bijection $T$ taken from one of the fibers of $F_{(\Gamma,\iota,G)}$ provides a way to canonically identify $G$ with~$\bigsqcup_{v\in V}\Lambda(v)$. Using this, the set $F_{(\Gamma,\iota,G)}$ may be identified with set consisting of all those insertions $\iota'$ from $G$ to $\Gamma$ such that
\[ (\Gamma,\iota',G) \cong (\Gamma,\iota,G). \]

Counting the elements of the latter using the orbit-stabilizer theorem we find that
\[ |F_{(\Gamma,\iota,G)}| = \frac{|\Aut(\Gamma)||\Aut(G)|}{|\Aut(\Gamma,\iota,G)|}. \]
Putting all this together, we arrive at the equation
\begin{equation} \label{eqn_intgrpactsumright}
W(W(I,P_1),P_2) = \sum_{(\Gamma,\iota,G)}\frac{\gamma^{g(\Gamma)}\nu^{b(\Gamma)}}{|\Aut(\Gamma,\iota,G)|} w_{\Gamma,\iota,G;F_{\ast}(I,P_1)}(I,P_2).
\end{equation}
Comparing equations \eqref{eqn_intgrpactsumleft} and \eqref{eqn_intgrpactsumright} and applying the equivalence of categories from Theorem \ref{thm_catequiv}, we see that Equation \eqref{eqn_intgrpactsum} follows.
\end{proof}

\subsubsection{The renormalization group flow at the tree level}

At the tree level the renormalization group flow is quite simple, so here we will take a moment to briefly describe this situation.

\begin{defi} \label{def_tree}
We will say that a stable graph or stable ribbon graph is a \emph{tree} if it is connected and has loop number zero.
\end{defi}

\begin{rem} \label{rem_tree}
This is equivalent to requiring the first Betti number of the graph to vanish and the loop number to vanish at every vertex of the graph. For a stable ribbon graph this in particular implies that every vertex consists of just a single cycle of valency at least three and has vanishing genus and boundary; in other words, it is an ordinary ribbon graph, cf. \cite{KontFeyn} and \cite[\S 5.5]{OpAlgTopPhys}. Furthermore, any such graph $\Gamma$ must have vanishing $g(\Gamma)$ and $b(\Gamma)$, at least three legs, and $|C(\Gamma/\Gamma)|=1$.
\end{rem}

\begin{defi}
Given a family of free theories $\mathcal{E}$ over $\mathcal{A}$, we will say that $I\in\intnoncommI{\mathcal{E},\mathcal{A}}$ is a \emph{tree-level interaction} if
\[ I_{ijk}=0, \quad\text{whenever }(i,j,k)\neq(0,0,1). \]
We will denote the subspace of $\intnoncommI{\mathcal{E},\mathcal{A}}$ consisting of all tree-level interactions by $\intnoncommItree{\mathcal{E},\mathcal{A}}$.
\end{defi}

Consider the projection
\[ \intnoncommI{\mathcal{E},\mathcal{A}}\to\intnoncommItree{\mathcal{E},\mathcal{A}}, \qquad I\mapsto I_{001}. \]

\begin{prop} \label{prop_treeflow}
Let $P\in\mathcal{E}\cotimes\mathcal{E}\cotimes\mathcal{A}$ be a family of propagators for a family $\mathcal{E}$ of free theories over $\mathcal{A}$. Define $W^{\mathrm{Tree}}(I,P)$ by the same formula as \eqref{eqn_RGFnoncomm}, except that now the sum is taken over all \emph{trees}. Then the following diagram commutes:
\begin{displaymath}
\xymatrix{
\intnoncommI{\mathcal{E},\mathcal{A}} \ar[d] \ar[rr]^{W(-,P)} && \intnoncommI{\mathcal{E},\mathcal{A}} \ar[d] \\ \intnoncommItree{\mathcal{E},\mathcal{A}} \ar[rr]^{W^{\mathrm{Tree}}(-,P)} && \intnoncommItree{\mathcal{E},\mathcal{A}}
}
\end{displaymath}
\end{prop}

\begin{proof}
Follows immediately from the discussion in Remark \ref{rem_tree}.
\end{proof}

\subsubsection{The passage from noncommutative to commutative geometry}

As we will now explain, we can pass from noncommutative to commutative geometry by simply forgetting the additional structure present on the noncommutative side. We will see that when we do this, the usual renormalization group flow is lifted by the flow in noncommutative geometry.

\begin{defi} \label{def_mapNCtoCom}
Given a free field theory $\mathcal{E}$, consider the natural quotient map
\[ \sigma:\KontHamP{\mathcal{E}}\to\csalgPdelim{(\mathcal{E}^{\dag})}\subset\intcomm{\mathcal{E}}, \]
where we pass from $\cyc{n}$ to $\sg{n}$-coinvariants. We extend this to a map
\[ \sigma_{\gamma,\nu}:\intnoncomm{\mathcal{E}}\to\intcomm{\mathcal{E}} \]
by defining
\[ \sigma_{\gamma,\nu}(\gamma^i\nu^j x_1\cdots x_k) = \hbar^{2i+j+k-1}\sigma(x_1)\cdots\sigma(x_k), \quad x_1,\ldots,x_k\in\KontHamP{\mathcal{E}}. \]
\end{defi}

Note that the above defines a map from $\intnoncommP{\mathcal{E}}$ to $\intcommP{\mathcal{E}}$ and hence on the spaces of interactions, cf. \eqref{eqn_interactionconstraints} and \eqref{eqn_interactionNCconstraints}.

\begin{theorem} \label{thm_flowNCtoCom}
The map $\sigma_{\gamma,\nu}$ intertwines the two renormalization group flows. That is, if $P\in\mathcal{E}\cotimes\mathcal{E}\cotimes\mathcal{A}$ is a family of propagators for a family $\mathcal{E}$ of free theories over $\mathcal{A}$ then the following diagram commutes:
\begin{displaymath}
\xymatrix{
\intnoncommI{\mathcal{E},\mathcal{A}} \ar[rr]^{W(-,P)} \ar[d]_{\sigma_{\gamma,\nu}} && \intnoncommI{\mathcal{E},\mathcal{A}} \ar[d]^{\sigma_{\gamma,\nu}} \\
\intcommI{\mathcal{E},\mathcal{A}} \ar[rr]^{W(-,P)} && \intcommI{\mathcal{E},\mathcal{A}}
}
\end{displaymath}
\end{theorem}

\begin{proof}
Recall from Example \ref{exm_cyclicwordproduct} that if we are given a family of interactions $I\in\intnoncommI{\mathcal{E},\mathcal{A}}$, we may write
\begin{equation} \label{eqn_interactionnottilde}
I_{ijkl} = \sum_{\begin{subarray}{c} r_1,\ldots,r_k\geq 1: \\ r_1+\cdots+r_k=l \end{subarray}}\lceil I_{ijkl}^{\mathbf{r}}\rceil, \quad\text{for some } I_{ijkl}^\mathbf{r}\in(\mathcal{E}^{\dag})^{\cotimes l}\cotimes\mathcal{A};
\end{equation}
where $\lceil I_{ijkl}^{\mathbf{r}}\rceil$ denotes the image of $I_{ijkl}^{\mathbf{r}}$ under the map \eqref{eqn_attachmapscycdecompoff} that is induced by the cyclic decomposition $C_{\mathbf{r}}$ defined by \eqref{eqn_cycdecpartition}.

Set $\tilde{I}:=\sigma_{\gamma,\nu}(I)$, then
\begin{equation} \label{eqn_interactiontilde}
\tilde{I}_{pq} = \sum_{\begin{subarray}{c} i,j,k\geq 0: \\ 2i+j+k-1 = p\end{subarray}}\sum_{\mathbf{r}}\tilde{I}_{ijkq}^{\mathbf{r}};
\end{equation}
where $\tilde{I}_{ijkq}^{\mathbf{r}}$ denotes the image of the tensor $I_{ijkq}^{\mathbf{r}}$ under the right-hand map of \eqref{eqn_attachmapssym}.

We may write
\[ \sigma_{\gamma,\nu}(W(I,P)) = \sum_{\Gamma} \frac{\hbar^{\ell(\Gamma)}}{|\Aut(\Gamma)|}\tilde{F}_{\Gamma}(I,P), \]
where the sum is taken over all connected stable \emph{ribbon} graphs and $\tilde{F}_{\Gamma}(I,P)$ denotes the image of the Feynman amplitude defined by Definition \ref{def_Feynmanampnoncomm} under the right hand map of \eqref{eqn_attachmapssym}.

To every stable \emph{ribbon} graph, there is naturally associated a stable graph that is obtained by forgetting the cyclic decomposition at each vertex and defining the loop number at each vertex by \eqref{eqn_vertexribbonloopnumber}. Given a stable graph $G$, let $\lfloor G\rfloor$ denote the corresponding fiber of (isomorphism classes of) stable ribbon graphs that sit over $G$ under this mapping. Then we may write
\[ \sigma_{\gamma,\nu}(W(I,P)) = \sum_G\hbar^{\ell(G)}\sum_{\Gamma\in\lfloor G\rfloor} \frac{1}{|\Aut(\Gamma)|}\tilde{F}_{\Gamma}(I,P), \]
where the sum runs over all connected stable graphs $G$, and all stable \emph{ribbon} graphs $\Gamma$ in the fiber $\lfloor G\rfloor$.

From the above, we see that it suffices to show that for every connected stable graph~$G$,
\begin{equation} \label{eqn_graphtoribbongraph}
\frac{1}{|\Aut(G)|}w_G(\tilde{I},P) = \sum_{\Gamma\in\lfloor G\rfloor}\frac{1}{|\Aut(\Gamma)|}\tilde{F}_{\Gamma}(I,P).
\end{equation}

Using \eqref{eqn_interactiontilde} we calculate
\[ \tilde{I}(G) = \sum_{\begin{subarray}{c} g:v\mapsto g(v) \\ b:v\mapsto b(v) \\ k:v\mapsto k(v) \end{subarray}}\sum_{\mathbf{r}:v\mapsto\mathbf{r}(v)}\sum_{\Psi:v\mapsto\Psi_v} \underset{v\in V(G)}{\otimes} \Psi_v^{\#}\Bigl(I_{g(v)b(v)k(v)|v|}^{\mathbf{r}(v)}\Bigr), \]
where the sum is taken over all:
\begin{itemize}
\item
those functions $g$, $b$ and $k$ that assign to every vertex $v$ of $G$, nonnegative integers satisfying
\[2g(v)+b(v)+k(v)-1=l(v);\]
\item
those functions $\mathbf{r}$ that assign to every vertex $v$ of $G$, a list $\mathbf{r}(v)$ consisting of $k(v)$ positive integers  $r_1(v),\ldots,r_{k(v)}(v)$ whose sum is $|v|$;
\item
functions $\Psi$ that assign to every vertex $v$ of $G$, a bijection $\Psi_v$ from the set $\{1,\ldots,|v|\}$ to the incident half-edges of $v$.
\end{itemize}

We may then rewrite this as
\begin{equation} \label{eqn_interactiontildegraph}
\tilde{I}(G) = \sum_{\begin{subarray}{c} g:v\mapsto g(v) \\ b:v\mapsto b(v) \\ k:v\mapsto k(v) \end{subarray}}\sum_{\begin{subarray}{c} \mathbf{r}:v\mapsto\mathbf{r}(v) \\ C:v\mapsto C(v) \end{subarray}}\sum_{\Phi:v\mapsto\Phi_v} \underset{v\in V(G)}{\otimes} \Phi_v^{\#}\Bigl(I_{g(v)b(v)k(v)|v|}^{\mathbf{r}(v)}\Bigr),
\end{equation}
where the sum is taken over all:
\begin{itemize}
\item
functions $g$, $b$, $k$ and $\mathbf{r}$ as above;
\item
functions $C$ that assign to every vertex $v$ of $G$, a cyclic decomposition $C(v)$ of the incident half-edges of $v$ that consists of precisely $k(v)$ cycles;
\item
functions $\Phi$ that assign to every vertex $v$ of $G$, a bijection $\Phi_v$ from the set $\{1,\ldots,|v|\}$ to the incident half-edges of $v$ \emph{that takes the cyclic decomposition $C_{\mathbf{r}(v)}$ to $C(v)$}.
\end{itemize}

Given a quadruple $(g,b,k,C)$ as above, we may turn the stable graph $G$ into a stable \emph{ribbon} graph
\[ \Gamma_{(g,b,k,C)}\in\lfloor G\rfloor \]
in the tautological way.

Now from \eqref{eqn_interactionnottilde} and Lemma \ref{lem_cyclicsetdettachattach} it follows that
\[ I(\Gamma_{(g,b,k,C)}) = \sum_{\mathbf{r}:v\to\mathbf{r}(v)}\sum_{\Phi:v\mapsto\Phi_v}\underset{v\in V(\Gamma_{(g,b,k,C)})}{\otimes} \Phi_v^{\#}\Bigl(I_{g(v)b(v)k(v)|v|}^{\mathbf{r}(v)}\Bigr), \]
where the sum is taken over all those functions $\mathbf{r}$ and $\Phi$ as above. Combining this with \eqref{eqn_interactiontildegraph} and passing to the Feynman amplitudes we get
\begin{displaymath}
\begin{split}
F_G(\tilde{I},P) &= \sum_{g,b,k,C} F_{\Gamma_{(g,b,k,C)}}(I,P), \\
w_G(\tilde{I},P) &= \sum_{\Gamma\in\lfloor G\rfloor}N_{\Gamma}\tilde{F}_{\Gamma}(I,P);
\end{split}
\end{displaymath}
where $N_{\Gamma}$ counts the number of quadruples $(g,b,k,C)$ as above for which the stable ribbon graph $\Gamma_{(g,b,k,C)}$ is isomorphic to $\Gamma$. Note that the group $\Aut(G)$ of automorphisms of the stable graph $G$ acts transitively on the set of such quadruples, with stabilizer the group $\Aut(\Gamma)$ of automorphisms of the stable ribbon graph $\Gamma$; therefore
\[ N_{\Gamma}=\frac{|\Aut(G)|}{|\Aut(\Gamma)|}. \]
This proves \eqref{eqn_graphtoribbongraph}.
\end{proof}

\subsection{Open topological field theories} \label{sec_OTFTflow}

One very important property of the renormalization group flow that we introduced in Section \ref{sec_noncommflow} is that it is compatible with a certain transformation of the space of interactions which we will introduce in this section that is defined by a two-dimensional Open Topological Field Theory (OTFT); indeed, we will see that it is the very axioms of an OTFT that will ensure its compatibility. This transformation was first introduced in a different context in \cite{HamNCBV}.

A well-known theorem due to Atiyah et. al. \cite{AtiyahTFT} states that these theories are in one-to-one correspondence with Frobenius algebras. The most interesting examples for us will come from considering matrix algebras. In this case the transformation we get from our OTFT arises as a certain variation of the map \cite[\S 1.2]{LodayCyclicHomology} that appears in the statement of the Morita invariance of Hochschild cohomology, as we first explained in \cite{GiGqHaZeLQT}. Combining this transformation with the map $\sigma_{\gamma,\nu}$ defined by Definition \ref{def_mapNCtoCom} leads to a variation of the map appearing in the Loday-Quillen-Tsygan Theorem \cite{LodayQuillen, Tsygan}. This will be important for us later when we wish to consider certain large $N$ behavior, such as in Chern-Simons Theory, cf.~\cite{GwHaZeGUE, NCRBV}.

\subsubsection{The transformation defined by an Open Topological Field Theory}

Let $\textgoth{A}$ be a (unital) $\mathbb{Z}$-graded Frobenius algebra. This means that it has a nondegenerate degree zero trace map $\Tr:\textgoth{A}\to\gf$. The nondegeneracy condition ensures there is a degree zero symmetric tensor
\[ \innprod_{\textgoth{A}}^{-1} = \sum_i x_i\otimes y_i \in \textgoth{A}\otimes \textgoth{A} \]
such that for all $a\in \textgoth{A}$,
\[ a = \sum_i \Tr(ax_i)y_i. \]
In particular, $\textgoth{A}$ must be finite-dimensional.

We will define a family of multilinear maps
\begin{equation} \label{eqn_OTFTmaps}
\OTFT{g,b}{r_1,\ldots,r_k}:\textgoth{A}^{\otimes r_1}\otimes\cdots\otimes \textgoth{A}^{\otimes r_k}\to\gf; \qquad g,b\geq 0,\quad r_1,\ldots,r_k\geq 1.
\end{equation}
Let $\Surface{g,b}{r_1,\ldots,r_k}$ denote the standard oriented surface of genus $g$ with $b$ unlabeled boundary components and $k$ labeled boundary components in which the labeling on the $i$th boundary component consists of $r_i$ subintervals embedded into the boundary and labeled from $1$ to $r_i$ clockwise around the boundary according to the orientation. This surface is depicted in Figure \ref{fig_standardsurface}.

\begin{figure}[htp]
\centering
\includegraphics{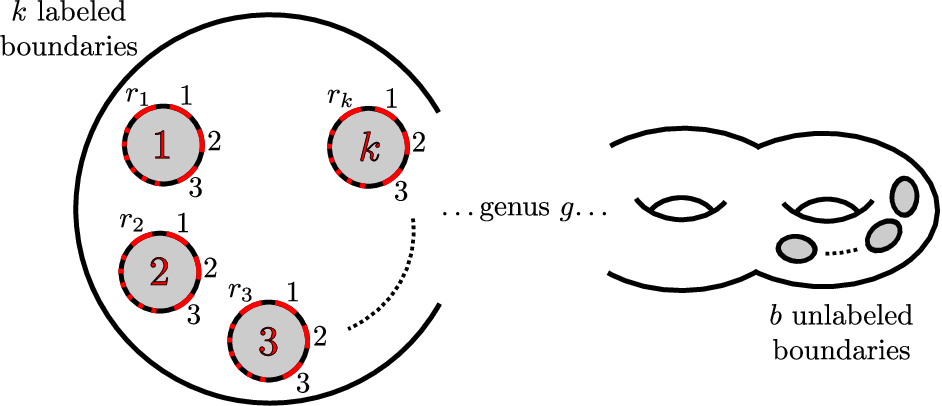}
\caption{The standard surface $\Surface{g,b}{r_1,\ldots,r_k}$ of genus $g$ with $b$ unlabeled boundary components and $k$ labeled components.}
\label{fig_standardsurface}
\end{figure}

The map $\OTFT{g,b}{r_1,\ldots,r_k}$ is the multilinear map assigned to the standard surface $\Surface{g,b}{r_1,\ldots,r_k}$ by the Open Topological Field Theory that is constructed from the Frobenius algebra $\textgoth{A}$. The preceding statement has a precise meaning, not to be recalled here, that is formulated in Section 3 of \cite{ChLaOTFT} in which OTFTs are defined as algebras over a certain modular operad.

Instead of recalling those accompanying details, we will simply provide an explicit formula for these maps. We begin by defining central elements
\[ \Xi_{\mathrm{bdry}} := \sum_i x_i y_i \quad\text{and}\quad \Xi_{\mathrm{gen}} := \sum_{i,j} (-1)^{|x_j||y_i|}x_ix_jy_iy_j \]
of degree zero in our Frobenius algebra $\textgoth{A}$. The multilinear maps are then defined by
\begin{multline*}
\OTFT{g,b}{r_1,\ldots,r_k}(a_{11},\ldots,a_{1r_1};\ldots;a_{k1},\ldots, a_{kr_k}) := \\
\sum_{i_1,\ldots,i_k} (-1)^p \Tr\bigl(x_{i_k}\cdots x_{i_1}\bigr)\Tr\bigl(\Xi_{\mathrm{bdry}}^b\Xi_{\mathrm{gen}}^g y_{i_1}a_{11}\cdots a_{1r_1}\cdots y_{i_k}a_{k1}\cdots a_{kr_k}\bigr),
\end{multline*}
where the sign given by the Koszul sign rule is $p:=\sum_{1\leq t<s\leq k}\sum_{j=1}^{r_t}|y_{i_s}||a_{tj}|$.

\begin{example} \label{exm_matrixOTFT}
When $\textgoth{A}$ is the Frobenius algebra of $N$-by-$N$ matrices $\mat{N}{\gf}$ then the OTFT maps are
\[ \OTFT{g,b}{r_1,\ldots,r_k}(A_{11},\ldots,A_{1r_1};\ldots;A_{k1},\ldots,A_{kr_k}) = N^b\Tr(A_{11}\cdots A_{1r_1})\cdots\Tr(A_{k1}\cdots A_{kr_k}). \]
See \cite[\S 4.5]{GiGqHaZeLQT} for details.
\end{example}

The maps $\OTFT{g,b}{\mathbf{r}}$ satisfy some basic symmetry relations coming from the axioms of an OTFT. Since there is an automorphism of the standard surface $\Surface{g,b}{\mathbf{r}}$ that permutes the boundary components, we have that for all $\varsigma\in\sg{k}$;
\[ \OTFT{g,b}{r_1,\ldots,r_k}(\mathbf{a}_1;\ldots;\mathbf{a}_k) = \\
\pm\OTFT{g,b}{r_{\varsigma(1)},\ldots,r_{\varsigma(k)}}(\mathbf{a}_{\varsigma(1)};\ldots;\mathbf{a}_{\varsigma(k)}), \quad \mathbf{a}_i\in\textgoth{A}^{\otimes r_i}. \]
Likewise, since there is an automorphism of the standard surface that rotates any boundary component of $\Surface{g,b}{\mathbf{r}}$, we have that for any collection of cyclic permutations ${\varsigma_i\in\cyc{r_i}}$;
\[ \OTFT{g,b}{r_1,\ldots,r_k}(\mathbf{a}_1;\ldots;\mathbf{a}_k) = \\
\pm\OTFT{g,b}{r_1,\ldots,r_k}(\varsigma_1\cdot\mathbf{a}_1;\ldots;\varsigma_k\cdot\mathbf{a}_k), \quad\mathbf{a}_i\in\textgoth{A}^{\otimes r_i}. \]

A more succinct way to state these symmetry relations is as follows. Suppose that $\mathbf{r}$ and $\mathbf{r}'$ are both lists of $k$ positive integers whose sum is $l$, and consider the cyclic decompositions $C_\mathbf{r}$ and $C_\mathbf{r'}$ defined by \eqref{eqn_cycdecpartition}; then
\begin{equation} \label{eqn_OTFTrelations}
\OTFT{g,b}{\mathbf{r}} = \OTFT{g,b}{\mathbf{r'}}\circ\varsigma, \quad\text{for every }\varsigma\in\sg{l}\text{ satisfying }\varsigma\cdot C_\mathbf{r} = C_\mathbf{r'}.
\end{equation}

\begin{example} \label{exm_OTFTfreethy}
Let $\mathcal{E}=\Gamma(M,E)$ be a family of free theories over $\mathcal{A}$. Consider the tensor product $E\otimes \textgoth{A}$ of the vector bundle $E$ with the trivial bundle $\textgoth{A}$. Denote the sections by
\[ \mathcal{E}_{\textgoth{A}}:=\Gamma(M,E\otimes \textgoth{A}) = \mathcal{E}\otimes \textgoth{A}. \]
A local pairing may be defined by
\[ \innprodloc[v_1\otimes a_1,v_2\otimes a_2]^{\textgoth{A}} := (-1)^{|a_1||v_2|}\innprodloc[v_1,v_2]\Tr(a_1 a_2).  \]
Extend the family $H$ of generalized Laplacians on $\mathcal{E}$ to $\mathcal{E}_{\textgoth{A}}$ in the obvious way by setting
\[ H_{\textgoth{A}} := (\mathds{1}\cotimes\tau)(H\otimes\mathds{1}):\mathcal{E}_{\textgoth{A}}\to\mathcal{E}_{\textgoth{A}}\cotimes\mathcal{A}. \]

This defines a family $\mathcal{E}_{\textgoth{A}}$ of free theories over $\mathcal{A}$. If $K$ is the heat kernel for $\mathcal{E}$ then
\[ K_{\textgoth{A}} := K\otimes\innprod_{\textgoth{A}}^{-1} \in \smooth{0,\infty}\underset{\mathbb{R}}{\cotimes}\mathcal{E}_{\textgoth{A}}\cotimes\mathcal{E}_{\textgoth{A}}\cotimes\mathcal{A} \]
will be the heat kernel for $\mathcal{E}_{\textgoth{A}}$. We will be most interested in the case where $\mathcal{E}$ is the de Rham algebra and $\textgoth{A}$ is a matrix algebra, which leads to spaces of fields formed from connections.
\end{example}

We define our transformation of the space of interactions that is associated to an OTFT by tensoring our interactions with our OTFT.

\begin{defi} \label{def_OTFTtransformation}
Let $\mathcal{E}$ be a free theory and $\textgoth{A}$ be a Frobenius algebra. Recall from Example \ref{exm_cyclicwordproduct} that given any $I\in\intnoncomm{\mathcal{E}}$ we may write
\[ I_{ijkl} = \sum_{\mathbf{r}}\lceil I_{ijkl}^{\mathbf{r}}\rceil \]
for some (nonunique) choice of distributions $I_{ijkl}^\mathbf{r}\in(\mathcal{E}^{\dag})^{\cotimes l}$. We define a map
\[ \Morita:\intnoncomm{\mathcal{E}}\longrightarrow\intnoncomm{\mathcal{E}_{\textgoth{A}}}, \qquad I=\sum_{i,j,k,l,\mathbf{r}}\gamma^i\nu^j\lceil I_{ijkl}^{\mathbf{r}}\rceil \longmapsto \sum_{i,j,k,l,\mathbf{r}}\gamma^i\nu^j\lceil I_{ijkl}^{\mathbf{r}}\otimes\OTFT{i,j}{\mathbf{r}}\rceil; \]
where
\[ I_{ijkl}^{\mathbf{r}}\otimes\OTFT{i,j}{\mathbf{r}}\in(\mathcal{E}^{\dag})^{\cotimes l}\otimes(\textgoth{A}^{\dag})^{\otimes l} = (\mathcal{E}_{\textgoth{A}}^{\dag})^{\cotimes l} \]
and $\lceil I_{ijkl}^{\mathbf{r}}\otimes\OTFT{i,j}{\mathbf{r}}\rceil$ denotes, as before, its image in $\intnoncomm{\mathcal{E}_{\textgoth{A}}}$ under the map \eqref{eqn_attachmapscycdecompoff} determined by the cyclic decomposition $C_{\mathbf{r}}$ defined by \eqref{eqn_cycdecpartition}.
\end{defi}

Note that $\Morita$ maps $\intnoncommP{\mathcal{E}}$ to $\intnoncommP{\mathcal{E}_{\textgoth{A}}}$ and hence will map interactions to interactions. The map $\Morita$ is well-defined and does not depend on the choice of distributions $I_{ijkl}^\mathbf{r}$. This is a consequence of the symmetry conditions \eqref{eqn_OTFTrelations} imposed by the OTFT.

\subsubsection{Axioms of an Open Topological Field Theory} \label{sec_OTFTaxioms}

We wish to demonstrate that the map $\Morita$ is compatible with the renormalization group flow. For this we will need to recall the properties of an OTFT in the form of the following two Lemmas.

Let $\Gamma$ be a connected stable ribbon graph and at each vertex $v$ of $\Gamma$, use the cyclic decomposition $C(v)$ to attach the labeled subintervals of a standard surface $\Surface{g(v),b(v)}{\mathbf{r}(v)}$ to the incident half-edges of $v$; each cycle of $C(v)$ is attached to a separate boundary component of $\Surface{g(v),b(v)}{\mathbf{r}(v)}$ using the cyclic ordering. Replacing the edges of $\Gamma$ with ribbons $I\times I$ attached to the boundary subintervals yields an oriented surface $\Sigma_{\Gamma}$. The remaining boundary subintervals of $\Sigma_{\Gamma}$ are labeled by the legs of $\Gamma$.

Note that $\Sigma_{\Gamma}$ is independent (up to homeomorphism) of the choices involved in its construction. Since we may rotate and permute the boundary components of the standard surface using a homeomorphism, it does not matter how we use the cyclic decomposition $C(v)$ to attach the standard surface at each vertex.

\begin{lemma} \label{lem_surfacecontract}
Let $\Gamma$ be a connected stable ribbon graph with $l$ legs. Make any choice of a list $\mathbf{r}$ of ${k:=|C(\Gamma/\Gamma)|}$ positive integers, together with a bijection $\phi$ from $\{1,\ldots,l\}$ to $L(\Gamma)$ that takes the cyclic decomposition $C_{\mathbf{r}}$ defined by \eqref{eqn_cycdecpartition} to the canonical decomposition on $L(\Gamma)$. Then $\Sigma_{\Gamma}$ is homeomorphic to the standard surface $\Surface{g(\Gamma),b(\Gamma)}{\mathbf{r}}$, and a homeomorphism may be chosen so that on the boundary subintervals, it becomes the bijection $\phi$.
\end{lemma}

\begin{proof}
By Lemma \ref{lem_contractsubgraph} and using induction on the number of edges of $\Gamma$, it suffices to consider the case when $\Gamma$ has only a single edge. In this case we see immediately that the rules listed in Definition \ref{def_contractedge} ensure that contracting an edge corresponds to attaching a ribbon $I\times I$, cf. Figure \ref{fig_contractedge}.
\begin{figure}[htp]
\centering
\includegraphics{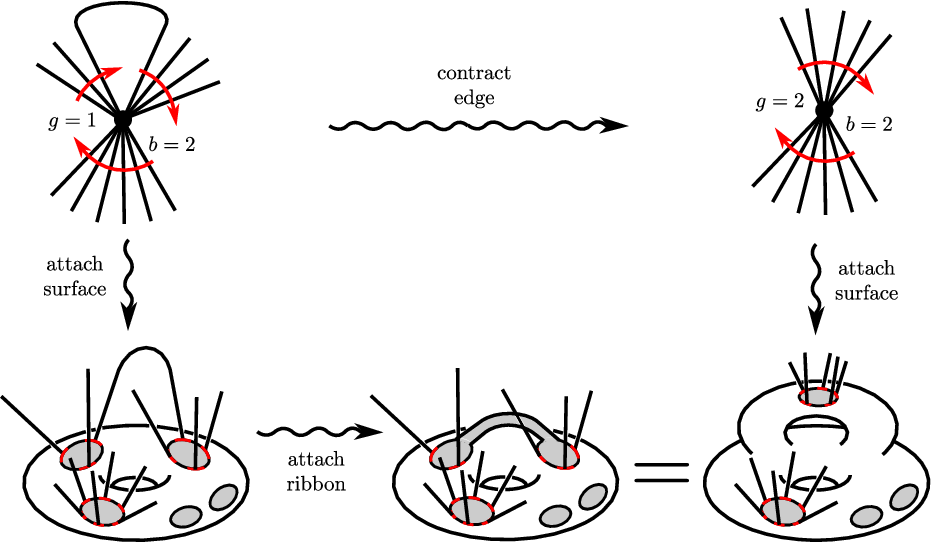}
\caption{Contracting an edge in a stable ribbon graph corresponds to attaching a ribbon to the standard surface (or surfaces) inserted at the corresponding vertex (or vertices). The rules for the genus and boundary functions keep track of the genus and unlabeled boundary components of the corresponding surfaces.}
\label{fig_contractedge}
\end{figure}\end{proof}

Let $\textgoth{A}$ be a Frobenius algebra and $\Gamma$ be a connected stable ribbon graph. Using the same construction as above, we may attach the OTFT maps
\[ \OTFT{g(v),b(v)}{\mathbf{r}}\in (\textgoth{A}^{\dag})^{\otimes |v|} \]
to every vertex $v$ of $\Gamma$ using any list of positive integers $\mathbf{r}$ and bijection $\phi$ that takes the cyclic decomposition $C_{\mathbf{r}}$ to $C(v)$. Again, the symmetry relations \eqref{eqn_OTFTrelations} ensure that the resulting map $\mathbf{T}(v)\in\textgoth{A}^{\dag}(\!(v)\!)$ does not depend upon the choice of $\mathbf{r}$ and $\phi$. Combining the maps $\mathbf{T}(v)$ from all the vertices $v$ of $\Gamma$ yields a map $\mathbf{T}(\Gamma)\in\textgoth{A}^{\dag}(\!(H(\Gamma))\!)$. Now attach the symmetric tensors $\innprod_{\textgoth{A}}^{-1}$ to every edge $e$ of $\Gamma$ to form a tensor $\innprod_{\textgoth{A}}^{-1}(\Gamma)$. Evaluating one on the other defines a map
\[ F_{\Gamma}(\textgoth{A}) := \mathbf{T}(\Gamma)\bigl[\innprod_{\textgoth{A}}^{-1}(\Gamma)\bigr] \in\textgoth{A}^{\dag}(\!(L(\Gamma))\!). \]

\begin{lemma} \label{lem_OTFTaxiom}
Let $\Gamma$ be a connected stable ribbon graph with $l$ legs and $\textgoth{A}$ be a Frobenius algebra. As before, make any choice of a list $\mathbf{r}$ of ${k:=|C(\Gamma/\Gamma)|}$ positive integers, together with a bijection $\phi$ that takes the cyclic decomposition $C_{\mathbf{r}}$ to the canonical decomposition on $L(\Gamma)$. Then the image in $(\textgoth{A}^{\dag})^{\otimes l}$ of $F_{\Gamma}(\textgoth{A})$ under the map \eqref{eqn_labelmaps} defined by the bijection $\phi$ is the multilinear map $\OTFT{g(\Gamma),b(\Gamma)}{\mathbf{r}}$.
\end{lemma}

\begin{proof}
This is precisely the statement (when applied in conjunction with Lemma \ref{lem_contractsubgraph} and Lemma \ref{lem_surfacecontract}) that the maps \eqref{eqn_OTFTmaps} define an Open Topological Field Theory. It appears as Proposition 3.4 and Theorem 3.7(2) in \cite{ChLaOTFT}, where the axioms of an OTFT are recast and extended using the language of modular operads.
\end{proof}

\subsubsection{Compatibility with the renormalization group flow} We will now show that the transformation defined by an OTFT is compatible with the renormalization group flow. This will be a consequence of the axioms for an OTFT.

\begin{theorem} \label{thm_flowOTFT}
Let $P\in\mathcal{E}\cotimes\mathcal{E}\cotimes\mathcal{A}$ be a family of propagators for a family of free theories $\mathcal{E}$ over $\mathcal{A}$, and let $\textgoth{A}$ be a Frobenius algebra. Define a family of propagators for $\mathcal{E}_{\textgoth{A}}$ by
\begin{equation} \label{eqn_OTFTpropagator}
P_{\textgoth{A}} := P\otimes\innprod_{\textgoth{A}}^{-1}\in\mathcal{E}_{\textgoth{A}}\cotimes\mathcal{E}_{\textgoth{A}}\cotimes\mathcal{A}.
\end{equation}
Then the following diagram commutes:
\begin{displaymath}
\xymatrix{
\intnoncommI{\mathcal{E},\mathcal{A}} \ar[rr]^{W(-,P)} \ar[d]_{\Morita} && \intnoncommI{\mathcal{E},\mathcal{A}} \ar[d]^{\Morita} \\
\intnoncommI{\mathcal{E}_{\textgoth{A}},\mathcal{A}} \ar[rr]^{W(-,P_{\textgoth{A}})} && \intnoncommI{\mathcal{E}_{\textgoth{A}},\mathcal{A}}
}
\end{displaymath}
\end{theorem}

\begin{proof}
We must show that for every connected stable ribbon graph $\Gamma$ and interaction ${I\in\intnoncommI{\mathcal{E},\mathcal{A}}}$,
\begin{equation} \label{eqn_OTFTcommute}
\gamma^{g(\Gamma)}\nu^{b(\Gamma)}w_{\Gamma}\bigl(\Morita(I),P_{\textgoth{A}}\bigr) = \Morita\bigl(\gamma^{g(\Gamma)}\nu^{b(\Gamma)}w_{\Gamma}(I,P)\bigr).
\end{equation}
Note that using a representation for $I$ of the form \eqref{eqn_intcycrep}, it follows from Lemma \ref{lem_cyclicsetdettachattach} that for every vertex $v$ of $\Gamma$,
\begin{displaymath}
\begin{split}
\Morita(I)(v) &= \sum_{\mathbf{r}}\sum_{\phi}\phi^{\#}(I^{\mathbf{r}}_{ijkl})\otimes\phi^{\#}(\OTFT{i,j}{\mathbf{r}}) \\
&= \sum_{\mathbf{r}}\sum_{\phi}\phi^{\#}(I^{\mathbf{r}}_{ijkl})\otimes\mathbf{T}(v) = I(v)\otimes\mathbf{T}(v);
\end{split}
\end{displaymath}
where $i:=g(v)$, $j:=b(v)$, $k:=|C(v)|$, $l:=|v|$ and we sum over all $k$-tuples of positive integers $\mathbf{r}$ whose sum is $l$ and all bijections $\phi$ that take $C_{\mathbf{r}}$ to $C(v)$.

From this we conclude that
\[ F_{\Gamma}\bigl(\Morita(I),P_{\textgoth{A}}\bigr) = F_{\Gamma}(I,P)\otimes F_{\Gamma}(\textgoth{A}). \]
To find an expression for the Feynman weights $w_{\Gamma}$, choose an $\mathbf{r}$ and a bijection $\phi$ that takes the cyclic decomposition on $L(\Gamma)$ to $C_{\mathbf{r}}$. We have
\begin{displaymath}
\begin{split}
\Morita\bigl(\gamma^{g(\Gamma)}\nu^{b(\Gamma)}w_{\Gamma}(I,P)\bigr) &= \Morita\bigl(\gamma^{g(\Gamma)}\nu^{b(\Gamma)}\bigl\lceil\phi^{\#}F_{\Gamma}(I,P)\bigr\rceil\bigr) \\
&= \gamma^{g(\Gamma)}\nu^{b(\Gamma)}\bigl\lceil\phi^{\#}F_{\Gamma}(I,P)\otimes\OTFT{g(\Gamma),b(\Gamma)}{\mathbf{r}}\bigr\rceil, \\
w_{\Gamma}\bigl(\Morita(I),P_{\textgoth{A}}\bigr) &= \bigl\lceil\phi^{\#}F_{\Gamma}\bigl(\Morita(I),P_{\textgoth{A}}\bigr)\bigr\rceil \\
&= \bigl\lceil\phi^{\#}F_{\Gamma}(I,P)\otimes\phi^{\#}F_{\Gamma}(\textgoth{A})\bigr\rceil; \end{split}
\end{displaymath}
where $\lceil x\rceil$ denotes the image of $x$ under the map \eqref{eqn_attachmapscycdecompoff} defined by the cyclic decomposition $C_{\mathbf{r}}$. Equation \eqref{eqn_OTFTcommute} now follows from Lemma \ref{lem_OTFTaxiom}.
\end{proof}

\section{Noncommutative effective field theories} \label{sec_NCeffthy}

In this section we will formulate and prove a noncommutative analogue of one of the main results of \cite{CosEffThy} that is fundamental to the perspective put forth there by Costello. Using our version of the renormalization group flow developed in the preceding section, we will introduce---following Costello---the notion of a noncommutative effective field theory. We will then prove one of the main results of this paper---that noncommutative effective field theories are in one-to-one correspondence with local interactions that are taken from our space based upon the noncommutative geometry of Kontsevich.

The correspondence is based upon the process of renormalization and we will begin by introducing and recalling from \cite{CosEffThy} the requisite notions and results that are necessary to discuss this process in our context.

\subsection{Locality}

We introduced in Definition \ref{def_localdistribution} the notion of a local distribution, which we now use to define the notion of a local functional; cf. \cite[\S 2.13.4]{CosEffThy}.

\begin{defi}
Let $\mathcal{E}$ be a family of free theories over $\mathcal{A}$. We will say that a functional in $\bigl[(\mathcal{E}^{\dag})^{\cotimes j}\bigr]_{\sg{j}}\cotimes\mathcal{A}$ is \emph{local} if it may be represented by a local $\mathcal{A}$-valued distribution. More generally, a functional $I\in\intcomm{\mathcal{E},\mathcal{A}}$ will be called local if $I_{ij}$ is local for all $i$ and $j$. We will denote the subspace of $\intcomm{\mathcal{E},\mathcal{A}}$ consisting of all local functionals by $\intcommL{\mathcal{E},\mathcal{A}}$ and the subspace consisting of all local interactions by $\intcommLI{\mathcal{E},\mathcal{A}}$.

Turning our attention now to the noncommutative case, we will say that $I\in\intnoncomm{\mathcal{E},\mathcal{A}}$ is \emph{local} if for every $i$, $j$, $k$ and $l$ there is a representation
\[ I_{ijkl} = \sum_{\mathbf{r}}\lceil I_{ijkl}^{\mathbf{r}}\rceil, \quad I_{ijkl}^\mathbf{r}\in(\mathcal{E}^{\dag})^{\cotimes l}\cotimes\mathcal{A} \]
of the form \eqref{eqn_intcycrep} in which every $I_{ijkl}^{\mathbf{r}}$ is a local $\mathcal{A}$-valued distribution. We will denote the subspace of $\intnoncomm{\mathcal{E},\mathcal{A}}$ consisting of all local functionals by $\intnoncommL{\mathcal{E},\mathcal{A}}$ and the subspace consisting of all local interactions by $\intnoncommLI{\mathcal{E},\mathcal{A}}$.
\end{defi}

We note that it follows tautologically from the definitions that the maps $\sigma_{\gamma,\nu}$ and $\Morita$ described in Definition \ref{def_mapNCtoCom} and \ref{def_OTFTtransformation} respectively take local functionals to local functionals.

\subsection{Renormalization and counterterms}

The construction of an effective field theory is carried out through the process of renormalization. Here we will recall the relevant details from \cite{CosEffThy} that describe the analytic properties of the Feynman amplitudes that make them amenable to this process and prove the existence of a set of counterterms for any local interaction.

\subsubsection{Renormalization schemes}

To define the renormalization of the Feynman amplitudes, we need the notion of a renormalization scheme.

\begin{defi}
Consider the space $\smooth{0,1}$ of smooth functions on the open unit interval and denote by $\smoothlimit{0,1}$ the subspace consisting of those functions that admit a limit as $\varepsilon\to 0$. A \emph{renormalization scheme} is a choice of a complimentary subspace
\[ \smooth{0,1} = \smoothlimit{0,1}\oplus\smoothsing{0,1}. \]
We will denote the projection onto this subspace by $\mathpzc{p}:\smooth{0,1}\to\smoothsing{0,1}$.
\end{defi}

For the rest of the paper we will assume that we have made a fixed choice of renormalization scheme.

\begin{rem}
We should explain the relationship of the above definition to that of Costello's, cf. Definition 9.4.1 of \cite[\S 2.9.4]{CosEffThy}. In that definition, the space of smooth functions is replaced with the space of periods. The above definition encompasses the latter in the following sense. Suppose that $\mathscr{P}(0,1)$ is the space of periods (or, in fact, any subspace of $\smooth{0,1}$) and we have a decomposition
\[ \mathscr{P}(0,1) = \mathscr{P}_{\geq 0}(0,1)\oplus \mathscr{P}_{<0}(0,1), \]
where $\mathscr{P}_{\geq 0}(0,1) := \mathscr{P}(0,1)\cap\smoothlimit{0,1}$. Then we may choose, arbitrarily, a complement
\[ \smooth{0,1} = \smoothlimit{0,1}\oplus \mathscr{P}_{<0}(0,1)\oplus F \]
and set $\smoothsing{0,1} := \mathscr{P}_{<0}(0,1)\oplus F$. The projection $\mathpzc{p}:\smooth{0,1}\to\smoothsing{0,1}$ defined in this way becomes the projection onto the singular periods when restricted to the subspace of periods $\mathscr{P}(0,1)$.
\end{rem}

\subsubsection{Counterterms and asymptotic expansions}

Central to the construction of counterterms is the fact that the singular behavior of the Feynman amplitudes of local interactions may be explicitly described in terms of an asymptotic expansion in the short-time cutoff parameter $\varepsilon$. This was done by Costello in Theorem 4.0.2 of \cite[App. 1]{CosEffThy}, and we begin by recalling the details of this theorem that are relevant to us here.

\begin{theorem} \label{thm_asymptotics}
Let $\mathcal{E}$ be a family of free theories over $\mathcal{A}$ and let $\Gamma$ be a stable ribbon graph. Let $P(\varepsilon,L)$ be a propagator of the form \eqref{eqn_canonicalpropagator} from Example \ref{exm_canonicalpropagator} and let $I\in\intnoncommLI{\mathcal{E},\mathcal{A}}$ be a family of local interactions. Then there are functionals
\[ \Psi_r\in\csalg{\KontHamP{\mathcal{E}}}\cotimes\mathcal{A}\underset{\mathbb{R}}{\cotimes}\smooth{0,\infty},\quad r\geq 1; \]
smooth functions $g_r\in\smooth{0,1}$, and a monotonically increasing sequence of integers $d_R$ tending to infinity, such that for all $R$ and $L>0$,
\[ \lim_{\varepsilon\to 0}\varepsilon^{-d_R}\biggl( w_{\Gamma}(I,P(\varepsilon,L)) - \sum_{r=1}^R g_r(\varepsilon)\Psi_r(L)\biggr) = 0; \]
where $\Psi_r(L)\in\csalg{\KontHamP{\mathcal{E}}}\cotimes\mathcal{A}$ denotes the functional evaluated at the point $L$.

Furthermore, there are local functionals
\[ \Phi_{rs}\in\csalg{\KontHamP{\mathcal{E}}}\cotimes\mathcal{A}, \quad r,s\geq 1 \]
and smooth functions $f_{rs}\in\smooth{0,\infty}$ such that for every $r$, there is a monotonically increasing sequence of integers $d_S$ tending to infinity, such that for all $S$,
\[ \lim_{L\to 0}L^{-d_S}\biggl(\Psi_r(L)-\sum_{s=1}^S f_{rs}(L)\Phi_{rs}\biggr) = 0. \]
\end{theorem}

\begin{rem}
As explained in \cite{CosEffThy}, the functions $g_r$ are actually periods and have finite order poles at zero.
\end{rem}

\begin{rem}
Theorem \ref{thm_asymptotics} follows directly from the above cited theorem of \cite{CosEffThy}, except perhaps that it may be necessary to explain two minor details.

First, the theorem in \cite{CosEffThy} is formulated using totally ordered graphs. The Feynman amplitude for a stable ribbon graph will be a sum of such amplitudes over all possible compatible orderings and so Theorem \ref{thm_asymptotics} follows immediately from this.

Second, the cited theorem does not directly consider propagators of the form \eqref{eqn_canonicalpropagator} which are modified by a differential operator $D$. However, the operator $D$ may instead be absorbed into the interaction terms, which is possible using the total ordering on the graph.
\end{rem}

Following \cite{CosEffThy}, we may use the preceding theorem to extract the singular parts of the Feynman weights.

\begin{defi}
Let $\mathcal{E}$ be a family of free theories over $\mathcal{A}$ and suppose that
\[ W:(0,1)\to\intnoncomm{\mathcal{E},\mathcal{A}}, \qquad \varepsilon\mapsto W(\varepsilon) \]
is a function for which we can find $R$ functionals $I_r\in\intnoncomm{\mathcal{E},\mathcal{A}}$ and smooth functions $g_r\in\smooth{0,1}$ such that
\[ \lim_{\varepsilon\to 0}\biggl(W(\varepsilon) - \sum_{r=1}^R g_r(\varepsilon)I_r\biggr) = 0. \]
Then we may define the singular part of this function by
\[ \Sing(W) := \sum_{r=1}^R \mathpzc{p}(g_r)I_r. \]
\end{defi}

It is straightforward to show that $\Sing$ is well-defined and has the following properties:
\begin{itemize}
\item
$\Sing$ is an idempotent linear operator on the space of functions defined above.
\item
$W$ is nonsingular (that is $W(\varepsilon)$ converges as $\varepsilon\to 0$) if and only if $\Sing(W)=0$.
\item
$(W-\Sing(W))$ is always nonsingular.
\item
$\Sing$ is natural with respect to continuous transformations of $\intnoncomm{\mathcal{E},\mathcal{A}}$.
\end{itemize}

To construct the counterterms for our theory, we will use an inductive procedure similar to the one employed by Costello in Theorem 10.1.1 of \cite[\S 2.10.1]{CosEffThy}, which was shown in \cite{HamBPHZ} to be equivalent to the usual BPHZ algorithm \cite[\S 5]{ColRenormalization}. We therefore consider the lexicographic order on the terms in our interaction. This well-ordering is defined for lists $\mathbf{a}$ and $\mathbf{b}$ consisting of $n$ nonnegative integers by
\[ \mathbf{a}<\mathbf{b} \Leftrightarrow \exists i:(a_i<b_i \land \forall j<i,a_j=b_j).\]

For convenience, we introduce the following notation for a functional $I\in\intnoncomm{\mathcal{E},\mathcal{A}}$:
\[ I_{<(i,j,k,l)}:=\sum_{(p,q,r,s)<(i,j,k,l)}\gamma^p\nu^q I_{pqrs}, \qquad I_{\leq(i,j,k,l)}:=\sum_{(p,q,r,s)\leq(i,j,k,l)}\gamma^p\nu^q I_{pqrs}. \]
The construction of the counterterms rests on the following basic formula.

\begin{prop} \label{prop_orderformula}
Let $\mathcal{E}$ be a family of free theories over $\mathcal{A}$ and $P$ be a family of propagators for $\mathcal{E}$. Let $i$, $j$, $k$ and $l$ be nonnegative integers and $I,J\in\intnoncommI{\mathcal{E},\mathcal{A}}$ be families of interactions such that $J_{<(i,j,k,l)}=0$. Then
\[ W(I+J,P) = W(I,P) + \gamma^i\nu^j J_{ijkl} + O, \]
where $O_{\leq(i,j,k,l)}$ vanishes.
\end{prop}

\begin{proof}
We make the following claim. Suppose that we are given
\begin{equation} \label{eqn_fourtupleineq}
(p,q,r,s)\leq (i,j,k,l)\leq (p',q',r',s'),
\end{equation}
a connected stable ribbon graph $\Gamma$ for which
\[ g(\Gamma)=p, \quad b(\Gamma)=q, \quad |C(\Gamma/\Gamma)|=r \quad\text{and}\quad |L(\Gamma)|=s \]
and with a vertex $v\in V(\Gamma)$ for which
\[ g(v)=p', \quad b(v)=q', \quad |C(v)|=r', \quad\text{and}\quad |v|=s'. \]
Then $\Gamma$ is a corolla and $(p,q,r,s)=(i,j,k,l)=(p',q',r',s')$.

The proposition follows as a direct consequence of Equation \eqref{eqn_intgrpactzero} and the above claim, which we now prove. By \eqref{eqn_fourtupleineq} we have $p\leq p'$ and by \eqref{eqn_vertexinequalities} we know that $p'\leq p$ so that $p=p'$. Using the leftmost inequality of \eqref{eqn_genbdryinequalities} it then follows that
\begin{equation} \label{eqn_orderformulaaux1}
1 - |V(\Gamma)| + \frac{1}{2}\bigl(|E(\Gamma)|+|C(\Gamma)|-\beta_{\Gamma}^{\#}\bigr) = 0
\end{equation}
and that $g(v')=0$ for all vertices $v'\neq v$.

Next, since $p=p'$, by \eqref{eqn_fourtupleineq} we must have $q\leq q'$ and by \eqref{eqn_vertexinequalities} we must have $q'\leq q$. Hence $q=q'$ and, using the rightmost inequality of \eqref{eqn_genbdryinequalities}, it follows that
\begin{equation} \label{eqn_orderformulaaux2}
\beta_{\Gamma}^{\#}=|C(\Gamma/\Gamma)|=r
\end{equation}
and that $b(v')=0$---and consequently, by \eqref{eqn_valencyribbonconditions}, that $|C(v')|\geq 1$---for all vertices $v'\neq v$.

Next we write
\begin{equation} \label{eqn_orderformulaaux3}
\begin{split}
|C(\Gamma)| &= r'+\sum_{v'\neq v}|C(v')| \\
&= r' + |V(\Gamma)|-1 + \sum_{v'\neq v}\bigl(|C(v')|-1\bigr).
\end{split}
\end{equation}
Substituting \eqref{eqn_orderformulaaux2} and \eqref{eqn_orderformulaaux3} into \eqref{eqn_orderformulaaux1} yields the equation
\[ \mathbf{b}_1(\Gamma) + \sum_{v'\neq v}\bigl(|C(v')|-1\bigr) + r'-r = 0. \]
Since $p=p'$ and $q=q'$ we must have $r\leq r'$ by \eqref{eqn_fourtupleineq}. It then follows from the above equation that $r=r'$, the first Betti number $\mathbf{b}_1(\Gamma)$ of the graph vanishes, the Euler characteristic is $\chi(\Gamma)=1$, and that $|C(v')|=1$ for all vertices $v'\neq v$.

Now it follows from \eqref{eqn_valencyribbonconditions} that every vertex $v'\neq v$ is at least trivalent. From this we obtain
\begin{displaymath}
\begin{split}
s &= |H(\Gamma)|-2|E(\Gamma)| = s' + \sum_{v'\neq v}|v'| -2|E(\Gamma)| \\
&\geq s' + 3(|V(\Gamma)|-1) -2|E(\Gamma)| = s' + |E(\Gamma)|.
\end{split}
\end{displaymath}
Now since $p=p'$, $q=q'$ and $r=r'$ we must have $s\leq s'$ by \eqref{eqn_fourtupleineq}. It therefore follows from the above inequality that $s=s'$, $|E(\Gamma)|=0$ and $|V(\Gamma)|=1$.
\end{proof}

Now we are ready to explain the construction of the counterterms.

\begin{theorem} \label{thm_counterterms}
Let $\mathcal{E}$ be a family of free theories over $\mathcal{A}$ and $I\in\intnoncommLI{\mathcal{E},\mathcal{A}}$ a family of local interactions. Suppose that $P(\varepsilon,L)$ is a propagator of the form \eqref{eqn_canonicalpropagator} defined by Example \ref{exm_canonicalpropagator}. Then there is a unique series of purely singular local counterterms for $I$ and $P(-,-)$; that is, there is a unique $\varepsilon$-parameterized interaction
\[ I^{\mathrm{CT}}\in\intnoncommI{\mathcal{E},\mathcal{A}}\underset{\mathbb{R}}{\cotimes}\smooth{0,1} \]
such that:
\begin{itemize}
\item
for all $i,j,k,l\geq 0$;
\[ I^{\mathrm{CT}}_{ijkl}\in\intnoncommL{\mathcal{E},\mathcal{A}}\underset{\mathbb{R}}{\otimes}\smoothsing{0,1}, \]
\item
and for all $L>0$, the limit
\[ \lim_{\varepsilon\to 0}W\bigl(I-I^{\mathrm{CT}}(\varepsilon),P(\varepsilon,L)\bigr) \]
exists, where convergence takes place in $\intnoncommI{\mathcal{E},\mathcal{A}}$. In the above expression, $I^{\mathrm{CT}}(\varepsilon)\in\intnoncommI{\mathcal{E},\mathcal{A}}$ denotes the value of $I^{\mathrm{CT}}$ at the point $\varepsilon\in (0,1)$.
\end{itemize}
\end{theorem}

\begin{rem}
Here we emphasize that the use of  $\otimes$ above denotes the ordinary algebraic tensor product, in which only finite sums are allowed.
\end{rem}

\begin{rem} \label{rem_CtermCos}
In Theorem 13.4.3 of \cite[\S 2.13.4]{CosEffThy} Costello proves the same theorem for the space $\intcomm{\mathcal{E},\mathcal{A}}$. That is, the above theorem holds if we replace $\intnoncomm{\mathcal{E},\mathcal{A}}$ everywhere in the above by $\intcomm{\mathcal{E},\mathcal{A}}$. This will be relevant later in Section \ref{sec_RenormEffThy} when we talk about the construction of effective field theories.
\end{rem}

\begin{rem} \label{rem_ctermtransform}
Consider the transformations $\sigma_{\gamma,\nu}$ and $\Morita$ defined by Definition \ref{def_mapNCtoCom} and Definition \ref{def_OTFTtransformation} respectively. It follows from Theorem \ref{thm_flowNCtoCom} and Theorem \ref{thm_flowOTFT} that these transformations respect the counterterms assigned to a local interaction.
\end{rem}

\begin{proof}[Proof of Theorem \ref{thm_counterterms}]
The argument is essentially the same as that of Theorem 10.1.1 of \cite[\S 2.10]{CosEffThy}, which we nonetheless include here for the sake of completeness.

First, we address the uniqueness of the counterterms. Suppose that $I^{\mathrm{CT}}$ and $J^{\mathrm{CT}}$ are two series of purely singular local counterterms for $I$ and $P(-,-)$. We prove by induction on the well-ordering that
\begin{equation} \label{eqn_uniquecounterterms}
I_{pqrs}^{\mathrm{CT}} = J_{pqrs}^{\mathrm{CT}}, \quad\text{for all }p,q,r,s\geq 0.
\end{equation}

Suppose that \eqref{eqn_uniquecounterterms} holds for all $(p,q,r,s)<(i,j,k,l)$. Using Proposition \ref{prop_orderformula} we see that
\begin{displaymath}
\begin{split}
W_{ijkl}\bigl(I-I^{\mathrm{CT}}(\varepsilon),P(\varepsilon,L)\bigr) &= W_{ijkl}\bigl(I-I_{\leq(i,j,k,l)}^{\mathrm{CT}}(\varepsilon),P(\varepsilon,L)\bigr) \\
&= W_{ijkl}\bigl(I-J_{\leq(i,j,k,l)}^{\mathrm{CT}}(\varepsilon) - \bigl(I_{ijkl}^{\mathrm{CT}}(\varepsilon)-J_{ijkl}^{\mathrm{CT}}(\varepsilon)\bigr),P(\varepsilon,L)\bigr) \\
&= W_{ijkl}\bigl(I-J^{\mathrm{CT}}(\varepsilon),P(\varepsilon,L)\bigr) - \bigl(I_{ijkl}^{\mathrm{CT}}(\varepsilon)-J_{ijkl}^{\mathrm{CT}}(\varepsilon)\bigr).
\end{split}
\end{displaymath}
It follows that
\[ 0=\Sing\bigl(I_{ijkl}^{\mathrm{CT}}-J_{ijkl}^{\mathrm{CT}}\bigr) = I_{ijkl}^{\mathrm{CT}}-J_{ijkl}^{\mathrm{CT}}, \]
where we have used the fact that the counterterms are purely singular.

To construct the counterterms $I_{ijkl}^{\mathrm{CT}}$, we again proceed by induction on the well-ordering. Suppose that purely singular local counterterms
\[ I_{pqrs}^{\mathrm{CT}}, \quad (p,q,r,s)<(i,j,k,l) \]
have been constructed such that the limit
\[ \lim_{\varepsilon\to 0}W_{pqrs}\bigl(I-I_{\leq(p,q,r,s)}^{\mathrm{CT}}(\varepsilon),P(\varepsilon,L)\bigr) \]
exists for all $(p,q,r,s)<(i,j,k,l)$ and $L>0$.

Define
\begin{equation} \label{eqn_counterterms}
I_{ijkl}^{\mathrm{CT}}(\varepsilon,L) := \Sing\Bigl[\varepsilon\mapsto W_{ijkl}\bigl(I-I_{<(i,j,k,l)}^{\mathrm{CT}}(\varepsilon),P(\varepsilon,L)\bigr)\Bigr].
\end{equation}
This makes sense by Theorem \ref{thm_asymptotics}. Now, using Proposition \ref{prop_orderformula} we see that for all $L>0$,
\begin{multline*}
\Sing\Bigl[\varepsilon\mapsto W_{ijkl}\bigl(I-I_{<(i,j,k,l)}^{\mathrm{CT}}(\varepsilon)-I_{ijkl}^{\mathrm{CT}}(\varepsilon,L),P(\varepsilon,L)\bigr)\Bigr] = \\
\Sing\Bigl[\varepsilon\mapsto W_{ijkl}\bigl(I-I_{<(i,j,k,l)}^{\mathrm{CT}}(\varepsilon),P(\varepsilon,L)\bigr)-I_{ijkl}^{\mathrm{CT}}(\varepsilon,L)\Bigr]=0.
\end{multline*}

It remains to show that $I_{ijkl}^{\mathrm{CT}}(\varepsilon,L)$ does not depend on $L$ and has the required form. We begin with the first problem and use Equation \eqref{eqn_intgrpactsum} of Theorem \ref{thm_intgrpact} along with Proposition \ref{prop_orderformula}.
\begin{displaymath}
\begin{split}
I_{ijkl}^{\mathrm{CT}}(\varepsilon,L') &= \Sing\Bigl[\varepsilon\mapsto W_{ijkl}\bigl(I-I_{<(i,j,k,l)}^{\mathrm{CT}}(\varepsilon),P(\varepsilon,L')\bigr)\Bigr] \\
&= \Sing\Bigl[\varepsilon\mapsto W_{ijkl}\bigl(W_{\leq(i,j,k,l)}\bigl(I-I_{<(i,j,k,l)}^{\mathrm{CT}}(\varepsilon),P(\varepsilon,L)\bigr),P(L,L')\bigr)\Bigr] \\
&= W_{ijkl}\Bigl(\Sing\Bigl[\varepsilon\mapsto W_{\leq(i,j,k,l)}\bigl(I-I_{<(i,j,k,l)}^{\mathrm{CT}}(\varepsilon),P(\varepsilon,L)\bigr)\Bigr],P(L,L')\Bigr).
\end{split}
\end{displaymath}
However, for all $(p,q,r,s)<(i,j,k,l)$;
\[ W_{pqrs}\bigl(I-I_{<(i,j,k,l)}^{\mathrm{CT}}(\varepsilon),P(\varepsilon,L)\bigr) = W_{pqrs}\bigl(I-I_{\leq(p,q,r,s)}^{\mathrm{CT}}(\varepsilon),P(\varepsilon,L)\bigr), \]
which is nonsingular by the inductive hypothesis. Therefore,
\begin{displaymath}
\begin{split}
I_{ijkl}^{\mathrm{CT}}(\varepsilon,L') &= W_{ijkl}\Bigl(\Sing\Bigl[\varepsilon\mapsto W_{ijkl}\bigl(I-I_{<(i,j,k,l)}^{\mathrm{CT}}(\varepsilon),P(\varepsilon,L)\bigr)\Bigr],P(L,L')\Bigr) \\
&= W_{ijkl}\Bigl(I_{ijkl}^{\mathrm{CT}}(\varepsilon,L),P(L,L')\Bigr) = I_{ijkl}^{\mathrm{CT}}(\varepsilon,L).
\end{split}
\end{displaymath}

By Theorem \ref{thm_asymptotics} we may write
\[ I_{ijkl}^{\mathrm{CT}}(\varepsilon,L) = \sum_{r=1}^R g_r(\varepsilon)\Psi_r(L) \]
where the $g_r\in\smoothsing{0,1}$ are linearly independent and each $\Psi_r$ has an approximation
\[ \lim_{L\to 0}\biggl(\Psi_r(L)-\sum_{s=1}^S f_{rs}(L)\Phi_{rs}\biggr) = 0 \]
in terms of local functionals $\Phi_{rs}\in\csalg{\KontHamP{\mathcal{E}}}\cotimes\mathcal{A}$. Since the counterterms $I_{ijkl}^{\mathrm{CT}}(\varepsilon,L)$ do not depend upon $L$, neither will each $\Psi_r$, and it follows from the above approximation that each $\Psi_r$ is a local functional.
\end{proof}

We will need the following very basic and routine result concerning the counterterms associated to trees.

\begin{lemma} \label{lem_treecterterm}
Let $I\in\intnoncommI{\mathcal{E},\mathcal{A}}$ be a family of interactions for a family $\mathcal{E}$ of free theories over $\mathcal{A}$. Suppose furthermore that the tree-level part $I_{001}$ of $I$ is a local interaction, and let $P(\varepsilon,L)$ be a family of propagators of the form \eqref{eqn_canonicalpropagator}.
\begin{enumerate}
\item \label{itm_treecterterm1}
Let $\Gamma$ be a connected stable ribbon graph such that:
\begin{itemize}
\item
the first Betti number $\mathbf{b}_1(\Gamma)$ vanishes, and
\item
the loop number of every vertex of $\Gamma$, except possibly one, vanishes.
\end{itemize}
Then for all $L>0$, the limit
\[ \lim_{\varepsilon\to 0} F_{\Gamma}(I,P(\varepsilon,L)) \]
exists in $\Hom_{\gf}(\mathcal{E}(\!(L(\Gamma))\!),\mathcal{A})$. Moreover, it then converges to zero as $L\to 0$, unless $\Gamma$ is a corolla.
\item \label{itm_treecterterm2}
The tree-level counterterms vanish,
\[ I_{001}^{\mathrm{CT}} = 0. \]
\end{enumerate}
\end{lemma}

\begin{proof}
Note that \eqref{itm_treecterterm2} follows from \eqref{itm_treecterterm1}, as then we may show that $I_{001l}^{\mathrm{CT}}$ vanishes for all~$l$ by a simple induction.

A basic fact about any connected graph $\Gamma$ with vanishing first Betti number, that is not a corolla, is that there are always at least two vertices that are only attached to the graph by a single edge. Hence any graph $\Gamma$ of the type described above is obtained by attaching a corolla $\Gamma_1$ with vanishing loop number to a smaller graph $\Gamma_2$ of the same type. The proof may therefore proceed by induction on the number of vertices in such a graph. Let $l$, $l_1$ and $l_2$ be the number of legs of the graphs $\Gamma$, $\Gamma_1$ and $\Gamma_2$ respectively.

Using a representative from the cyclic ordering on the legs of $\Gamma_1$, we regard $I(\Gamma_1)$ as a cyclically invariant map from $\mathcal{E}^{\cotimes l_1}$ to $\mathcal{A}$. Since the interaction $I_{001}$ is local, we may use integration by parts to write
\[ I(\Gamma_1)(s_1,\ldots,s_{l_1}) = \innprod[s_1,X(s_2,\ldots,s_{l_1})]_{\mathcal{A}}; \]
for some operator $X:\mathcal{E}^{\cotimes l_1 - 1}\to\mathcal{E}\cotimes\mathcal{A}$, see Remark 2.17 of \cite{NCRBV}. Next, choosing appropriate representatives for the cyclic orderings on the legs of $\Gamma$ and $\Gamma_2$, we may identify $F_{\Gamma}(I,P(\varepsilon,L))$ (up to a sign) with the map
\[ \mu_{\mathcal{A}}\bigl(F_{\Gamma_2}(I,P(\varepsilon,L))\cotimes\mathds{1}_{\mathcal{A}}\bigr)\bigl(\mathds{1}_{\mathcal{E}^{\cotimes l_2-1}}\cotimes(P(\varepsilon,L)\star X)\bigr)\in\Hom_{\gf}\bigl(\mathcal{E}^{\cotimes l},\mathcal{A}\bigr). \]

By the basic property of the heat kernel, we know that the operator $P(\varepsilon,L)\star$ converges pointwise as $\varepsilon\to 0$. By the inductive hypothesis, the operator $F_{\Gamma_2}(I,P(\varepsilon,L))$ also converges weakly as $\varepsilon\to 0$. Applying the Banach-Steinhaus Theorem and Proposition 32.5 of \cite{Treves}, it follows that the operators
\begin{displaymath}
\begin{split}
\mathds{1}_{\mathcal{E}^{\cotimes l_2-1}}\cotimes(P(\varepsilon,l)\star X) &\in \Hom_{\gf}\bigl(\mathcal{E}^{\cotimes l},\mathcal{E}^{\cotimes l_2}\cotimes\mathcal{A}\bigr) \\
F_{\Gamma_2}(I,P(\varepsilon,L))\cotimes\mathds{1}_{\mathcal{A}} &\in \Hom_{\gf}\bigl(\mathcal{E}^{\cotimes l_2}\cotimes\mathcal{A},\mathcal{A}\cotimes\mathcal{A}\bigr)
\end{split}
\end{displaymath}
both converge weakly. Now apply Lemma \ref{lem_compositionconverge}.

The same argument shows that if we then take the limit as $L\to 0$, the Feynman amplitude will vanish unless $\Gamma$ is a corolla.
\end{proof}

\subsection{Effective field theory}

We are now ready to introduce the notion of a \emph{noncommutative} effective field theory and to explain its relationship to the notion of a (commutative) effective field theory described by Costello in \cite{CosEffThy}. We will then prove an analogue of one of the main results of \cite{CosEffThy} connecting effective field theories to local interactions.

\subsubsection{The definition of an effective field theory}

We begin by introducing the definition of a noncommutative effective field theory following \cite{CosEffThy}. We will then briefly recall the definition from \cite[\S 2.13.4]{CosEffThy} of a (commutative) effective field theory.

\begin{defi} \label{def_NCeffthy}
Let $\mathcal{E}$ be a family of free theories over $\mathcal{A}$ and $P(\varepsilon,L)$ be a family of propagators of the form \eqref{eqn_canonicalpropagator}. Suppose that
\[ I\in\intnoncommI{\mathcal{E},\mathcal{A}}\underset{\mathbb{R}}{\cotimes}\smooth{0,\infty} \]
is an $L$-parameterized family of interactions, with $I[L]\in\intnoncommI{\mathcal{E},\mathcal{A}}$ denoting the interaction defined at the point $L>0$. Then we say that $I$ defines a \emph{noncommutative effective field theory} if:
\begin{itemize}
\item
the renormalization group equation
\[ I[L'] = W(I[L],P(L,L')) \]
holds for all $L,L'>0$ and
\item
the interactions are asymptotically local; that is for all $i,j,k,l\geq 0$ there is a small $L$ asymptotic expansion
\[ I_{ijkl}[L]\simeq\sum_{r=1}^{\infty}f_r(L)\Phi_r, \]
where $f_r\in\smooth{0,\infty}$ and $\Phi_r\in\intnoncommLI{\mathcal{E},\mathcal{A}}$ are local interactions.
\end{itemize}
We will denote the set of noncommutative effective field theories by
\[ \NCEffThy{\mathcal{E},\mathcal{A}}\subset\intnoncommI{\mathcal{E},\mathcal{A}}\underset{\mathbb{R}}{\cotimes}\smooth{0,\infty}. \]
\end{defi}

To specify more precisely the statement in Definition \ref{def_NCeffthy} that the interactions are asymptotically local; this means that there is a monotonically increasing sequence of integers $d_R$ tending to infinity, such that for all $R$
\[ \lim_{L\to 0}L^{-d_R}\biggl(I_{ijkl}[L]-\sum_{r=1}^R f_r(L)\Phi_r\biggr) = 0. \]

\begin{rem} \label{rem_effthy}
The definition of a (commutative) effective field theory is introduced in \cite{CosEffThy} in exactly the same way, with the only difference being that we replace $\intnoncomm{\mathcal{E},\mathcal{A}}$ everywhere in Definition \ref{def_NCeffthy} with $\intcomm{\mathcal{E},\mathcal{A}}$ and $I_{ijkl}[L]$ with $I_{ij}[L]$. We will denote the set of effective field theories by $\EffThy{\mathcal{E},\mathcal{A}}$.
\end{rem}

\subsubsection{Passing from noncommutative to commutative geometry}

One basic feature of a noncommutative effective field theory is that any noncommutative effective field theory defines a (commutative) effective field theory.

\begin{prop} \label{prop_NCtocomEffThy}
Let $\mathcal{E}$ be a family of free theories over $\mathcal{A}$ and $P(\varepsilon,L)$ be a family of propagators of the form \eqref{eqn_canonicalpropagator}. Then the map
\[ \sigma_{\gamma,\nu}:\intnoncommI{\mathcal{E},\mathcal{A}}\underset{\mathbb{R}}{\cotimes}\smooth{0,\infty}\to\intcommI{\mathcal{E},\mathcal{A}}\underset{\mathbb{R}}{\cotimes}\smooth{0,\infty} \]
of Definition \ref{def_mapNCtoCom} transforms noncommutative effective field theories into (commutative) effective field theories.
\end{prop}

\begin{proof}
This is a simple corollary of Theorem \ref{thm_flowNCtoCom}.
\end{proof}

It is also worth pointing out that an Open Topological Field Theory also defines a transformation of noncommutative effective field theories. We will see the importance of this feature when we come to discuss the large $N$ correspondence in Section \ref{sec_largeN}.

\begin{prop} \label{prop_OTFTEffThy}
Let $\mathcal{E}$ be a family of free theories over $\mathcal{A}$ and $P(\varepsilon,L)$ be a family of propagators of the form \eqref{eqn_canonicalpropagator}. Then for any Frobenius algebra $\textgoth{A}$, the map
\[ \Morita:\intnoncommI{\mathcal{E},\mathcal{A}}\underset{\mathbb{R}}{\cotimes}\smooth{0,\infty}\to\intnoncommI{\mathcal{E}_{\textgoth{A}},\mathcal{A}}\underset{\mathbb{R}}{\cotimes}\smooth{0,\infty} \]
of Definition \ref{def_OTFTtransformation} transforms noncommutative effective field theories for the propagator $P(\varepsilon,L)$ into noncommutative effective field theories for the propagator $P_{\textgoth{A}}(\varepsilon,L)$ defined by \eqref{eqn_OTFTpropagator}.
\end{prop}

\begin{proof}
Again, this is just a simple corollary of Theorem \ref{thm_flowOTFT}.
\end{proof}

\subsubsection{Renormalization and effective field theory} \label{sec_RenormEffThy}

Another basic feature of effective field theory is that effective field theories are produced by renormalizing the interaction, see \cite[\S 2.11]{CosEffThy}.

\begin{defi} \label{def_intrenormalized}
Let $\mathcal{E}$ be a family of free theories over $\mathcal{A}$ and $P(\varepsilon,L)$ be a family of propagators of the form \eqref{eqn_canonicalpropagator}. Given a local interaction $I\in\intnoncommLI{\mathcal{E},\mathcal{A}}$, we define a renormalized interaction
\[ I^\mathrm{R}[L] :=  \lim_{\varepsilon\to 0}W\bigl(I-I^{\mathrm{CT}}(\varepsilon),P(\varepsilon,L)\bigr), \quad L>0. \]
\end{defi}

Likewise, given a local interaction $I\in\intcommLI{\mathcal{E},\mathcal{A}}$, a renormalized interaction $I^\mathrm{R}[L]$ was defined in \cite{CosEffThy} by the same formula. This and the definition above make sense by Theorem \ref{thm_counterterms}, cf. Remark \ref{rem_CtermCos}. In \cite[\S 2.11]{CosEffThy} it was shown that the interactions $I^\mathrm{R}[L]$ defined there form a (commutative) effective field theory, cf. Remark \ref{rem_effthy}. We now formulate the corresponding analogue of this result in noncommutative effective field theory.

\begin{prop} \label{prop_renormalizedthy}
Let $I\in\intnoncommLI{\mathcal{E},\mathcal{A}}$ be a family of local interactions for a family $\mathcal{E}$ of free theories over $\mathcal{A}$ and let $P(\varepsilon,L)$ be a family of propagators of the form \eqref{eqn_canonicalpropagator}. Then the interactions
\[ I^\mathrm{R}[L]\in\intnoncommI{\mathcal{E},\mathcal{A}}, \quad L>0 \]
define a noncommutative effective field theory.
\end{prop}

\begin{proof}
The renormalization group equation is a consequence of Theorem \ref{thm_intgrpact} and the fact that the renormalization group flow is continuous. That the interactions are asymptotically local follows from Theorem \ref{thm_asymptotics}.
\end{proof}

This construction of effective field theories is compatible with the transformations described in Proposition \ref{prop_NCtocomEffThy} and \ref{prop_OTFTEffThy}.

\begin{prop} \label{prop_renormcommute}
Let $\mathcal{E}$ be a family of free theories over $\mathcal{A}$ and $P(\varepsilon,L)$ be a family of propagators of the form \eqref{eqn_canonicalpropagator}. Then the following diagram commutes:
\begin{displaymath}
\xymatrix{
\intnoncommLI{\mathcal{E},\mathcal{A}} \ar_{\sigma_{\gamma,\nu}}[d] \ar^{I\mapsto I^{\mathrm{R}}}[rr] && \NCEffThy{\mathcal{E},\mathcal{A}} \ar^{\sigma_{\gamma,\nu}}[d] \\
\intcommLI{\mathcal{E},\mathcal{A}} \ar^{I\mapsto I^{\mathrm{R}}}[rr] && \EffThy{\mathcal{E},\mathcal{A}}
}
\end{displaymath}
If in addition $\textgoth{A}$ is a Frobenius algebra, then the following diagram commutes:
\begin{displaymath}
\xymatrix{
\intnoncommLI{\mathcal{E},\mathcal{A}} \ar_{\Morita}[d] \ar^{I\mapsto I^{\mathrm{R}}}[rr] && \NCEffThy{\mathcal{E},\mathcal{A}} \ar^{\Morita}[d] \\
\intnoncommLI{\mathcal{E}_{\textgoth{A}},\mathcal{A}} \ar^{I\mapsto I^{\mathrm{R}}}[rr] && \NCEffThy{\mathcal{E}_{\textgoth{A}},\mathcal{A}}
}
\end{displaymath}
\end{prop}

\begin{proof}
This, again, is a simple consequence of Theorem \ref{thm_flowNCtoCom} and Theorem \ref{thm_flowOTFT}; see Remark \ref{rem_ctermtransform}.
\end{proof}

\subsubsection{Finite-level effective field theories}

In what follows, and in \cite{NCRBV}, we will need to consider manipulations of finite-level effective field theories. That is, we will define a certain filtration on our spaces of interactions and we will consider the process of building effective field theories step-by-step through the levels of our filtration. This will be of particular importance in \cite{NCRBV} when we apply these ideas to the Batalin-Vilkovisky formalism and determine the controlling obstruction theory for constructing a solution to the quantum master equation.

We will filter our spaces of interactions based on the loop number.

\begin{defi}
Let $\mathcal{E}$ be a family of free theories over $\mathcal{A}$. We define a complete Hausdorff decreasing filtration on $\intcomm{\mathcal{E},\mathcal{A}}$ by
\[ F_p\intcomm{\mathcal{E},\mathcal{A}} :=\hbar^p\intcomm{\mathcal{E},\mathcal{A}}, \quad p\geq 0. \]
If, given $I\in\intcomm{\mathcal{E},\mathcal{A}}$, we define
\[ I_{[i]}:=\hbar^i\sum_{j=0}^{\infty} I_{ij}, \]
then $F_p\intcomm{\mathcal{E},\mathcal{A}}$ is the subspace of $\intcomm{\mathcal{E},\mathcal{A}}$ consisting of all those functionals $I$ satisfying
\[ I_{[i]}=0, \quad \text{for all }0\leq i<p. \]

Given $I\in\intnoncomm{\mathcal{E},\mathcal{A}}$ define
\[ I_{[n]} := \sum_{\begin{subarray}{c} i,j,k\geq 0: \\ 2i+j+k-1 = n\end{subarray}} \gamma^i\nu^j I_{ijk} \quad\text{and}\quad I_{[n]l} := \sum_{\begin{subarray}{c} i,j,k\geq 0: \\ 2i+j+k-1 = n\end{subarray}} \gamma^i\nu^j I_{ijkl}. \]
Then we likewise define a complete Hausdorff decreasing filtration on $\intnoncomm{\mathcal{E},\mathcal{A}}$ by defining $F_p\intnoncomm{\mathcal{E},\mathcal{A}}$ to be the subspace of $\intnoncomm{\mathcal{E},\mathcal{A}}$ consisting of all those functionals $I$ satisfying
\[ I_{[n]}=0, \quad\text{for all } 0\leq n<p. \]
\end{defi}

\begin{rem} \label{rem_filter}
Note that $I_{[0]}=I_{001}$ is the tree-level part of an interaction $I\in\intnoncommI{\mathcal{E},\mathcal{A}}$ and that we may identify the space of tree-level interactions by
\[ \intnoncommItree{\mathcal{E},\mathcal{A}} = \intnoncommI{\mathcal{E},\mathcal{A}}/F_1 \intnoncommI{\mathcal{E},\mathcal{A}}. \]

Note also that the maps $\sigma_{\gamma,\nu}$ of Definition \ref{def_mapNCtoCom} and $\Morita$ of Definition \ref{def_OTFTtransformation} both respect the filtrations $F_p$ defined above using the loop number. More precisely, if $I'=\sigma_{\gamma,\nu}(I)$ then $I'_{[n]}=\sigma_{\gamma,\nu}(I_{[n]})$, and likewise for $\Morita$.
\end{rem}

Central to our discussion of finite-level theories will be the following formula involving the above filtration. Before introducing it, we make the following definition.

\begin{defi} \label{def_ptree}
We will say that a stable ribbon graph (or stable graph) $\Gamma$ is a \emph{$p$-tree} if:
\begin{itemize}
\item
it is connected,
\item
the first Betti number $\mathbf{b}_1(\Gamma)$ vanishes, and
\item
the loop number of every vertex of $\Gamma$ vanishes, except one, which must have loop number $p$.
\end{itemize}
We will denote the class of $p$-trees by $\ptree$.
\end{defi}

\begin{rem} \label{rem_filterformula}
Note that if $\Gamma\in\ptree$ and if $v\in V(\Gamma)$ is the vertex of loop number $p$ then by \eqref{eqn_valencyribbonconditions},
\[ |L(\Gamma)| \geq |v| + |V(\Gamma)|-1. \]
Hence, the valency of $v$ can be no larger than the number of legs of $\Gamma$, with equality if and only if $\Gamma$ is a corolla.
\end{rem}

Note that $p$-trees are more general objects than ordinary trees, which were defined by Definition \ref{def_tree}. In this sense, an ordinary tree is the same thing as a $0$-tree. Any $p$-tree will have loop number $p$.

\begin{prop} \label{prop_filterformula}
Let $\mathcal{E}$ be a family of free theories over $\mathcal{A}$ and $P$ be a family of propagators for $\mathcal{E}$. For every $I\in\intnoncommI{\mathcal{E},\mathcal{A}}$ and $J\in F_p\intnoncommI{\mathcal{E},\mathcal{A}}$ we have, providing that $p>0$, the formula
\[ W(I+J,P) = W(I,P) + \sum_{\Gamma\in\ptree} \frac{\gamma^{g(\Gamma)}\nu^{b(\Gamma)}}{|\Aut(\Gamma)|}w_{\Gamma}\bigl(I_{[0]}+J_{[p]},P\bigr) \mod F_{p+1}; \]
where the sum is taken over all those stable ribbon graphs which are $p$-trees.
\end{prop}

\begin{rem} \label{rem_filteranalogue}
The same formula holds when we replace $\intnoncomm{\mathcal{E},\mathcal{A}}$ with $\intcomm{\mathcal{E},\mathcal{A}}$, providing that we replace stable ribbon graphs with stable graphs and $\gamma^{g(\Gamma)}\nu^{b(\Gamma)}$ with $\hbar^p$, the proof being exactly the same.
\end{rem}

\begin{rem}
Note that as a corollary of Proposition \ref{prop_filterformula} and Equation \eqref{eqn_intgrpactsum}, as well as Proposition \ref{prop_treeflow}, we obtain the formula
\begin{multline} \label{eqn_ptreegrpact}
\sum_{\Gamma\in\ptree}\frac{\gamma^{g(\Gamma)}\nu^{b(\Gamma)}}{|\Aut(\Gamma)|}w_{\Gamma}\bigl(I_{[0]}+J_{[p]},P_1+P_2\bigr) = \\
\sum_{\Gamma\in\ptree}\frac{\gamma^{g(\Gamma)}\nu^{b(\Gamma)}}{|\Aut(\Gamma)|}w_{\Gamma}\biggl(W^{\mathrm{Tree}}\bigl(I_{[0]},P_1\bigr)+\sum_{\Gamma'\in\ptree}\frac{\gamma^{g(\Gamma')}\nu^{b(\Gamma')}}{|\Aut(\Gamma')|}w_{\Gamma'}\bigl(I_{[0]}+J_{[p]},P_1\bigr),P_2\biggr).
\end{multline}
\end{rem}

\begin{proof}[Proof of Proposition \ref{prop_filterformula}]
If $\Gamma$ is a stable ribbon graph then
\[ w_{\Gamma}(I+J,P) = w_{\Gamma}(I,P), \]
unless $\Gamma$ has a vertex $v$ with loop number at least $p$. If in addition the loop number of $\Gamma$ is no more than $p$, then it follows from \eqref{eqn_loopnumber} that the first Betti number $\mathbf{b}_1(\Gamma)$ vanishes, the loop number of $v$ is $p$, and the loop number of every other vertex of $\Gamma$ vanishes. Hence, $\Gamma$ is a $p$-tree and
\[ w_{\Gamma}(I+J,P) = w_{\Gamma}(I,P) + w_{\Gamma}\bigl(I_{[0]}+J_{[p]},P\bigr). \]
\end{proof}

The preceding formula implies in particular that the renormalization group flow $W(-,P)$ is well-defined modulo $F_p$. This allows us to introduce the notion of a level $p$ effective field theory.

\begin{defi}
Let $\mathcal{E}$ be a family of free theories over $\mathcal{A}$ and $P(\varepsilon,L)$ be a family of propagators of the form \eqref{eqn_canonicalpropagator}. If $p\geq 0$ then we say that
\[ I\in\bigl(\intnoncommI{\mathcal{E},\mathcal{A}}/F_{p+1}\intnoncommI{\mathcal{E},\mathcal{A}}\bigr)\underset{\mathbb{R}}{\cotimes}\smooth{0,\infty} \]
defines a \emph{level $p$ noncommutative effective field theory} if:
\begin{itemize}
\item
the renormalization group equation
\[ I[L'] = W(I[L],P(L,L')) \]
holds for all $L,L'>0$ and
\item
each $I_{ijkl}[L]$ is asymptotically local as $L\to 0$.
\end{itemize}

We will denote the set of level $p$ noncommutative effective field theories by
\[ \NCEffThyL{p}{\mathcal{E},\mathcal{A}}\subset\bigl(\intnoncommI{\mathcal{E},\mathcal{A}}/F_{p+1}\intnoncommI{\mathcal{E},\mathcal{A}}\bigr)\underset{\mathbb{R}}{\cotimes}\smooth{0,\infty}. \]
We will call a level zero noncommutative effective field theory a \emph{tree-level noncommutative effective field theory}.
\end{defi}

The definition of a level $p$ (commutative) effective field theory is the same, except that we replace $\intnoncomm{\mathcal{E},\mathcal{A}}$ everywhere with $\intcomm{\mathcal{E},\mathcal{A}}$ and $I_{ijkl}[L]$ with $I_{ij}[L]$. We will denote the set of level $p$ effective field theories by $\EffThyL{p}{\mathcal{E},\mathcal{A}}$.

Any noncommutative/commutative effective field theory determines level $p$ theories simply by projecting onto the quotient modulo $F_{p+1}$. Likewise, any level $(p+1)$ theory projects onto a level $p$ theory. We would like to describe the fiber of such a projection. Before we do so however, we will need the following lemma.

\begin{lemma} \label{lem_localintegration}
Let $\mathcal{E}$ be a family of free theories over $\mathcal{A}$ and $P(\varepsilon,L)$ be a family of propagators of the form \eqref{eqn_canonicalpropagator}. Suppose that
\[ I,I'\in\NCEffThyL{p}{\mathcal{E},\mathcal{A}} \]
are two level $p$ theories such that
\[ I=I'\mod F_p \qquad\text{and}\qquad I_{[p]q}=I_{[p]q}',\quad\text{for all }q<l. \]
Then there is a local interaction
\[ J\in F_p\intnoncommLI{\mathcal{E},\mathcal{A}}/F_{p+1}\intnoncommLI{\mathcal{E},\mathcal{A}} \]
such that
\[ I_{[p]l}[L] - I_{[p]l}'[L] = J, \quad\text{for all }L>0. \]
\end{lemma}

\begin{proof}
The argument is essentially the same as that given in \cite[\S 2.11]{CosEffThy}. Set
\[ J[L] := I_{[p]l}[L] - I_{[p]l}'[L]. \]
We must show that $J$ does not depend on $L$ and is local.

From the renormalization group equation, Proposition \ref{prop_filterformula} and Remark \ref{rem_filterformula} following it, as well as Equation \eqref{eqn_intgrpactzero}, we get;
\begin{displaymath}
\begin{split}
I_{[p]l}'[L_2] + J[L_2] &= W_{[p]l}(I[L_1],P(L_1,L_2)) \\
&= W_{[p]l}(I'[L_1],P(L_1,L_2)) + I_{[p]l}[L_1]-I_{[p]l}'[L_1] \\
&=  I_{[p]l}'[L_2] + J[L_1].
\end{split}
\end{displaymath}
Note that it follows from Proposition \ref{prop_orderformula} that the equation still holds when $p=0$. Since $J$ is asymptotically local and does not depend on $L$, it follows that $J$ must be local as claimed.
\end{proof}

Next, we will need the following description of the tree-level theories.

\begin{prop} \label{prop_treelevelthy}
Let $\mathcal{E}$ be a family of free theories over $\mathcal{A}$ and $P(\varepsilon,L)$ be a family of propagators of the form \eqref{eqn_canonicalpropagator}. Then for every tree-level theory
\[ I\in\NCEffThyL{0}{\mathcal{E},\mathcal{A}} \]
there is a unique tree-level local interaction $I^*\in\intnoncommItree{\mathcal{E},\mathcal{A}}$ such that
\begin{equation} \label{eqn_treelevelthy}
I[L] = \lim_{\varepsilon\to 0}W^{\mathrm{Tree}}(I^*,P(\varepsilon,L)).
\end{equation}
We then have $I[L]\to I^*$ as $L\to 0$.
\end{prop}

\begin{rem}
We should be careful to emphasize that the convergence of $I[L]$ as $L\to 0$ is something specific to \emph{tree-level} theories due to the absence of counterterms, and does not hold for effective field theories in general. We will denote the limit by
\[ I[0] := \lim_{L\to 0}I[L]. \]
\end{rem}

\begin{proof}[Proof of Proposition \ref{prop_treelevelthy}]
It is sufficient to show that $I^*$ exists, for if it does then it follows from Lemma \ref{lem_treecterterm}\eqref{itm_treecterterm1} and Equation \eqref{eqn_intgrpactzero} that $I[L]\to I^*$ as $L\to 0$, establishing its uniqueness. Also note that given any local interaction $I^*\in\intnoncommItree{\mathcal{E},\mathcal{A}}$, it follows from Proposition \ref{prop_renormalizedthy} and Lemma \ref{lem_treecterterm}\eqref{itm_treecterterm2} that Equation \eqref{eqn_treelevelthy} defines a tree-level theory.

We then use Lemma \ref{lem_localintegration} to construct, by induction on $l$, a series of local tree-level interactions
\[ I^*=\sum_{l=3}^{\infty}I^*_l\in\intnoncommItree{\mathcal{E},\mathcal{A}} \]
satisfying
\[ I^*_l = I_{[0]l}[L] -\lim_{\varepsilon\to 0}W_{[0]l}^{\mathrm{Tree}}\biggl(\sum_{r=3}^{l-1}I^*_r,P(\varepsilon,L)\biggr), \quad\text{for all }L>0. \]
It then follows from Proposition \ref{prop_orderformula} that $I^*$ satisfies Equation \eqref{eqn_treelevelthy}.
\end{proof}

\subsubsection{The main theorem: effective field theories and local interactions}

Now we are ready to describe the fiber of level (p+1) theories that sit over a level $p$ theory. The following result is the analogue in noncommutative effective field theory of one of the main results of \cite{CosEffThy}; cf. \S 1.1.2, Theorem 1.2.1.

\begin{prop} \label{prop_fiberaction}
Let $\mathcal{E}$ be a family of free theories over $\mathcal{A}$ and $P(\varepsilon,L)$ be a family of propagators of the form \eqref{eqn_canonicalpropagator}. Suppose that
\[ I\in\NCEffThyL{p}{\mathcal{E},\mathcal{A}} \]
is a level $p$ noncommutative effective field theory.
Denote by
\[ \mathscr{F}_{p+1}(I)\subset\NCEffThyL{p+1}{\mathcal{E},\mathcal{A}} \]
the fiber of level $(p+1)$ theories that sit over the level $p$ theory $I$. Then the abelian group
\[ \mathcal{G}_{p+1} := F_{p+1}\intnoncommLI{\mathcal{E},\mathcal{A}}/F_{p+2}\intnoncommLI{\mathcal{E},\mathcal{A}} \]
of local interactions acts freely and transitively on $\mathscr{F}_{p+1}(I)$ according to the formula
\begin{equation} \label{eqn_fiberaction}
\bigl(J\cdot\bar{I}\bigr)[L] := \bar{I}[L] + \sum_{\Gamma\in\ptree[p+1]} \frac{\gamma^{g(\Gamma)}\nu^{b(\Gamma)}}{|\Aut(\Gamma)|} \lim_{\varepsilon\to 0}w_{\Gamma}\bigl(I_{[0]}[0]+J,P(\varepsilon,L)\bigr);
\end{equation}
where $J\in\mathcal{G}_{p+1}$, $\bar{I}\in\mathscr{F}_{p+1}(I)$ and the sum is taken over all those stable ribbon graphs which are $(p+1)$-trees.
\end{prop}

\begin{proof}
Note that it follows from Lemma \ref{lem_treecterterm}\eqref{itm_treecterterm1} that the expression in Equation \eqref{eqn_fiberaction} converges as $\varepsilon\to 0$. It follows from Proposition \ref{prop_filterformula} and Equation \eqref{eqn_ptreegrpact} applied to $I_{[0]}[0]$ and $J$ that \eqref{eqn_fiberaction} satisfies the renormalization group equation. Asymptotic locality of the effective actions follows from Theorem \ref{thm_asymptotics}.

We begin by noting that it follows from Remark \ref{rem_filterformula} that for every nonnegative integer $l$ and all $J\in\mathcal{G}_{p+1}$, $\bar{I}\in\mathscr{F}_{p+1}(I)$;
\begin{equation} \label{eqn_recursivefiberformula}
\bigl(J\cdot\bar{I}\bigr)_{[p+1]l}[L] = J_{[p+1]l} + \biggl(\sum_{r=0}^{l-1} J_{[p+1]r}\cdot\bar{I}\biggr)_{[p+1]l}[L], \quad\text{for all } L>0.
\end{equation}
It then follows by a simple induction on $l$ that if $J\in\mathcal{G}_{p+1}$ fixes the effective action, then $J_{[p+1]l}=0$ for all $l\geq 0$, in which case $J=0$.

Similarly, given $\bar{I},\bar{I}'\in\mathscr{F}_{p+1}(I)$ it follows from Equation \eqref{eqn_recursivefiberformula} that we may use Lemma \ref{lem_localintegration} and induction on $l$ to construct a series of local functionals
\[ J = \sum_{l=0}^\infty J_l \in\mathcal{G}_{p+1} \]
satisfying;
\[ J_l = \bar{I}'_{[p+1]l}[L] - \biggl(\sum_{r=0}^{l-1} J_r\cdot\bar{I}\biggr)_{[p+1]l}[L], \quad\text{for all }L>0. \]
It then follows from Equation \eqref{eqn_recursivefiberformula} again that $J\cdot\bar{I}=\bar{I}'$.
\end{proof}

The formulation of the preceding proposition did not rely in any way on the construction of counterterms and hence is completely independent of the choice of a renormalization scheme. We now turn our attention to proving one of the main theorems of the paper, for which such counterterms will be required.

\begin{lemma} \label{lem_filtercterms}
Let $\mathcal{E}$ be a family of free theories over $\mathcal{A}$ and $P(\varepsilon,L)$ be a family of propagators of the form \eqref{eqn_canonicalpropagator}. If $p$ is a nonnegative integer and
\[ I,I'\in\intnoncommLI{\mathcal{E},\mathcal{A}} \]
are two local interactions that agree modulo $F_{p+1}$, then their corresponding counterterms $I^{\mathrm{CT}}$ and $I'^{\mathrm{CT}}$ agree modulo $F_{p+2}$.
\end{lemma}

\begin{proof}
We must show that
\[ I^{\mathrm{CT}}_{[q]l}=I'^{\mathrm{CT}}_{[q]l}, \quad\text{for all }l\geq 0\text{ and }q\leq p+1.\]
We do this by induction on the lexicographical well-ordering of the pair $(q,l)$, so assume that the above holds for all smaller pairs in the lexicographical order. Note that by Lemma \ref{lem_treecterterm}\eqref{itm_treecterterm2} we may assume that $q>0$. Denote by $\ptree[ql]$, the class of $q$-trees that have $l$ legs. Using Proposition \ref{prop_filterformula} and Remark \ref{rem_filterformula} we have
\begin{displaymath}
\begin{split}
W_{[q]l}\bigl(I-I^{\mathrm{CT}}(\varepsilon),P(\varepsilon,L)\bigr) =& W_{[q]l}\biggl(\sum_{q'<q}I_{[q']} - I_{[q']}^{\mathrm{CT}}(\varepsilon),P(\varepsilon,L)\biggr) - I_{[q]l}^{\mathrm{CT}}(\varepsilon) \\
&+ \sum_{\Gamma\in\ptree[ql]}\frac{\gamma^{g(\Gamma)}\nu^{b(\Gamma)}}{|\Aut(\Gamma)|} w_{\Gamma}\biggl(I_{[0]}+I_{[q]}-\sum_{r<l}I_{[q]r}^{\mathrm{CT}}(\varepsilon),P(\varepsilon,L)\biggr),
\end{split}
\end{displaymath}
where we have again made use of Lemma \ref{lem_treecterterm}\eqref{itm_treecterterm2}.

By Theorem \ref{thm_counterterms} we know that the above expression converges as $\varepsilon\to 0$. Of course, the same formula holds for the interaction $I'$. Taking the difference of these two expressions for the interactions $I$ and $I'$ and applying the inductive hypothesis we get
\[ \sum_{\Gamma\in\ptree[ql]}\frac{\gamma^{g(\Gamma)}\nu^{b(\Gamma)}}{|\Aut(\Gamma)|} w_{\Gamma}\bigl(I_{[0]}+I_{[q]}-I'_{[q]},P(\varepsilon,L)\bigr) - \bigl(I_{[q]l}^{\mathrm{CT}}(\varepsilon)-I'^{\mathrm{CT}}_{[q]l}(\varepsilon)\bigr), \]
which of course is still convergent. Now, by Lemma \ref{lem_treecterterm}\eqref{itm_treecterterm1} the sum on the left converges as $\varepsilon\to 0$, and hence the difference between the counterterms will also converge. Since the counterterms are purely singular, they must be equal, completing the induction.
\end{proof}

Note that it follows from the preceding lemma that Definition \ref{def_intrenormalized} and Proposition \ref{prop_renormalizedthy} provide a well-defined  map from local interactions modulo $F_{p+1}$ to level $p$ effective field theories.

\begin{lemma} \label{lem_liftgrpequivariance}
Let $\mathcal{E}$ be a family of free theories over $\mathcal{A}$ and $P(\varepsilon,L)$ be a family of propagators of the form \eqref{eqn_canonicalpropagator}. Given a nonnegative integer $p$, consider the commutative diagram
\begin{displaymath}
\xymatrix{
\intnoncommLI{\mathcal{E},\mathcal{A}}/F_{p+2}\intnoncommLI{\mathcal{E},\mathcal{A}} \ar^-{I\mapsto I^{\mathrm{R}}}[rr] \ar[d] && \NCEffThyL{p+1}{\mathcal{E},\mathcal{A}} \ar[d] \\
\intnoncommLI{\mathcal{E},\mathcal{A}}/F_{p+1}\intnoncommLI{\mathcal{E},\mathcal{A}} \ar^-{I\mapsto I^{\mathrm{R}}}[rr] && \NCEffThyL{p}{\mathcal{E},\mathcal{A}}
}
\end{displaymath}
and the group
\[ \mathcal{G}_{p+1} := F_{p+1}\intnoncommLI{\mathcal{E},\mathcal{A}}/F_{p+2}\intnoncommLI{\mathcal{E},\mathcal{A}} \]
which acts freely and transitively on the fibers of the vertical projection maps in the above diagram. Then the top horizontal map of this diagram is $\mathcal{G}_{p+1}$-equivariant.
\end{lemma}

\begin{proof}
It follows from Lemma \ref{lem_filtercterms}, Proposition \ref{prop_filterformula}, Lemma \ref{lem_treecterterm}\eqref{itm_treecterterm2} and Proposition \ref{prop_treelevelthy} that
\begin{displaymath}
\begin{split}
(I+J)^{\mathrm{R}}[L] &= I^{\mathrm{R}}[L]  + \sum_{\Gamma\in\ptree[p+1]}\frac{\gamma^{g(\Gamma)}\nu^{b(\Gamma)}}{|\Aut(\Gamma)|} \lim_{\varepsilon\to 0}w_{\Gamma}\bigl(I_{[0]}+J,P(\varepsilon,L)\bigr) \\
&= (J\cdot I^{\mathrm{R}})[L],
\end{split}
\end{displaymath}
for all $I\in\intnoncommLI{\mathcal{E},\mathcal{A}}/F_{p+2}\intnoncommLI{\mathcal{E},\mathcal{A}}$ and $J\in\mathcal{G}_{p+1}$.
\end{proof}

Finally, we are ready to formulate and prove our main theorem, which is the corresponding analogue in noncommutative geometry of Theorem 13.4.3 of \cite[\S 2.13.4]{CosEffThy}.

\begin{theorem} \label{thm_locinteffthy}
Let $\mathcal{E}$ be a family of free theories over $\mathcal{A}$ and $P(\varepsilon,L)$ be a family of propagators of the form \eqref{eqn_canonicalpropagator}.
\begin{enumerate}
\item \label{itm_locinteffthy1}
There is a one-to-one correspondence between local interactions modulo $F_{p+1}$ and level $p$ effective field theories;
\[ \intnoncommLI{\mathcal{E},\mathcal{A}}/F_{p+1}\intnoncommLI{\mathcal{E},\mathcal{A}}\longrightarrow\NCEffThyL{p}{\mathcal{E},\mathcal{A}}, \qquad I\mapsto I^{\mathrm{R}}. \]
\item \label{itm_locinteffthy2}
There is a one-to-one correspondence between local interactions and effective field theories;
\[ \intnoncommLI{\mathcal{E},\mathcal{A}}\longrightarrow\NCEffThy{\mathcal{E},\mathcal{A}}, \qquad I\mapsto I^{\mathrm{R}}. \]
\end{enumerate}
\end{theorem}

\begin{proof}
The proof of \eqref{itm_locinteffthy1} follows from Lemma \ref{lem_liftgrpequivariance} by induction on $p$. Note that the base case $p=0$ is taken care of by Proposition \ref{prop_treelevelthy} and Lemma \ref{lem_treecterterm}\eqref{itm_treecterterm2}.

Now \eqref{itm_locinteffthy2} merely follows as a formal consequence of \eqref{itm_locinteffthy1}:
\begin{itemize}
\item
If $I$ and $I'$ are two local interactions such that the effective field theories $I^{\mathrm{R}}$ and $I'^{\mathrm{R}}$ agree, then in particular they agree modulo $F_{p+1}$, and by \eqref{itm_locinteffthy1} it follows that the interactions $I$ and $I'$ agree modulo $F_{p+1}$. Since this holds for all $p$ and the filtration is Hausdorff, it follows that $I=I'$.
\item
Further, if $I'\in\NCEffThy{\mathcal{E},\mathcal{A}}$ is an effective field theory then by \eqref{itm_locinteffthy1} it follows that there is a sequence
\[ I_p\in\intnoncommLI{\mathcal{E},\mathcal{A}}/F_{p+1}\intnoncommLI{\mathcal{E},\mathcal{A}}, \quad p\geq 0 \]
of local interactions such that $I_p^{\mathrm{R}}$ and $I'$ agree modulo $F_{p+1}$.

Using \eqref{itm_locinteffthy1} again it follows that $I_{p+1}$ and $I_p$ agree modulo $F_{p+1}$. Since the filtration is complete, there is a local interaction $I\in\intnoncommLI{\mathcal{E},\mathcal{A}}$ such that $I$ and $I_p$ agree modulo $F_{p+1}$ for all $p\geq 0$. Hence, $I^{\mathrm{R}}$ and $I'$ agree modulo $F_{p+1}$ for all $p$, and since the filtration is Hausdorff it follows that $I^{\mathrm{R}}=I'$.
\end{itemize}
\end{proof}

\begin{rem}
By Remark \ref{rem_filteranalogue} there is an analogue of Proposition \ref{prop_filterformula} for $\intcomm{\mathcal{E},\mathcal{A}}$. We note that all the arguments provided above for noncommutative effective field theories could be equally well applied to (commutative) effective field theories, where the corresponding results appear in \cite{CosEffThy}.
\end{rem}

\section{Large $N$ phenomena} \label{sec_largeN}

In this section we will explain how noncommutative effective field theories provide a framework for the exploration of large $N$ limits in effective field theories, and how this is connected to the string-gauge theory paradigm. We should be careful to explain what we mean by the latter, for here we emphasize that we are not explicitly referring to some of the more specific formulations of this principle; such as the AdS/CFT correspondence \cite{AdSCFT}, the Gopakumar-Vafa conjecture \cite{GopVafa}, or the holographic principle \cite{MAGOO}. Instead, we are referring more broadly to the general phenomena discovered by 't~Hooft in \cite{tHooftplanar}, that when calculations in $U(N)$ gauge theories are expanded in powers of the rank $N$, the Feynman diagram expansion arranges itself into open string diagrams in which planar diagrams dominate.

\subsection{The string/gauge theory correspondence} \label{sec_StrGaugeCorr}

We have explained in Section \ref{sec_RGflow} how the renormalization group flow in noncommutative effective field theory is defined through the use of Feynman amplitudes of stable ribbon graphs. These stable ribbon graphs parameterize orbi-cells in certain compactifications of the moduli space of Riemann surfaces \cite{KontAiry}, \cite{LooCompact}; that is, they describe diagrams of open string interactions in which the worldsheet is a Riemann surface with marked points that fix the endpoints of the open strings. In other words, a noncommutative effective field theory is a type of open string theory.

Now suppose that $\mathcal{E}$ is a family of free theories over $\mathcal{A}$ and $\textgoth{A}:=\mat{N}{\mathbb{\gf}}$ is a matrix algebra. Let $P(\varepsilon,L)$ be a family of propagators of the form \eqref{eqn_canonicalpropagator} and $P_{\textgoth{A}}(\varepsilon,L)$ be defined by \eqref{eqn_OTFTpropagator}. Consider the maps $\sigma_{\gamma,\nu}$ and $\Morita$ defined by Definition \ref{def_mapNCtoCom} and Definition \ref{def_OTFTtransformation} respectively. Combining these maps yields the following diagram of effective field theories:
\begin{equation} \label{dig_largeN}
\xymatrix{
& \NCEffThy{\mathcal{E},\mathcal{A}} \ar^<<<<<<<<{\ldots}_>>>>>>>{\sigma_{\gamma,\nu}\circ\Morita[\mat{1}{\gf}]}[ld] \ar^>>>{\, \sigma_{\gamma,\nu}\circ\Morita[\mat{N}{\gf}]}[d] \ar^>>>>>>{\qquad \sigma_{\gamma,\nu}\circ\Morita[\mat{N+1}{\gf}] \ \ldots}[rd] \\
\EffThy{\mathcal{E}_{\mat{1}{\gf}},\mathcal{A}}\ldots & \EffThy{\mathcal{E}_{\mat{N}{\gf}},\mathcal{A}} & \EffThy{\mathcal{E}_{\mat{N+1}{\gf}},\mathcal{A}}\ldots
}
\end{equation}
That is, any noncommutative effective field theory on $\mathcal{E}$ gives rise to effective field theories on $\mathcal{E}_{\mat{N}{\mathbb{\gf}}}$ at all ranks $N$. Identifying the former as a type of open string theory provides a concrete realization of the principle that connects string theories to large $N$ limits of gauge theories. It follows from the formula in Example \ref{exm_matrixOTFT} that under the above correspondence, the formal variable $\nu$ is mapped to $N\hbar$ while $\gamma$ is mapped to $\hbar^2$. Hence, at a fixed order in $\hbar$, the order in $N$ is maximized when the order in $\gamma$ is zero. In the stable ribbon graph Feynman diagram expansion, this corresponds to surfaces of genus zero; that is, planar diagrams.

We should emphasize that there are no horizontal arrows in Diagram \eqref{dig_largeN}. There is, in principle, no way to pass between the effective field theories at different ranks $N$. It is only the extra structure present in the noncommutative geometry that is encoded on the space of interactions that makes the above diagram possible. We will see that the correspondence \eqref{dig_largeN} may be made even more concrete through the application of the Loday-Quillen-Tsygan Theorem, which will allow us to draw conclusions about the noncommutative effective field theory using the properties of the rank $N$ effective field theories.

\subsection{Application of the Loday-Quillen-Tsygan Theorem}

The Loday-Quillen-Tsygan Theorem is a result that arose in algebraic $K$-theory as a result of attempts to compute the homology of matrix Lie algebras, cf. \cite[Ch. 10]{LodayCyclicHomology}. It was proved independently by Loday-Quillen in \cite{LodayQuillen} and Tsygan in \cite{Tsygan}. It arises as a consequence of the invariant theory for the general linear group. In this section we will make use of that part of the Loday-Quillen-Tsygan Theorem which relies on this invariant theory, and use it to provide a necessary and sufficient condition for the vanishing of terms in the noncommutative effective field theory. This result will be an appropriate analogue of Lemma 7.2 of \cite{GiGqHaZeLQT}, and will be proved in essentially the same way.

\subsubsection{Invariant theory from the Loday-Quillen-Tsygan Theorem}

We now recall the relevant parts that we will need from the Loday-Quillen-Tsygan Theorem. Let $\mathcal{E}$ be the space of sections of a vector bundle over a smooth manifold and consider the matrix algebra $\mat{N}{\gf}$. Recall from Definition \ref{def_OTFTtransformation} that we have a map
\[ \csalg{\KontHamP{\mathcal{E}}}\longrightarrow\csalgdelim{(\KontHamPdelim{(\mathcal{E}_{\mat{N}{\gf}})})}, \qquad \lceil x\rceil^{\mathbf{r}} \longmapsto \lceil x\otimes\OTFT{0,0}{\mathbf{r}}\rceil^{\mathbf{r}}; \]
where we remind the reader:
\begin{itemize}
\item
that $x\in(\mathcal{E}^{\dag})^{\cotimes l}$ and $r_1,\ldots,r_k$ is a list of $k$ positive integers whose sum is $l$,
\item
that $\lceil x\rceil^{\mathbf{r}}$ denotes the image of $x$ under the map \eqref{eqn_attachmapscycdecompoff} determined by the cyclic decomposition $C_{\mathbf{r}}$ defined by \eqref{eqn_cycdecpartition},
\item
that $\OTFT{0,0}{\mathbf{r}}$ is the multilinear map on matrices
\[ A_{11},\ldots,A_{1r_1};\ldots;A_{k1},\ldots,A_{kr_k} \longmapsto \Tr(A_{11}\cdots A_{1r_1})\cdots\Tr(A_{k1}\cdots A_{kr_k}) \]
defined in Example \ref{exm_matrixOTFT},
\item
and $\mathcal{E}_{\mat{N}{\gf}}$ was used to denote $\mathcal{E}\otimes\mat{N}{\gf}$.
\end{itemize}
Additionally, recall from Definition \ref{def_mapNCtoCom} that we also have a map
\[ \csalgdelim{\bigl(\KontHamPdelim{(\mathcal{E}_{\mat{N}{\gf}})}\bigr)}\longrightarrow\csalgdelim{\bigl(\mathcal{E}^{\dag}_{\mat{N}{\gf}}\bigr)}, \qquad x_1\cdots x_k \longmapsto \sigma(x_1)\cdots \sigma(x_k); \]
where $x_1,\ldots,x_k\in\KontHamPdelim{(\mathcal{E}_{\mat{N}{\gf}})}$.

Note that the Lie algebra $\gl{N}{\gf}$ acts on $\mat{N}{\gf}$ by the adjoint action, and hence on $\mathcal{E}_{\mat{N}{\gf}}$ and $\csalgdelim{(\mathcal{E}^{\dag}_{\mat{N}{\gf}})}$. Since the trace map on matrices is $\gl{N}{\gf}$-invariant, composing the maps defined above yields a family of maps
\[ \csalgdelim{\bigl(\KontHamP{\mathcal{E}}\bigr)}\longrightarrow\csalgdelim{\bigl(\mathcal{E}^{\dag}_{\mat{N}{\gf}}\bigr)}^{\gl{N}{\gf}}, \quad N\geq 1. \]

Finally, for every $N\geq 1$ we also have a map
\[ \csalgdelim{\bigl(\mathcal{E}^{\dag}_{\mat{N+1}{\gf}}\bigr)}^{\gl{N+1}{\gf}}\longrightarrow\csalgdelim{\bigl(\mathcal{E}^{\dag}_{\mat{N}{\gf}}\bigr)}^{\gl{N}{\gf}} \]
which is induced by the inclusion of the algebra $\mat{N}{\gf}$ into $\mat{N+1}{\gf}$, defined by inserting an $N$-by-$N$ matrix into the top left-hand corner of an $(N+1)$-by-$(N+1)$ matrix and setting the remaining entries to zero.

All told, combining the above maps we have the following commutative diagram:
\begin{equation} \label{dig_LQTThm}
\xymatrix{
& \csalgdelim{(\KontHamP{\mathcal{E}})} \ar[ld]^{\qquad\qquad\ldots\ldots} \ar[rd] \ar[rrd]^>>>>>>>>{\quad\ldots} \\
\csalgdelim{\bigl(\mathcal{E}^{\dag}_{\mat{1}{\gf}}\bigr)}^{\gl{1}{\gf}} & \ldots\ldots \ar[l] & \csalgdelim{\bigl(\mathcal{E}^{\dag}_{\mat{N}{\gf}}\bigr)}^{\gl{N}{\gf}} \ar[l] & \csalgdelim{\bigl(\mathcal{E}^{\dag}_{\mat{N+1}{\gf}}\bigr)}^{\gl{N+1}{\gf}}\ldots \ar[l]
}
\end{equation}

\begin{theorem} \label{thm_LQT}
Diagram \eqref{dig_LQTThm} establishes $\csalgdelim{(\KontHamP{\mathcal{E}})}$ as a projective limit
\[ \csalgdelim{(\KontHamP{\mathcal{E}})} = \varprojlim_N \Bigl[ \csalgdelim{\bigl(\mathcal{E}^{\dag}_{\mat{N}{\gf}}\bigr)}^{\gl{N}{\gf}}\Bigr]. \]
\end{theorem}

The proof of Theorem \ref{thm_LQT} follows the arguments of \cite[\S 10.2.10]{LodayCyclicHomology} with no significant modifications, making use of the invariant theory from Proposition 9.2.2 of the cited text. We therefore summarize these arguments below only briefly.

When $N\geq k$, we may use the invariant theory to identify
\begin{displaymath}
\begin{split}
\Bigl[\Bigl[\bigl(\mathcal{E}^{\dag}_{\mat{N}{\gf}}\bigr)^{\cotimes l}\Bigr]_{\sg{l}}\Bigr]^{\gl{N}{\gf}} &= \Bigl[\bigl(\mathcal{E}^{\dag}\bigr)^{\cotimes l}\otimes\Bigl[\bigl(\mat{N}{\gf}^{\dag}\bigr)^{\otimes l}\Bigr]^{\gl{N}{\gf}}\Bigr]_{\sg{l}} \\
&= \Bigl[\bigl(\mathcal{E}^{\dag}\bigr)^{\cotimes l}\otimes\gf[\sg{l}]\Bigr]_{\sg{l}} \\
&= \prod_{k=1}^l\biggl[\prod_{r_1+\cdots+r_k=l}\bigl[\bigl(\mathcal{E}^{\dag}\bigr)^{\cotimes r_1}\bigr]_{\cyc{r_1}}\cotimes\ldots\cotimes\bigl[\bigl(\mathcal{E}^{\dag}\bigr)^{\cotimes r_k}\bigr]_{\cyc{r_k}}\biggr]_{\sg{k}};
\end{split}
\end{displaymath}
where $\gf[\sg{l}]$, on which $\sg{l}$ acts by conjugation, is identified with the $\gl{N}{\gf}$-invariants by assigning the map $\OTFT{0,0}{\mathbf{r}}$ above to the permutation $C_{\mathbf{r}}\in\sg{l}$ that is defined by \eqref{eqn_cycdecpartition}. Then, on the final line, the two spaces are identified by applying the map \eqref{eqn_attachmapscycdecompoff} determined by the cyclic decomposition $C_{\mathbf{r}}$.

\subsubsection{A vanishing criteria} \label{sec_vanishing}

We now formulate a result that will allow us to draw conclusions about a noncommutative effective field theory from the knowledge of the large $N$ behavior of the corresponding effective field theories. This will shortly be applied in Section \ref{sec_CSthy}, as well as later in \cite{NCRBV}.

\begin{theorem} \label{thm_vanishing}
Let $\mathcal{E}$ be a family of free theories over $\mathcal{A}$. Then a functional $I\in\intnoncomm{\mathcal{E},\mathcal{A}}$ vanishes if and only if
\[ \sigma_{\gamma,\nu}\Morita[\mat{N}{\gf}](I)=0, \quad\text{for all } N\geq 1. \]
\end{theorem}

\begin{proof}
If the condition above holds then by Remark \ref{rem_filter} we have
\[ \sigma_{\gamma,\nu}\Morita[\mat{N}{\gf}](I_{[n]})=0; \quad\text{for all } n\geq 0,N\geq 1. \]
If we denote, temporarily, the vertical maps in the projective limit \eqref{dig_LQTThm} by $\pi_N$, then using the formula from Example \ref{exm_matrixOTFT} the above equation becomes
\[ \sum_{\begin{subarray}{c} i,j,k\geq 0: \\ 2i+j+k-1 = n\end{subarray}} N^j \pi_N(I_{ijk}) = 0; \quad\text{for all } n\geq 0,N\geq 1. \]
Passing to the left in Diagram \eqref{dig_LQTThm} we get
\[ \sum_{\begin{subarray}{c} i,j,k\geq 0: \\ 2i+j+k-1 = n\end{subarray}} N^j \pi_M(I_{ijk}) = 0; \quad\text{for all } n\geq 0,N\geq M\geq 1. \]

Since a nontrivial polynomial can have only finitely many roots, it follows that
\[ \sum_{\begin{subarray}{c} i,k\geq 0: \\ 2i+j+k-1 = n\end{subarray}} \pi_M(I_{ijk}) = 0; \quad\text{for all } n\geq 0, M\geq 1\text{ and }0\leq j\leq n+1. \]
Applying Theorem \ref{thm_LQT} we conclude that
\[ \sum_{\begin{subarray}{c} i,k\geq 0: \\ 2i+j+k-1 = n\end{subarray}} I_{ijk} = 0; \quad\text{for all } n\geq 0\text{ and }0\leq j\leq n+1. \]
From this it follows that $I=0$.
\end{proof}

We will also need a version of this result that applies modulo constants. The constants are those parts of $\intcomm{\mathcal{E}}$ and $\intnoncomm{\mathcal{E}}$ that are homogeneous of order zero in $\mathcal{E}^{\dag}$; that is, those parts that are contained in $\gf[[\hbar]]$ and $\gf[[\gamma,\nu]]$ respectively.

\begin{cor} \label{cor_vanishing}
Let $\mathcal{E}$ be a family of free theories over $\mathcal{A}$ and $I\in\intnoncomm{\mathcal{E},\mathcal{A}}$ be a functional. Then $I\in\gf[[\gamma,\nu]]\cotimes\mathcal{A}$ if and only if
\[ \sigma_{\gamma,\nu}\Morita[\mat{N}{\gf}](I)\in\gf[[\hbar]]\cotimes\mathcal{A}, \quad\text{for all } N\geq 1. \]
\end{cor}

\begin{proof}
This follows from Theorem \ref{thm_vanishing} since the maps $\sigma_{\gamma,\nu}$ and $\Morita[\mat{N}{\gf}]$ respect the grading by order in $\mathcal{E}^{\dag}$.
\end{proof}

\section{Noncommutative Chern-Simons theory} \label{sec_CSthy}

In this section we will consider a noncommutative analogue of Chern-Simons theory---that is, we will give a treatment of Chern-Simons theory as a noncommutative effective field theory and demonstrate that it produces $U(N)$ Chern-Simons theory at all ranks $N$ under the correspondence defined by Diagram \eqref{dig_largeN}.

Our construction of this noncommutative Chern-Simons theory will involve a very simple minded generalization of ordinary Chern-Simons theory. In this way it differs, for example, from the theory considered by Susskind in \cite[\S 7]{SussCS}, in which the Chern-Simons Lagrangian was deformed using a Moyal star-product. Instead, it treats Chern-Simons interactions through open string diagrams, and thus is more in keeping with the spirit of \cite{CosTCFT} and Witten's treatment of Chern-Simons theory in \cite{WitCSstring}.

We must of course acknowledge that, perhaps most significantly, Chern-Simons theory is a gauge theory. We do not treat the gauge theory aspects of this noncommutative theory here. Instead, this subject will be taken up in \cite{NCRBV}, where it will be treated in the framework of the Batalin-Vilkovisky formalism following \cite{CosEffThy}.

After constructing the promised theory, we will apply our string-gauge theory principle to conclude that the counterterms for our noncommutative Chern-Simons theory all vanish.

In this section we will not consider families of theories; that is, we will assume that $\mathcal{A}$ is the ground field $\gf$. Although we work with three-manifolds throughout this section, we emphasize that the results of this section hold for any manifold of odd dimension at least three, provided that we are willing to give up a $\mathbb{Z}$-grading for a $\mathbb{Z}/2\mathbb{Z}$-grading.

\subsection{Construction of the theory}

Before constructing our noncommutative theory, we review the commutative case for the gauge group $U(N)$. For this, we must briefly discuss complexification.

\subsubsection{Complexification} \label{sec_complexification}

Given a real topological vector space $\mathcal{V}$, we will denote its complexification by
\[ \mathcal{V}_{\mathbb{C}}:=\mathcal{V}\underset{\mathbb{R}}{\otimes}\mathbb{C}. \]
Complexification is a monoidal functor from real to complex topological vector spaces.

Consequently, it follows that if $\mathcal{E}$ is a family of free real theories over $\mathcal{A}$, then $\mathcal{E}_{\mathbb{C}}$ will be a family of free complex theories over $\mathcal{A}_{\mathbb{C}}$. The vector bundle, the local pairing and the generalized Laplacian are all obtained from applying this functor; that is, by extension of scalars. Likewise for the heat kernel and any propagator for this theory.

Furthermore, for any real topological vector space $\mathcal{V}$ we have
\[ \bigl[\mathcal{V}_{\mathbb{C}}\bigr]^{\dag} = \bigl[\mathcal{V}^{\dag}\bigr]_{\mathbb{C}}. \]
It therefore follows that the space of interactions for the complexification of a free theory is the complexification of its space of interactions;
\begin{align*}
\intcomm{\mathcal{E}_{\mathbb{C}},\mathcal{A}_{\mathbb{C}}} &= \intcomm{\mathcal{E},\mathcal{A}}_{\mathbb{C}}, & \intcommL{\mathcal{E}_{\mathbb{C}},\mathcal{A}_{\mathbb{C}}} &= \intcommL{\mathcal{E},\mathcal{A}}_{\mathbb{C}}; \\
\intnoncomm{\mathcal{E}_{\mathbb{C}},\mathcal{A}_{\mathbb{C}}} &= \intnoncomm{\mathcal{E},\mathcal{A}}_{\mathbb{C}}, & \intnoncommL{\mathcal{E}_{\mathbb{C}},\mathcal{A}_{\mathbb{C}}} &= \intnoncommL{\mathcal{E},\mathcal{A}}_{\mathbb{C}}.
\end{align*}

We therefore conclude from the above that we may equivalently treat a real theory whose interactions are complex-valued, as a complex theory over its complexification.

\subsubsection{The commutative case}

Here we recall the framework of $U(N)$ Chern-Simons theory. Suppose that $M$ is a compact oriented Riemannian three-manifold and consider the free Chern-Simons theory defined by Example \ref{exm_freeCSthy}, where we use the real Lie algebra $\un{N}$ of $N$-by-$N$ skew-Hermitian matrices with its negative-definite trace pairing. This has space of fields
\[ \mathcal{E}^{U(N)}=\Sigma\Omega^{\bullet}(M,\mathbb{R})\otimes\un{N} \]
and propagator
\[ P^{U(N)}(\varepsilon,L) := \left(\int_{t=\varepsilon}^L\mathrm{d}t\cotimes Q^*\cotimes\mathds{1}\right)[K] \]
given by Example \ref{exm_CSpropagator}, where $Q^*$ is the Hodge adjoint of the exterior derivative $Q$ on $\mathcal{E}^{U(N)}$ and $K$ is the heat kernel for the free theory.

The Chern-Simons interaction is defined as follows. Consider the multilinear map
\begin{equation} \label{eqn_CSinteraction}
\bigl(\mathcal{E}^{U(N)}\bigr)^{\otimes 3}\to\mathbb{C}, \qquad A_1,A_2,A_3 \longmapsto (-1)^{|A_2|}\frac{1}{3}\int_M\Tr(A_1\wedge A_2\wedge A_3).
\end{equation}
Then we define the (commutative) interaction for the $U(N)$ theory
\[ I^{U(N)}\in\intcommLIdelim{\bigl(\mathcal{E}^{U(N)}\bigr)}_{\mathbb{C}}=\intcommLIdelim{\bigl(\mathcal{E}^{U(N)}_{\mathbb{C}}\bigr)} \]
to be the interaction such that $I^{U(N)}_{ij}$ vanishes for all $i$ and $j$, except $I^{U(N)}_{03}$, which is represented by \eqref{eqn_CSinteraction}.

Following the principle outlined in Section \ref{sec_complexification}, we treat this complex-valued real theory as a theory over the complexification of the free theory $\mathcal{E}^{U(N)}$. Using the fact that the complexification of the real Lie algebra $\mathfrak{u}_N$ is $\gl{N}{\mathbb{C}}$, we see that the complexification of $\mathcal{E}^{U(N)}$ is the free complex Chern-Simons theory
\[ \mathcal{E}^{GL_N(\mathbb{C})}=\Sigma\Omega^{\bullet}(M,\mathbb{C})\otimes\gl{N}{\mathbb{C}} \]
defined by Example \ref{exm_freeCSthy}, where we use the trace pairing on $\gl{N}{\mathbb{C}}$. The complexification of the propagator $P^{U(N)}$ is the propagator $P^{GL_N(\mathbb{C})}$ defined by Example \ref{exm_CSpropagator} for the Chern-Simons theory with Lie algebra $\gl{N}{\mathbb{C}}$.

\subsubsection{The noncommutative case}

To define our noncommutative analogue of Chern-Simons theory, we consider the free theory whose vector bundle $E$ is the shifted complexified exterior algebra on the cotangent space so that
\[ \mathcal{E}^{\mathrm{NC}}=\Sigma\Omega^{\bullet}(M,\mathbb{C}). \]
This space has the usual integration pairing and Hodge-Laplacian determined by the orientation and Riemannian metric on $M$. Note that this is just the free $GL_1(\mathbb{C})$ Chern-Simons theory. Likewise, the propagator $P^{\mathrm{NC}}$ for our noncommutative theory is just the propagator $P^{GL_1(\mathbb{C})}$.

To define the noncommutative interaction, consider the multilinear map
\begin{equation} \label{eqn_NCCSinteraction}
\bigl(\mathcal{E}^{\mathrm{NC}}\bigr)^{\otimes 3}\to\mathbb{C}, \qquad \omega_1,\omega_2,\omega_3 \longmapsto (-1)^{|\omega_2|}\frac{1}{3}\int_M(\omega_1\wedge \omega_2\wedge \omega_3).
\end{equation}
We define the noncommutative interaction
\[ I^{\mathrm{NC}}\in\intnoncommLI{\mathcal{E}^{\mathrm{NC}}} \]
to be the interaction such that $I^{\mathrm{NC}}_{ijkl}$ vanishes for all $i$, $j$, $k$ and $l$; except $I^{\mathrm{NC}}_{0013}\in\KontHamP{\mathcal{E}^{\mathrm{NC}}}$, which is represented by \eqref{eqn_NCCSinteraction}.

\subsection{The correspondence}

Consider the construction of Example \ref{exm_OTFTfreethy} applied to the free theory $\mathcal{E}^{\mathrm{NC}}$ and Frobenius algebra $\mat{N}{\mathbb{C}}$. The resulting theory will be the free $GL_N(\mathbb{C})$ Chern-Simons theory,
\[ \mathcal{E}^{\mathrm{NC}}_{\mat{N}{\mathbb{C}}} = \mathcal{E}^{GL_N(\mathbb{C})}. \]
The propagator $P^{\mathrm{NC}}_{\mat{N}{\mathbb{C}}}$ defined by \eqref{eqn_OTFTpropagator} will be the propagator $P^{GL_N(\mathbb{C})}$.

Consider the maps
\[ \sigma_{\gamma,\nu}\circ\Morita[\mat{N}{\mathbb{C}}]:\intnoncomm{\mathcal{E}^{\mathrm{NC}}}\to\intcomm{\mathcal{E}^{GL_N(\mathbb{C})}} ,\quad N\geq 1. \]
Using the formula from Example \ref{exm_matrixOTFT} we see that
\begin{equation} \label{eqn_intNCtoUN}
\sigma_{\gamma,\nu}\Morita[\mat{N}{\mathbb{C}}](I^{\mathrm{NC}}) = I^{U(N)}, \quad\text{for all }N\geq 1.
\end{equation}

Let $(I^{\mathrm{NC}})^{\mathrm{R}}$ denote the noncommutative effective field theory associated to $I^{\mathrm{NC}}$ by Proposition \ref{prop_renormalizedthy}. Likewise, denote by $(I^{U(N)})^{\mathrm{R}}$ the effective field theory associated to the local interaction $I^{U(N)}$. It follows from the foregoing that our noncommutative Chern-Simons theory yields the $U(N)$ Chern-Simons theories at all ranks $N$ under the correspondence outlined in Section \ref{sec_StrGaugeCorr}.

\begin{prop}
The noncommutative Chern-Simons theory $(I^{\mathrm{NC}})^{\mathrm{R}}$ yields the $U(N)$ Chern-Simons theories $(I^{U(N)})^{\mathrm{R}}$ at all ranks $N$ under the correspondence defined by Diagram \eqref{dig_largeN}.
\end{prop}

\begin{proof}
This follows immediately from Equation \eqref{eqn_intNCtoUN} and Proposition \ref{prop_renormcommute}.
\end{proof}

\begin{rem}
Although we have previously mentioned that we will not address the gauge theory aspects of our theory here, we would be remiss if we did not point out that it is ordinarily necessary to quantize gauge theories such as Chern-Simons theory by perturbing the interaction. Nonetheless, it was shown in \cite{CosBVrenormalization} that this is not necessary in the presence of a flat metric, and this was used to analyze Chern-Simons theory on any manifold of dimension at least two.
\end{rem}

\subsection{Vanishing of the counterterms}

As a fairly simple application of the vanishing criteria outlined in Section \ref{sec_vanishing}, we will show that the counterterms for our noncommutative Chern-Simons theory vanish on a flat manifold. Here, we rely on the results of Costello \cite{CosBVrenormalization}, which in turn rely on the results of Kontsevich \cite{KontFeyn}. The analysis of the counterterms in Chern-Simons theory is a nontrivial problem, cf. \cite{AxelSingI, AxelSingII}, carried out in \cite{CosBVrenormalization, KontFeyn} by compactifying the configuration spaces of points using the methods of \cite{FuMcCompact}. Hence, it is convenient that this result may be deduced as a consequence of the correspondence \eqref{eqn_intNCtoUN}, without having to repeat any of this same detailed analysis.

\begin{prop} \label{prop_nocterms}
Let $M$ be a flat compact oriented Riemannian three-manifold and let
\[ (I^{\mathrm{NC}})^{\mathrm{CT}}(\varepsilon)\in\intnoncomm{\mathcal{E}^{\mathrm{NC}}} \]
denote the counterterms associated to the corresponding noncommutative Chern-Simons theory. Then these counterterms vanish modulo constants; that is, for all $\varepsilon>0$;
\[ (I^{\mathrm{NC}})^{\mathrm{CT}}(\varepsilon)\in\mathbb{C}[[\gamma,\nu]]. \]
\end{prop}

\begin{proof}
It follows from Theorem 15.2.1 of \cite{CosBVrenormalization} and Proposition 14.10.1(4) of \cite[\S 2.14.10]{CosEffThy} that the counterterms for $U(N)$ Chern-Simons theory all vanish modulo constants in a flat metric, that is
\[ (I^{U(N)})^{\mathrm{CT}}(\varepsilon)\in\mathbb{C}[[\hbar]], \quad\text{for all }N\geq 1. \]
It now follows from Remark \ref{rem_ctermtransform} and Equation \eqref{eqn_intNCtoUN} that for all $N\geq 1$,
\[ \sigma_{\gamma,\nu}\Morita[\mat{N}{\mathbb{C}}]\bigl((I^{\mathrm{NC}})^{\mathrm{CT}}(\varepsilon)\bigr) = (I^{U(N)})^{\mathrm{CT}}(\varepsilon)\in\mathbb{C}[[\hbar]]. \]
Now apply the vanishing criteria from Corollary \ref{cor_vanishing}.
\end{proof}

\begin{rem}
Again, we emphasize that Proposition \ref{prop_nocterms} holds for any manifold $M$ of odd dimension $n\geq 3$, provided that we switch to a $\mathbb{Z}/2\mathbb{Z}$-grading.

Stronger results may be derived using our vanishing criteria, but lie outside the scope of the current paper. For instance, we may show using the results of \cite{CosBVrenormalization} that the noncommutative Chern-Simons action on a flat manifold satisfies the noncommutative analogue of the quantum master equation, see \cite[\S 7.3]{NCRBV}. Moreover, in \cite{CosBVrenormalization} this fact was used to demonstrate that Chern-Simons theory has a canonical quantization on \emph{any} manifold. To do so, it was necessary to show that effective field theories form a sheaf over the manifold. This topic too, lies outside the present paper's scope.
\end{rem}

\appendix
\section{Topological vector spaces}

In this appendix we provide some basic definitions and recall some standard results on topological vector spaces. We first begin by providing a definition for the $C^{\infty}$-topology on the space of smooth sections $\mathcal{E}:=\Gamma(M,E)$ of a $\gf$-vector bundle $E$ over a smooth manifold $M$.

\begin{defi} \label{def_cinftopology}
Let $U$ be any open neighborhood of the zero section in the $i$th jet bundle $J^i(E)$ of $E$ and let $K$ be any compact subset of $M$. Consider the family
\[ \mathscr{G}^i(K,U):=\{\gamma\in\mathcal{E}:\gamma^i(K)\subset U\} \]
of subsets of sections of $E$ that is generated by allowing $K$ and $U$ to vary freely, where $\gamma^i$ denotes the prolongation of $\gamma$ to $J^i(E)$. This family forms a basis of neighborhoods of zero in $\mathcal{E}$ for the \emph{$C^i$-topology} on $\mathcal{E}$.

To obtain the \emph{$C^\infty$-topology}, we consider the basis of neighborhoods that is generated when we also allow $i\geq 0$ to vary. This will be the projective limit of the $C^i$-topologies.
\end{defi}

The $C^i$-topology may be equivalently defined by considering the family of seminorms that arises by placing any metric on the jet bundle $J^i(E)$. By choosing a system of charts for $M$ consisting of (countably many) relatively compact open subsets over which the bundle $E$ is trivial, we can express the $C^\infty$-topology on $\mathcal{E}$ as the projective limit of the $C^\infty$-topologies over each chart. The latter is known to form a nuclear Fr\'echet space, see \cite[Ch. 10, 51]{Treves}. It therefore follows from \cite[Ch. 50]{Treves} that $\mathcal{E}$ is also a nuclear Fr\'echet space.

If $E$ and $F$ are smooth $\gf$-vector bundles over manifolds $M$ and $N$ respectively, we may consider their external tensor product $E\boxtimes F$, which is a vector bundle over $M\times N$ formed by taking the tensor product of the pullback bundles of $E$ and $F$ along the canonical projection maps.

\begin{prop}
There is a canonical isomorphism of topological vector spaces
\begin{displaymath}
\begin{array}{ccc}
\Gamma(M,E)\cotimes\Gamma(N,F) & \cong & \Gamma(M\times N,E\boxtimes F) \\
(\xi,\eta) & \longmapsto & \left[(x,y)\mapsto\xi(x)\otimes\eta(y)\right]
\end{array}
\end{displaymath}
\end{prop}

\begin{proof}
This result is proved in Theorem 51.6 of \cite{Treves} when $E$ and $F$ are trivial line bundles and $M$ and $N$ are open subsets of real affine spaces. The general case above may be reduced to this case using a partition of unity.
\end{proof}

Given two locally convex Hausdorff topological vector spaces $\mathcal{U}$ and $\mathcal{V}$, there are two main topologies that we may consider on the space $\Hom_{\gf}(\mathcal{U},\mathcal{V})$ of continuous linear maps:
\begin{itemize}
\item
pointwise convergence, also known as \emph{weak} convergence; and
\item
uniform convergence on bounded sets, which is also called \emph{strong} convergence.
\end{itemize}

Now suppose that $\mathcal{U}$ is in fact a nuclear Fr\'echet space, in which case it must also be a Montel space; see \cite[Ch. 50]{Treves}. Given a one-parameter family of linear operators
\[ \phi_t:\mathcal{U}\to\mathcal{V}, \quad t\in (0,1) \]
it follows from the Banach-Steinhaus Theorem \cite[Ch. 33]{Treves} that $\phi_t$ converges weakly as $t\to 0$ if and only if it converges strongly.

\begin{lemma} \label{lem_compositionconverge}
Let $\mathcal{U}$ and $\mathcal{V}$ be nuclear Fr\'echet spaces and $\mathcal{W}$ be a locally convex Hausdorff topological vector space and suppose that as $t\to 0$, we have weakly converging one-parameter families of linear operators;
\[ \phi_t:\mathcal{U}\to\mathcal{V},\quad\psi_t:\mathcal{V}\to\mathcal{W}; \qquad t\in (0,1). \]
Then the one-parameter family $\psi_t\circ\phi_t:\mathcal{U}\to\mathcal{W}$ converges strongly as $t\to 0$.
\end{lemma}

\begin{proof}
The result would be trivial if it were not for the fact that composition of operators is not generally continuous (in the strong operator topology) when $\mathcal{V}$ is a nuclear space. By the Banach-Steinhaus Theorem, it is sufficient to show that if $\phi_n$ and $\psi_n$ are sequences of operators converging strongly to zero, then the sequence $\psi_n(\phi_n(u))$ converges to zero at every point $u\in\mathcal{U}$. Since the sequence $\phi_n(u)$ is convergent and therefore bounded, the result now follows trivially.
\end{proof}

Another result that we often make use of throughout the paper is a corollary of Proposition 50.5 of \cite{Treves}.

\begin{prop}
If $\mathcal{U}$ is a nuclear Fr\'echet space and $\mathcal{V}$ is a complete locally convex Hausdorff topological vector space then
\[ \mathcal{U}^{\dag}\cotimes\mathcal{V} = \Hom_{\gf}(\mathcal{U},\mathcal{V}). \]
\end{prop}

Note that when we make use of the above proposition in the text, we must be careful to apply the Koszul sign rule, so that $(f\otimes v)[u]:=(-1)^{|v||u|}f(u)v$.

\end{document}